\documentclass[10pt,a4paper]{amsart}
\usepackage{amsmath,amssymb,amsthm,dsfont,mathtools,wasysym}
\usepackage{color}
\usepackage[active]{srcltx}
\usepackage[top=2.75cm,bottom=2.4cm,left=3.1cm,right=3.1cm]{geometry}
\usepackage[colorlinks=true, linktocpage=true, linkcolor=red!70!black, citecolor=green!50!black, urlcolor=blue!50!black]{hyperref}
\usepackage{mathrsfs}
\usepackage{needspace}
\usepackage{enumitem}
\usepackage{enumitem}
\setenumerate{label={\rm (\alph{*})}}
\usepackage{trfsigns}
\usepackage{upref} 
\usepackage[font=small]{caption}
\usepackage{mathtools}
\usepackage{todonotes}
\usepackage{bbm}
\usepackage{esint}

\usepackage{tikz-cd}
\usetikzlibrary{decorations.pathreplacing}
\makeatletter
\DeclareRobustCommand\widecheck[1]{{\mathpalette\@widecheck{#1}}}
\def\@widecheck#1#2{%
	\setbox\z@\hbox{\m@th$#1#2$}%
	\setbox\tw@\hbox{\m@th$#1%
		\widehat{%
			\vrule\@width\z@\@height\ht\z@
			\vrule\@height\z@\@width\wd\z@}$}%
	\dp\tw@-\ht\z@
	\@tempdima\ht\z@ \advance\@tempdima2\ht\tw@ \divide\@tempdima\thr@@
	\setbox\tw@\hbox{%
		\raise\@tempdima\hbox{\scalebox{1}[-1]{\lower\@tempdima\box
				\tw@}}}%
	{\ooalign{\box\tw@ \cr \box\z@}}}
\makeatother
\newcommand*{\Lcorner}{%
	\mathchoice%
	{\mathrel{\makebox[7pt][c]{\rule{.4pt}{7.5pt}\rule{5pt}{.4pt}}}}%
	{\mathrel{\makebox[7pt][c]{\rule{.4pt}{7.5pt}\rule{5pt}{.4pt}}}}%
	{\mathrel{\makebox[5.5pt][c]{\rule{.4pt}{5.25pt}\rule{3.5pt}{.4pt}}}}%
	{\mathrel{\makebox[4pt][c]{\rule{.4pt}{3.75pt}\rule{2.5pt}{.4pt}}}}%
}

\newcommand{\abs}[1]{\left|#1\right|}
\newcommand\1{\mathds{1}}
\newcommand\B{{\rm B}}
\newcommand{\ball}{\mathrm{B}}
\newcommand{\bv}{\mathrm{BV}}
\newcommand\BD{\mathrm{BD}}
\newcommand{\bd}{\mathrm{BD}}
\newcommand{\LD}{\mathrm {LD}}
\newcommand{\ld}{\mathrm{LD}}
\newcommand{\locc}{\mathrm{loc}}
\newcommand{\lebe}{\mathrm{L}}
\newcommand\C{{\rm C}}

\newcommand\dist{{\rm dist}}

\newcommand{\dif}{\mathrm{d}}
\newcommand{\E}{{\mathrm{E}}}
\newcommand{\Di}{\mathscr{D}}

\newcommand\del\partial
\newcommand\eps\varepsilon
\newcommand\g\gamma
\newcommand\G\Gamma
\newcommand{\hold}{\mathrm{C}}
\newcommand{\Hd}{\mathscr{H}}

\renewcommand\l\lambda
\renewcommand{\L}{\mathrm{L}}
\newcommand{\trace}{\mathrm{tr}}

\newcommand{\Lip}{\operatorname{Lip}}
\newcommand\loc{{\rm loc}}
\newcommand\sym{{\rm sym}}

\newcommand{\dista}{\mathrm{dist}}

\newcommand\N{\mathds{N}}
\newcommand\qq\qquad
\newcommand\R{\mathds{R}}
\newcommand\spt{{\rm spt}}
\newcommand\vp\varphi
\newcommand\W{{\rm W}}
\newcommand{\xint}[3]{{\setbox0=\hbox{$#1{#2#3}{\int}$}
		\vcenter{\hbox{$#2#3$}}\kern-.5\wd0}}
   
\newcommand{\rsym}{\mathds{R}_{\operatorname{sym}}^{n\times n}}

\newcommand{\sobo}{\W}

\newcommand{\sg}{\varepsilon}
\newcommand{\dashint}{\fint}
\newcommand{\diam}{\mathrm{diam}}

\DeclareMathOperator{\Div}{div}

\newcommand{\dx}[1]{ \,\mathrm{d}#1}
\numberwithin{equation}{section}

\newtagform{boldbrackets}{\hspace{8ex}\bf(}{\bf)}
\newtagform{standardbrackets}{)}{)}



\allowdisplaybreaks[3]

\newtheorem{theorem}{Theorem}[section]
\newtheorem{corollary}[theorem]{Corollary}
\newtheorem{lemma}[theorem]{Lemma}
\newtheorem{proposition}[theorem]{Proposition}
\newtheorem{definition}[theorem]{Definition}

\theoremstyle{remark}
\newtheorem{remark}[theorem]{Remark}

\makeatletter
\newcommand{\setword}[2]{%
  \phantomsection
  #1\def\@currentlabel{\unexpanded{#1}}\label{#2}%
}
\makeatother

\begin{document}

\title{Gradient integrability for bounded $\BD$-minimizers}
\author[L. Beck]{Lisa Beck}
\author[F. Eitler]{Ferdinand Eitler}
\address{Lisa Beck \& Ferdinand Eitler: Institut f\"ur Mathematik, Universit\"at Augsburg, Universit\"atsstr.\ 12a, 86159 Augsburg, Germany.}
\email{lisa.beck@math.uni-augsburg.de}\email{ferdinand.eitler@math.uni-augsburg.de}
\author[F. Gmeineder]{Franz Gmeineder}
\address{Franz Gmeineder: Fachbereich Mathematik und Statistik, Universit\"at Konstanz, Universit\"atsstr.\ 10, 78464 Konstanz, Germany.}
\email{franz.gmeineder@uni-konstanz.de}
\subjclass{35B65, 35J60, 35J93, 49J45}
\keywords{Functions of bounded deformation, functionals of linear growth, functionals of $(p,q)$-growth, Sobolev regularity.}

\maketitle

\date{\today}

\begin{abstract}
We establish that locally bounded relaxed minimizers of degenerate elliptic symmetric gradient functionals on $\bd(\Omega)$ have weak  gradients in $\lebe_{\locc}^{1}(\Omega;\R^{n\times n})$.  This is achieved for the sharp ellipticity range that is presently known to yield $\sobo_{\locc}^{1,1}$-regularity in the full gradient case on $\bv(\Omega;\R^{n})$. As a consequence, we also obtain the first Sobolev regularity results for minimizers of the area-type functional on $\bd(\Omega)$.
\end{abstract}
\setcounter{tocdepth}{1}
\tableofcontents

\section{Introduction}

\subsection{Aim and scope}
Let $\Omega\subset\R^n$ with $n\geq 2$ be an open and bounded set with Lipschitz boundary~$\partial\Omega$. In this paper, we deal with the regularity of relaxed minimizers of functionals
\begin{align}\label{eq:functionalmain}
F[u;\Omega]\coloneqq \int_{\Omega}f(\varepsilon(u))\dx{x},\qquad u\colon\Omega\to\R^{n},
\end{align}
where $\varepsilon(u) \coloneqq \frac{1}{2}(\nabla u + \nabla u^{\top})$ denotes the symmetric part of the gradient of the map~$u$ and $f\colon\R_{\sym}^{n\times n}\to\R$ is an integrand to be specified further below. Functionals of the form~\eqref{eq:functionalmain} are key to the study of the elastic or plastic behaviour of solids or fluids. In such theories, these functionals  are used to model related energies in terms of the underlying displacement or velocity fields, respectively, while particular choices of the integrands~$f$ allow to model different aspects of the  material or fluid; see, e.g., \cite{FuchsSeregin,Lubliner}. 

Plasticity effects are usually accounted for by use of \emph{linear growth} functionals~\cite{FuchsSeregin}, and these constitute the starting point for the present paper. Specifically, we suppose that there exist constants $\gamma,\Gamma>0$ with
\begin{align}\label{eq:lingrowth1}
\gamma\abs{z} \leq f(z)\leq \Gamma(1+\abs{z})\quad\mbox{for all } z\in \R^{n\times n}_\sym.
\end{align}
This condition ensures that $F[-;\Omega]$ is finite on the space 
\begin{equation*}
\ld(\Omega)\coloneqq \{u\in\lebe^{1}(\Omega;\R^{n})\colon\; \sg(u)\in\lebe^{1}(\Omega;\rsym)\}, 
\end{equation*}
which is the symmetric gradient-variant of $\sobo^{1,1}(\Omega;\R^{n})$ and endowed with the canonical norm $\|u\|_{\ld(\Omega)} \coloneqq \|u\|_{\lebe^{1}(\Omega)}+\|\sg(u)\|_{\lebe^{1}(\Omega;\R^{n})}$. Moreover, minimizing sequences with prescribed boundary values $u_0 \in \ld(\Omega)$ remain bounded in $\ld(\Omega)$, but by non-reflexivity of $\ld(\Omega)$, they are not necessarily weakly relatively compact in $\ld(\Omega)$. As a consequence, integrands $f$ satisfying~\eqref{eq:lingrowth1} necessitate the relaxation of $F[-;\Omega]$ to the space $\bd(\Omega)$ of functions of bounded deformation. Different from the superlinear growth case $1<p<\infty$, there is no constant $c>0$ such that $\|Du\|_{\lebe^{1}(\Omega;\R^{n\times n})}\leq c\|\sg(u)\|_{\lebe^{1}(\Omega;\rsym)}$ holds for all $u\in\hold_{c}^{\infty}(\Omega;\R^{n})$. This key obstruction, also known as {Ornstein's} \emph{Non-Inequality}, implies that $\sobo^{1,1}(\Omega;\R^{n})\subsetneq\ld(\Omega)$ and  $\bv(\Omega;\R^{n})\subsetneq\bd(\Omega)$ (see, e.g.,  \cite{Ornstein,ContiFaracoMaggi,KirchheimKristensen}). Most notably, the full distributional gradients of $\ld$- or $\bd$-maps do not need to belong to $\lebe^1$ or to the space of $\R^{n\times n}$-valued Radon measures. Hence, results for linear growth functionals involving the full gradient cannot be applied in the current setting. 

This motivates the quest for conditions on $f$ such that relaxed minimizers genuinely belong to $\bv_{\locc}(\Omega;\R^{n})$ or even $\sobo_{\locc}^{1,1}(\Omega;\R^{n})$, finally striving for a parallel regularity theory to what is presently available for linear growth functionals involving full gradients on $\bv(\Omega;\R^{n})$; see \cite{BildhauerLecNotes,SchmidtHabil} for overviews. This especially concerns the critical degenerate elliptic regime. To reach the latter, even in the full gradient case on  $\bv(\Omega;\R^{n})$ or for the closely related functionals of $(p,q)$-growth, it is then customary to impose additional boundedness hypotheses on (generalized) minimizers; see \textsc{Mingione}'s overview~\cite{MinDarkSide}. Such boundedness assumptions are natural from the perspective of applications, for instance the displacement~$u$ in~\eqref{eq:functionalmain} being confined to a bounded spatial region. Still, aiming to parallel the full gradient theory and thereby taking the local boundedness as a standing assumption, none of the previously developed strategies and techniques allows to cover the entire critical ellipticity range for functionals~\eqref{eq:functionalmain} subject to~\eqref{eq:lingrowth1}. 
 
In this paper, we close this gap and thereby complete the picture of Sobolev regularity for linear growth functionals on $\bd$. To state our main result, Theorem~\ref{thm:main} below, we give the precise set-up and its natural context first. 

\subsection{Relaxations and generalised minimizers}
In all of what follows, we suppose that the Dirichlet datum satisfies  $u_{0}\in\ld(\Omega)$. Defining $\ld_{0}(\Omega)$ as the $\|\cdot\|_{\ld(\Omega)}$-closure of $\hold_{c}^{\infty}(\Omega;\R^{n})$, we then consider the variational principle 
\begin{align}\label{eq:varprin1}
\mbox{to minimise}\quad F[u;\Omega]\coloneqq \int_\Omega f(\eps(u))\dx{x}\quad\mbox{over}\quad u\in\Di_{u_0} \coloneqq u_{0}+\ld_{0}(\Omega),
\end{align}
with a convex integrand $f\colon\R_{\sym}^{n\times n}\to\R$ satisfying the linear growth condition~\eqref{eq:lingrowth1}. The lack of weak compactness in $\ld(\Omega)$ implies that solutions of the minimization problem~\eqref{eq:varprin1} do not need to exist in general,  and a relaxation to the space $\bd(\Omega)$ is required. To provide a unifying framework, we put for a subset $U\subseteq\Omega$ with Lipschitz boundary $\partial U$ and maps $u,v\in\bd(U)$
\begin{align}\label{eq:WeakStarRelaxation}
\overline{F}_{v}[u;U] \coloneqq \int_{U} f(\mathscr{E}u)\,\dif x + \int_{U} f^\infty\left(\!\frac{\mathrm{dE}^su}{\mathrm{d}\!\abs{\mathrm{E}^su}}\!\right)\dif{\abs{\mathrm{E}^su}} + \int_{\partial U} f^\infty(\mathrm{tr}_{\partial U}(v-u)\odot\nu_{\del U})\dx{\Hd^{n-1}}.
\end{align}
where the behaviour of the integrand at infinity is captured by the \emph{recession function} 
\begin{equation}\label{eq:recfunc}
 f^\infty(z) \coloneqq \lim_{s \to \infty} \frac{1}{s}f(sz) \qquad \text{for all } z\in \R^{n\times n}_{\sym}. 
\end{equation}
By convexity and the linear growth condition~\eqref{eq:lingrowth1}, $f^\infty$ is well-defined, $1$-homogeneous, finite and convex. We refer the reader to Section~\ref{sec:funspac} for a detailed description of the single constituents in~\eqref{eq:WeakStarRelaxation}. The notion of (local) minimality to be employed in the sequel then is defined as follows:
\begin{definition}[$\bd$- and local $\bd$-minimizers]\label{def:min} 
\leavevmode
\begin{enumerate} 
\item Given $u_{0}\in\bd(\Omega)$, a map $u\in\bd(\Omega)$ is called a \emph{$\bd$-minimizer (of $F$) subject to the Dirichlet datum $u_{0}$} if $\overline{F}_{u_{0}}[u;\Omega]\leq \overline{F}_{u_{0}}[w;\Omega]$ holds for all $w\in\bd(\Omega)$. 
\item A map $u\in\bd_{\locc}(\Omega)$ is called a \emph{local $\bd$-minimizer (of $F$)} if $\overline{F}_{u}[u;U]\leq \overline{F}_{u}[w;U]$ holds for all subsets $U\Subset\Omega$ with Lipschitz boundary $\partial U$ and all $w\in\bd(U)$. 
\end{enumerate} 
\end{definition}

As can be directly inferred from the definition, any $\bd$-minimizer is a local $\bd$-minimizer. If the integrand $f\in\hold(\rsym)$ is convex with~\eqref{eq:lingrowth1}, we have the consistency relation
\begin{align}\label{eq:consistency}
\inf_{u\in\mathscr{D}_{u_{0}}}F[u;\Omega]=\min_{u\in\bd(\Omega)}\overline{F}_{u_{0}}[u;\Omega].
\end{align}
Most importantly, by weak*-compactness principles on $\bd(\Omega)$, $\bd$-minimizers do exist indeed, see Corollary~\ref{cor:reshetnyak} and Remark~\ref{rem:consistency_proof} below.  Still, by Ornstein's Non-Inequality, the full gradients of generic  $\bd_{(\locc)}$-maps need not be Radon measures, and  it is thus a key theme to identify ellipticity conditions on $f$ and lower order hypotheses on $\bd$-minimizers that not only ensure the existence of \emph{full gradients} as Radon measures but also their higher integrability. Here, a suitable scale is that of $\mu$-ellipticity, which we describe in detail next.

\subsection{$\mu$-ellipticity, $(p,q)$-growth and previous results}
In the context of convex $\hold^{2}$-integrands of linear growth, the condition of $\mu$-ellipticity quantifies the degeneration of the second order derivatives of~$f$. Specifically, given $1<\mu<\infty$, a $\C^2$-integrand $f\colon \R^{n\times n}_\sym\to \R$ is called \emph{$\mu$-elliptic} provided that there exist constants $0<\lambda\leq \Lambda <\infty$ with 
\begin{align}\label{eq:MuEllipticity}
\lambda \frac{\abs{\xi}^2}{(1+\abs{z}^2)^{\frac{\mu}{2}}}\leq \langle \nabla^2 f(z)\xi, \xi\rangle \leq \Lambda \frac{\abs{\xi}^2}{(1+\abs{z}^2)^{\frac{1}{2}}}\quad\mbox{for all}\quad z, \xi\in \R^{n\times n}_\sym.
\end{align}
Here, the condition $\mu>1$ on the lower growth exponent is essential, as otherwise the integrands $f$ are not of linear growth; see Remark~\ref{rem:LlogL1} for more detail.  Because of $\mu>1$, it is clear that linear growth integrands with~\eqref{eq:MuEllipticity} feature a different growth behaviour from above and below \emph{on the level of second derivatives}. In this sense, they resemble integrands of $(p,q)$-growth. For the latter, improved gradient regularity estimates for minimizers are well-known to require a dimensional balance between $p$ and $q$, whereas dimension independent thresholds necessitate additional hypotheses. Starting with the foundational work of \textsc{Marcellini} \cite{Marcellini1,Marcellini2} in the scalar case, \textsc{Mingione} et al.  \cite{EspositoLeonettiMingione1999A,EspLeoMin,MinDarkSide} in the vectorial case  and subsequent contributions, typical bounds in this situation read as 
\begin{align}\label{eq:pq}
\frac{q}{p}<1+\frac{2}{n}\;\;\;\text{(unconstrained case)}\;\;\;\text{and}\;\;\;q\leq p+2\;\;\;\text{($\lebe_{\locc}^{\infty}$-constrained case)},
\end{align}
which can also be expressed as bounds $\tfrac{2p}{n}$ and $2$ on the difference $(q-2) - (p-2)$ of the growth exponents for $\nabla^2 f$.
Drawing this analogy for the relaxed functionals~\eqref{eq:WeakStarRelaxation} in view of~\eqref{eq:MuEllipticity}, one thus aims for higher local  gradient integrability for $\mu<1+\frac{2}{n}$ in the unconstrained case, whereas $\mu\leq 3$ comes up as the natural threshold in the $\lebe_{\locc}^{\infty}$-constrained situation; here, $(q-2) - (p-2)$ corresponds to $-1 - (-\mu)$, while we have $p=1$ in the bound $\frac{2p}{n}$. In the full gradient case and so for relaxed minimizers on $\bv$, these bounds have been achieved in \cite{BildhauerMu,BeckSchmidt}; also see~\cite{BildhauerLecNotes}. This threshold is of intrinsic theoretical relevance, as it allows to include the $3$-elliptic \emph{area integrand}
\begin{align*}
f(z)=\sqrt{1+|z|^{2}},\qquad z\in\rsym.
\end{align*}
In the \emph{symmetric} gradient case, however, only the unconstrained case has been tackled successfully so far by the third author~\cite{Gmeineder} for $\mu<1+\frac{2}{n}$. Previous results \cite{Gmeineder2016,Gmeineder1} made use of fractional methods \`{a} la \textsc{Mingione} \cite{MingioneBound,MingioneFract}, but only gave the restricted range $\mu<1+\frac{1}{n}$. Still, the desired ellipticity threshold $\mu\leq 3$ in the $\lebe_{\locc}^{\infty}$-constrained situation has remained open. As will be discussed in the following section, the sharp range $1<\mu\leq 3$ for higher gradient integrability to be addressed here comes with more fundamental obstructions than in previous contributions. Yet, dealing with this borderline case, we thereby obtain a fully parallel Sobolev regularity theory to what is presently known for linear growth functionals on $\bv$. 

\section{Main result and strategy of proof}\label{sec:mainresult}
We now proceed to display our main result and give a discussion of the underlying obstructions afterwards. More precisely, we have the following theorem: 

\begin{theorem}[Main theorem]\label{thm:main}
Let $\Omega\subset\R^{n}$ be open and bounded, and let $f\in\hold^{2}(\rsym)$ be a variational integrand which satisfies~\eqref{eq:lingrowth1} and~\eqref{eq:MuEllipticity} with $1<\mu\leq 3$. Then any local $\bd$-minimizer $u\in\bd_{\locc}(\Omega)\cap\lebe_{\locc}^{\infty}(\Omega;\R^{n})$ of~$F$ \emph{is of class} $\sobo_{\locc}^{1,1}(\Omega;\R^{n})$. More specifically, for every subset $U\Subset\Omega$ there exists a constant $c=c(n,\Omega,\lambda,\Lambda,\gamma,\Gamma, \|u\|_{\L^{\infty}(U;\R^{n})})>0$ such that whenever $\ball_{2r}(x_{0})\Subset U$, we have 
\begin{align}\label{eq:mainestimate}
\int_{\ball_{r}(x_{0})}|\nabla u|\log(1+|\nabla u|^{2})\dif x \leq c\Big(1+\frac{1}{r^{2}}\Big)\Big(r^{n} + r^{n-2} +|\E u|(\overline{\ball}_{2r}(x_{0}))\Big).
\end{align}
\end{theorem} 

This result will be established in Section~\ref{sec:main}. It can also be interpreted as a critical borderline case of the $\lebe_{\locc}^{p}$-constrained case for general $p < \infty$, which will be addressed in the follow-up paper~\cite{BeckEitlerGmeineder} based on the methods developed here. In order to motivate the particular set-up and strategy of the proof, we now pause to explain the key difficulties and novelties in detail. 

\textbf{\setword{(a)}{introa1} Singular measures, lack of maximum principles and non-uniqueness.} To get access to the ellipticity assumption~\eqref{eq:MuEllipticity} imposed on $f$, it is natural to employ a version of the difference quotient method. Since the distributional symmetric gradients of local $\bd$-minimizers $u\in\bd_{\locc}(\Omega)\cap\lebe_{\locc}^{\infty}(\Omega;\R^{n})$ a priori only belong to $\mathrm{RM}(\Omega;\rsym)$, the direct use of difference quotients is however  difficult to be implemented. In this context, a routine device is to consider plain viscosity approximations, so minimizers $v_{j}$ of the stabilized functionals
\begin{align}\label{eq:perturbed1}
F_{j}[v;\Omega]:=F[v;\Omega] + \frac{1}{j}\int_{\Omega}|\sg(v)|^{2}\dif x \qquad \text{for } j\in\N, 
\end{align}
on suitable Dirichlet classes. Different from the full gradient case, see \cite{BildhauerLecNotes,BeckSchmidt}, we cannot utilise tools such as maximum principles or truncation arguments to infer that the sequence $(v_{j})_{j \in \N}$ remains {locally uniformly bounded in~$\Omega$}. As explained in \ref{introd1} below, this however turns out to be crucial in the proof of Theorem~\ref{thm:main}, and reveals a conceptual difference to previous contributions \cite{Gmeineder1,Gmeineder} in the unconstrained case, where an approach based on~\eqref{eq:perturbed1} would yield estimates for one (local) $\bd$-minimizer. Moreover, even \emph{if} the sequence $(v_{j})_{j \in \N}$ remained locally uniformly bounded in $\lebe^{\infty}$ \emph{and} we could establish locally uniform higher integrability estimates on $(\nabla v_{j})_{j\in\N}$, this would only lead to a $\sobo_{\locc}^{1,1}$-regularity result for \emph{at most one} local $\bd$-minimizer. Yet, this does not rule out the existence of a local $\bd$-minimizer $u\in (\bd_{\locc}(\Omega)\cap\lebe_{\locc}^{\infty}(\Omega;\R^{n}))\setminus \sobo_{\locc}^{1,1}(\Omega;\R^{n})$: Since the positively $1$-homogeneous (and hereafter \emph{not} strictly convex) recession function in~\eqref{eq:WeakStarRelaxation} operates on singular measures for $\mathscr{L}^{n}$, the relaxed functional is \emph{not} strictly convex on $\bd_{(\locc)}(\Omega)$, and in general  (local) $\bd$-minimizers might be highly non-unique. 

\textbf{\setword{(b)}{introb1} Ekeland-type viscosity approximations and keeping boundedness constraints}. 
To circumvent this issue, we implement an approximation strategy based on the {Ekeland} variational principle~\cite{Ekeland}, see Lemma~\ref{lem:EkelandLSC} below. In the context of linear growth problems, this strategy appeared first in~\cite{BeckSchmidt}, and has subsequently been employed in \cite{Gmeineder1,Gmeineder} for problems depending on the symmetric gradient. Essentially, working from a fixed locally bounded local $\bd$-minimizer, the Ekeland variational principle in combination with a suitable stabilization yields a sequence whose single members satisfy a suitable  version of almost-minimality. The latter is quantified by a perturbation term and leads to an Euler-Lagrange \emph{inequality} rather than an \emph{equation}. To get access to degenerate second order bounds as the key ingredient of the proof, such perturbations have to be weak enough to be controllable by the available a-priori bounds.  Simultaneously, they have to be strong enough to give useful estimates. As will be discussed in Section~\ref{section:viscosity_approximation}, the natural Ekeland perturbation space then is the negative Sobolev space $\sobo^{-2,1}(\Omega;\R^{n})$.


Here, and as argued above in~\ref{introa1}, it is even more important than in previous  contributions to keep track of the local $\lebe^{\infty}$-bounds. This necessitates the incorporation of an additional $\lebe^{\infty}$-penalization term into the functional, to be dealt with by Ekeland's variational principle. The precise control of $\lebe^{\infty}$-norms of the Ekeland sequence is achieved by an adaptation of an idea appearing in \textsc{Schmidt}'s habilitation thesis~\cite{SchmidtHabil}, in turn being inspired by the penalization strategy from \textsc{Carozza} et al.~\cite{CKP} in the $(p,q)$-context. This leads to approximations
\begin{align*}
F_{j}[w;\Omega] = \underbrace{\int_{\Omega}f(\sg(w))\,\dif x}_{\text{original functional}} +\,  \underbrace{\delta_{j}\int_{\Omega}(1+|\sg(w)|^{2})^{\frac{n+1}{2}}\,\dif x}_{\text{viscosity stabilisation}} + \underbrace{\int_{\Omega}g_{M}(w)\,\dif x.}_{\text{$\lebe^{\infty}$-penalisation}}
\end{align*}
 Enforcing $\lebe^{\infty}$-bounds in this way, a chief issue is that it is \emph{a priori not clear} whether the Ekeland sequence remains away from the corresponding $\lebe^{\infty}$-threshold numbers. In fact, if a member of the Ekeland sequence attained the corresponding threshold, it would be impossible to extract \emph{any} information from the Euler-Lagrange inequality, cf. Remark~\ref{rem:W1n1regularisations}. 

By an overall set-up of the proof slightly different from previous contributions, we establish in Proposition~\ref{lem:ELperutrbed}  \emph{quantitative estimates} on the distance of the Ekeland sequence to the critical $\lebe^{\infty}$-threshold. This strategy, which is directly tailored to our purposes, hinges on a fine analysis of the blow-up of certain approximations in conjunction with a geometric argument.  In consequence, from then on, we will have access to the requisite Euler-Lagrange inequality. 

\textbf{\setword{(c)}{introc1} Ornstein, algebraic manipulations, $\lebe^{\infty}$-constraints and their interplay.} As discussed in~\ref{introb1}, Theorem~\ref{thm:main} is a consequence of degenerate weighted second order estimates. In stark contrast to the full gradient case \cite{BildhauerMu,BeckSchmidt} and in light of Ornstein's Non-Inequality, the derivation of such estimates must avoid the appearance of full gradients at all costs. Whereas robust fractional methods as in \textsc{Mingione}~\cite{MingioneFract} can be employed here too, they do not allow us to reach the desired ellipticity exponent $\mu=3$ from Theorem~\ref{thm:main}. This, in turn,  necessitates delicate algebraic manipulations that let us systematically re-introduce the symmetric gradients in the underlying estimates. 

It is here where the $\lebe^{\infty}$-penalization strategy from~\ref{introb1} has a substantially aggravating impact: Namely,  \emph{re-introducing} the symmetric gradients comes to the effect of \emph{introducing} certain pollution terms (of skew-symmetric structure) that are a priori difficult to be dealt with. In view of Ornstein's Non-Inequality, the only possibility here is to employ the Euler-Lagrange inequality or its differentiated version. As argued in~\ref{introb1}, the enforcing of uniform $\lebe^{\infty}$-bounds \emph{necessarily} requires $\lebe^{\infty}$-penalization terms, and the latter consequently appear in the Euler-Lagrange inequality too. When we aim to bring the appearing pollution terms into a form that is treatable by the Euler-Lagrange inequality, the $\lebe^{\infty}$-penalization terms make it impossible to use the available a priori-bounds. Specifically, the Ekeland approximation set-up and degenerate second order estimates here are intertwined in a form that has not arisen in former contributions. The resolution of this matter, which  requires a perhaps somewhat unusual set-up of the proof and comes with a slightly unnatural yet sufficient estimate as a chief outcome, is given in Section~\ref{subsection:second_order_uniform}, cf. Theorem~\ref{thm:uniformSecondOrder} and Remark~\ref{rem:strucproof}. 

\textbf{\setword{(d)}{introd1} Uniform higher integrability: Admissibility, Korn and  logarithmic losses.} Once the degenerate second order estimates from~\ref{introc1} are established, we are in a position to employ specific test functions along the lines of \textsc{Bildhauer}~\cite{BildhauerMu}. Similar to~\cite{BeckSchmidt}, these specific test function are not a priori admissible in the Euler-Lagrange inequality. Whereas strategies as e.g.~in~\cite{BeckSchmidt} do not apply to the situation here, this issue is circumvented by a novel approach which also simplifies the proof in the full gradient case; see Remark~\ref{rem:simple} for more detail. 

At this stage, all ingredients are available to establish Theorem~\ref{thm:main} in Section~\ref{section:proof_main_theorem}. Specifically, we arrive at uniform $\lebe\log^{2}\lebe_{\locc}$-estimates for the \emph{symmetric gradients} of the Ekeland approximation sequence and thus, as a direct by-product of the Reshetnyak lower semicontinuity theorem (cf. Lemma~\ref{lem:reshetnyak}), at the absolute continuity $\E u \ll \mathscr{L}^{n}$. The logarithmic loss when passing to gradients is a consequence of the singular integral representation for $u\in\hold_{c}^{\infty}(\R^{n};\R^{n})$ 
\begin{align*}
\nabla u^{(k)} = \frac{2}{n\omega_{n}}\sum_{i\leq j}\Big(\mathrm{p.v.}\partial_{i}\nabla \mathfrak{K}_{ij}*\varepsilon^{(jk)}(u) - \mathrm{p.v.}\partial_{k}\nabla \mathfrak{K}_{ij}*\varepsilon^{(ij)}(u) + \mathrm{p.v.}\partial_{j}\nabla\mathfrak{K}_{ij}*\varepsilon^{(ki)}(u)\Big)
\end{align*}
with kernels $\mathfrak{K}_{ij}(x)=x_{i}x_{j}/|x|^{n}$ and the Cauchy principal value $\mathrm{p.v.}$, see e.g.~\cite{Reshetnyak1}. It is well-known that such operators come with the loss of one logarithmic power in the scale of Orlicz spaces. Following  \textsc{Cianchi}~\cite{Cianchi}, this loss cannot be prevented, and so any  improvement of Theorem~\ref{thm:main} would require a conceptually different proof. This however would need to be compatible with the rather finely adjusted framework from Section~\ref{section:viscosity_approximation}, see Remark~\ref{rem:LlogL}. \\

At present, even for the Dirichlet problem on $\bv$, $\mu\leq 3$ is currently the best known and hence critical ellipticity threshold for which higher gradient integrability can be achieved, cf. \cite{BildhauerMu,BeckSchmidt,SchmidtHabil}. Inspired by~\cite{BBMS}, it is only in different scenarios such as Neumann problems on $\bv$ where~$\mu$ is known to be allowed to be increased~\cite{BeckBulicekGmeineder}. This, in turn, happens at the cost of no quantitative integrability gain and the inapplicability of the method to the Dirichlet problem. 

Lastly, we wish to stress that improved results can be obtained when passing from the critical threshold $\mu=3$ to $\mu<3$, see Section~\ref{sec:improvedmu}. In particular, this includes problems with logarithmic hardening, see \textsc{Fuchs \& Seregin} \cite{FuchsSereginLog,FuchsSeregin}. Such problems for full gradient functionals have recently re-attracted attention, see \textsc{De Filippis} et al.  \cite{DeFilippisMingione,DeFilippisDeFilippisPiccinini1,DeFilippisPiccinini2}, and are partially located at the borderline exponent $\mu=1$ of the ellipticity scale considered here. In the setting of the present paper, one then has the a priori existence of the full weak gradients, simplifying most the arguments for the case $\mu=3$ in a considerable way, see Remark~\ref{rem:LlogL1}. 
\subsection*{Structure of the paper}
Apart from the first two sections, the paper is organised as follows: In Section~\ref{sec:prelims} we fix notation, collect  definitions and important background results on function spaces, lower semicontinuity and the Ekeland variational principle. In Section~\ref{sec:main} we then embark on the proof of Theorem~\ref{thm:main}. Here we will make use of two particular extension and approximation results, which are  discussed for the reader's convenience in the appendix, Section~\ref{sec:appendix}.

\section{Preliminaries}\label{sec:prelims}

\subsection{Notation}
We briefly comment on some of the notation used in this paper. We denote by $\ball_{r}(x_{0}) \coloneqq \{z\in\R^{n}\colon\;|z-x_{0}|<r\}$ the open ball of radius $r>0$ centered at a point $x_{0} \in \R^n$. We write with slight abuse of notation $\mathbb{B}_{r}(\xi)\coloneqq \{\eta\in\rsym\colon\;|\eta-\xi|<r\}$ for $\xi\in\rsym$, where we use $|\cdot|$ for the usual Hilbert--Schmidt norm on $\R^{n\times n}$, which stems from the inner product $
\langle \xi,\eta\rangle \coloneqq \trace(\xi^{\top}\eta)$ for $\xi, \eta \in \R^{n\times n}$. It is useful to note that this inner product leads to the orthogonal sum decomposition 
\begin{align}\label{eq:orthsum}
\R^{n\times n} = \R_{\mathrm{sym}}^{n\times n} \oplus_{\bot} \R_{\mathrm{skew}}^{n\times n}
\end{align}
into symmetric and skew-symmetric matrices. Moreover, we employ the notation $a\otimes b\coloneqq a b^{\top}$ for the tensor product of two vectors $a,b\in\R^{n}$ and $a\odot b \coloneqq  \frac{1}{2}(a\otimes b+b\otimes a)$ for the pure (symmetric) tensor. 

We denote by $\mathscr{L}^{n}$ and $\mathscr{H}^{n-1}$ the $n$-dimensional Lebesgue and $(n-1)$-dimensional Hausdorff measures, respectively, and we abbreviate $\omega_n \coloneqq \mathscr{L}^{n}(\ball_1(0))$. Whenever $S \subset \R^n$ is a measurable set with $0<\mathscr{L}^n(S)<\infty$ and $w$ is an integrable function on~$S$, we write $(w)_S \coloneqq \dashint_S w \dx{x} \coloneqq (\mathscr{L}^n(S))^{-1} \int_S w \dx{x}$ for the mean value of~$w$ over~$S$. Moreover, for an open subset $\Omega$ of~$\R^{n}$ and a finite-dimensional inner product space~$V$, the space of (finite) $V$-valued Radon measures on $\Omega$ is denoted by $\mathrm{RM}_{(\mathrm{fin})}(\Omega;V)$; in particular, $\mathrm{RM}_{\mathrm{fin}}(\Omega;V)\cong\hold_{0}(\Omega;V)$. For $\mu\in\mathrm{RM}_{\mathrm{fin}}(\Omega;V)$, the total variation measure of~$\mu$ is denoted by $|\mu|$, and the Lebesgue--Radon--Nikod\'{y}m decomposition into its absolutely continuous and singular parts with respect to~$\mathscr{L}^{n}$ is given by 
\begin{align}\label{eq:LRNdecomp}
\mu = \mu^{a}+\mu^{s} = \frac{\dif\mu}{\dif\mathscr{L}^{n}}\mathscr{L}^{n} + \frac{\dif\mu}{\dif|\mu^{s}|}|\mu^{s}|.
\end{align}
Finally, $c>0$ denotes a generic constant whose value may change from line to line and whose dependencies are usually indicated, while its precise value is only specified if it is of interest. 

\subsection{Function spaces}\label{sec:funspac}
We first collect some definitions and results on the function spaces which play a pivotal role in our paper. Throughout, let $\Omega\subset\R^{n}$ be an open and bounded set.

\subsubsection{Functions of bounded deformation}\label{sec:bd}
A function $u\in\lebe^{1}(\Omega;\R^{n})$ is said to be \emph{of bounded deformation} on $\Omega$ provided its distributional symmetric gradient can be represented by a finite $\R_{\sym}^{n\times n}$-valued Radon measure $\E u\in\mathrm{RM}_{\mathrm{fin}}(\Omega;\rsym)$. This can equivalently be expressed by requiring that the \emph{total deformation}
\begin{align*}
|\E u|(\Omega)\coloneqq\sup\Big\{\int_{\Omega}u\cdot\mathrm{div}(\varphi)\,\dif x \colon \varphi\in\hold_{c}^{\infty}(\Omega;\R_{\sym}^{n\times n}),\; \|\varphi\|_{\L^\infty(\Omega;\R^{n\times n})}\leq 1 \Big\} 
\end{align*}
is finite. The linear space of all functions in $\lebe^{1}(\Omega;\R^{n})$ of bounded deformation is denoted by $\bd(\Omega)$, with the local variant $\bd_{\locc}(\Omega)$  defined in the obvious manner. We refer the reader to \cite{ACD,Kohn1,TemamStrang,GmeinederRaita2} for a detailed exposition on these spaces and proceed to give the relevant results required in the sequel.

We notice that the null space of the symmetric gradient~$\sg$ for functions on $\R^{n}$ is given by the space of \emph{rigid deformations} defined by
\begin{align*}
\mathscr{R}(\R^{n})\coloneqq \{x\mapsto Ax+b\colon A\in\R_{\mathrm{skew}}^{n\times n},\;b\in\R^{n}\}. 
\end{align*}
Moreover, for a bounded set $S \subset \R^n$, which  refers to a ball or a cube later on, we may introduce an $\lebe^{1}$-bounded projection $\mathbb{P}\colon\lebe^{1}(S;\R^{n})\to\mathscr{R}(\R^{n})$. To this end, we first choose an $\lebe^{2}(S;\R^{n})$-orthonormal basis $\{\pi_{1},\ldots,\pi_{M(n)}\}$ of $\mathscr{R}(\R^{n})$ and define 
 \begin{align*}
 \mathbb{P}u\coloneqq \sum_{j=1}^{M(n)}\bigg(\dashint_{S}\pi_{j}u\,\dif x\bigg)\,\pi_{j} \qquad \text{for } u\in\lebe^{1}(S;\R^{n}). 
 \end{align*}
 Since each of the functions $\pi_{1},\ldots,\pi_{M(n)}$ belongs to $\lebe^{\infty}(\ball_{1}(0);\R^{n})$, the operator $\mathbb{P}$ is indeed well-defined and $\lebe^{1}$-bounded.

For $u\in\bd(\Omega)$, the Lebesgue--Radon--Nikod\'{y}m decomposition~\eqref{eq:LRNdecomp} of $\E u$ into its absolutely continuous and singular parts with respect to the Lebesgue measure $\mathscr{L}^{n}$ reads as 
\begin{align}\label{eq:BDLNR}
\E u = \E^{a} u + \E^{s} u = \frac{\dif\E u}{\dif\mathscr{L}^{n}}\mathscr{L}^{n} + \frac{\dif\E u}{\dif|\E^{s} u|}|\E^{s} u| = \mathscr{E}u\mathscr{L}^{n} + \frac{\dif\E u}{\dif|\E^{s} u|}|\E^{s} u|.
\end{align}
Here, the density of~$\E^{a} u$ with respect to $\mathscr{L}^{n}$ can be identified with the symmetric part $\mathscr{E}u$ of the approximate gradient~$\nabla u$ of~$u$. 

On $\bd(\Omega)$, there are three different notions of convergence beyond the usual norm convergence which are of central importance: Given $u,u_{1},u_{2},\ldots \in\bd(\Omega)$, we say that the sequence $(u_{j})_{j \in \N}$ converges to $u$ in the 
\begin{itemize} 
\item \emph{\textup{(}symmetric\textup{)} weak*-sense} on $\bd(\Omega)$ and write $u_{j}\stackrel{*}{\rightharpoonup}u$ in $\bd(\Omega)$ if $u_{j}\to u$ strongly in $\lebe^{1}(\Omega;\R^{n})$ and $\E u_{j}\stackrel{*}{\rightharpoonup}\E u$ in the weak*-sense of $\R_{\sym}^{n\times n}$-valued Radon measures on $\Omega$,
\item \emph{\textup{(}symmetric\textup{)} strict sense} if $u_{j}\to u$ strongly in $\lebe^{1}(\Omega;\R^{n})$ and $|\E u_{j}|(\Omega)\to|\E u|(\Omega)$,
\item \emph{\textup{(}symmetric\textup{)} area-strict sense} and write $u_{j}\stackrel{\langle\cdot\rangle}{\to}u$ in $\bd(\Omega)$ if $u_{j}\to u$ strongly in $\lebe^{1}(\Omega;\R^{n})$ and $\langle \E u_{j}\rangle (\Omega) \to \langle \E u \rangle (\Omega)$, where we have abbreviated
\begin{align}\label{eq:areadef}
 \langle \E u\rangle(\Omega) \coloneqq \int_{\Omega}\sqrt{1+|\mathscr{E}u|^{2}}\,\dif x + |\E^{s}u|(\Omega) \quad \text{for } u \in \bd(\Omega). 
\end{align}
\end{itemize} 

We note that if $\Omega\subset\R^{n}$ has in addition a Lipschitz boundary, then there exists a bounded linear (boundary) trace operator $\trace_{\partial\Omega}\colon\bd(\Omega)\to\lebe^{1}(\partial\Omega;\R^{n})$, see~\cite{TemamStrang}. Adopting the notation $\langle\E u\rangle(\Omega)$ from~\eqref{eq:areadef}, this allows us to state the following approximation result:

\begin{lemma}[Symmetric area-strict smooth approximation]\label{lem:SAP}
Let $\Omega\subset\R^{n}$ be open and bounded  with Lipschitz boundary oriented by $\nu_{\partial\Omega}\colon\partial\Omega\to\mathbb{S}^{n-1}$, and let $u_{0}\in\ld(\Omega)$. Then for every map $u\in\bd(\Omega)$ there exists a sequence $(u_{j})$ in $u_{0}+\hold_{c}^{\infty}(\Omega;\R^{n})$ such that $\|u-u_{j}\|_{\lebe^{1}(\Omega;\R^{n})}\to 0$ and 
\begin{align}\label{eq:areastrictapprox1}
\langle\E u_{j}\rangle(\Omega) \to \langle\E u\rangle(\Omega) + \int_{\partial\Omega}|\mathrm{tr}_{\partial\Omega}(u-u_{0})\odot\nu_{\partial\Omega}|\dif\mathscr{H}^{n-1} \qquad \text{as } j\to\infty. 
\end{align}
Moreover, if $u \in \bd(\Omega)\cap\lebe^{\infty}(\Omega;\R^{n})$, then  there exists a constant $C_A(\Omega,n)>0$ such that 
\begin{align}\label{eq:Linftyapprox}
\|u_{j}\|_{\lebe^{\infty}(\Omega;\R^{n})}\leq C_A\|u\|_{\lebe^{\infty}(\Omega;\R^{n})}\qquad\text{for all }j\in\N. 
\end{align} 
\end{lemma}

Lemma~\ref{lem:SAP} is well-known to hold in the $\bv$-case, cf. \cite[Lem. B.2]{BildhauerLecNotes}, and its proof adapts in a straightforward manner to the $\bd$-case considered here to yield~\eqref{eq:areastrictapprox1}; see also \cite[Prop. 4.24]{BDG} or \cite[Lem. 2.1]{Gmeineder}. This particular construction directly yields the additional $\lebe^{\infty}$-estimate~\eqref{eq:Linftyapprox}, and we leave the details to the reader.

For our future purposes, we next record a lemma on an extension operator for $\ld$- and $\bd$-functions which preserves $\lebe^{\infty}$-bounds:

\begin{lemma}\label{lem:ext}
Let $\Omega, \Omega_0 \subset\R^{n}$ be two open and bounded sets with $\Omega\Subset\Omega_0$ such that~$\Omega$ has a Lipschitz boundary. There exists a linear and \textup{(}norm-\textup{)}bounded extension operator $\mathfrak{J}\colon \bd(\Omega)\to\bd(\R^{n})$ with the following properties: 
\begin{enumerate} 
\item $\spt(\mathfrak{J}u)\subset\Omega_0$ for all $u\in\bd(\Omega)$, 
\item $\mathfrak{J}\colon \ld(\Omega)\to\ld_{0}(\Omega_0)$, 
\item\label{item:ext2} There exists a constant $C(\Omega,n)>0$ such that $\|\mathfrak{J}u\|_{\lebe^{\infty}(\R^{n};\R^{n})}\leq C\|u\|_{\lebe^{\infty}(\Omega;\R^{n})}$ holds for all $u\in\bd(\Omega)\cap\lebe^{\infty}(\Omega;\R^{n})$. 
\end{enumerate} 
\end{lemma} 

By Ornstein's Non-Inequality, such an operator cannot be constructed by reflection and localisation as in the $\sobo^{1,1}$- or $\bv$-case. Instead, $\mathfrak{J}$ can be obtained as a Jones-type operator~\cite{Jones} as given by Raita and the third author in~\cite{GmeinederRaita1}. To keep our presentation self-contained, a quick discussion thereof is provided in the Appendix~\ref{subsec:construction_extension}. 

\subsubsection{Orlicz--Sobolev spaces}\label{sec:Orlicz}
In order to accomplish the passage from symmetric to full gradients in the proof of Theorem~\ref{thm:main}, we require a refined scaled version of Korn's inequality. To this end, let $A \colon [0,\infty)\to [0,\infty)$ be a Young function, i.e., $A$ has a representation 
\begin{equation*}
A(t) = \int_0^t a(\tau)\dx{\tau} \quad \text{for } t\geq 0, 
\end{equation*}
with a non-decreasing, left-continuous function $a \colon [0,\infty) \to [0,\infty]$ which is neither identical~$0$ nor~$\infty$. We define the Lebesgue--Orlicz space $\L^A(\Omega;\R^{n})$ as the linear space consisting of all measurable maps $u \colon \Omega\to\R^{n}$ such that the Luxemburg norm
\begin{align*}
	\|u\|_{\L^A(\Omega;\R^n)} \coloneqq \inf\left\{\lambda>0 \colon \int_\Omega A\left(\frac{\abs{u}}{\lambda}\right)\dx{x}\leq 1 \right\}
\end{align*}
is finite. We then define the corresponding Orlicz--Sobolev space $\W^{1,A}(\Omega;\R^n)$ involving full gradients and the function space $\mathrm{E}^{A}(\Omega)$ involving symmetric gradient by 
\begin{align*}
\W^{1,A}(\Omega;\R^n) & \coloneqq \{u\in \L^A(\Omega;\R^n): \nabla u\in \L^A(\Omega;\R^{n\times n})\}, \\
\mathrm{E}^{A}(\Omega) & \coloneqq \{u\in \L^A(\Omega;\R^{n}) \colon \eps(u)\in \L^A(\Omega;\R^{n\times n}_\sym)\}.
\end{align*}
As usual, the full gradients and the symmetric gradients, respectively, are understood in the sense of distributions. For $\alpha>0$, we further set 
\begin{equation*}
A_\alpha(t) \coloneqq t\log^\alpha(1+t^{2}) \qquad \text{for } t \geq 0
\end{equation*}
and write also $\W^{1,\L\log^\alpha\L}(\Omega;\R^{n})$ and $\mathrm{E}^{\L\log^\alpha\L}(\Omega)$ instead of $\W^{1,A_\alpha}(\Omega;\R^{n})$ and $\mathrm{E}^{A_\alpha}(\Omega)$. 

To conveniently derive the estimate of Theorem~\ref{thm:main}, we now give the requisite scaled version of Korn's inequality:  

\begin{lemma}[Korn-type inequality with scaling]\label{lem:Korn}
Let $x_{0}\in\R^{n}$ and $r>0$. For every $\alpha\geq 1$ there exists a constant $c=c(n,\alpha)>0$ such that for every $u\in \mathrm{E}^{1,A_{\alpha}}(\ball_{r}(x_{0}))$ there holds $u \in \W^{1,A_{\alpha-1}}(\ball_{r}(x_0);\R^{n})$ with 
\begin{align}\label{eq:kornllogl}
\dashint_{\ball_{r}(x_{0})}A_{\alpha-1}(|\nabla u|)\,\dif x \leq c\bigg(1+\dashint_{\ball_{r}(x_{0})}A_{\alpha-1}\left(\frac{|u|}{r}\right)\dif x + \dashint_{\ball_{r}(x_{0})}A_{\alpha}(|\sg(u)|)\dif x\bigg).
\end{align}
\end{lemma}

\begin{proof} 
The proof is a combination of the results of \textsc{Cianchi}~\cite{Cianchi}. By scaling, it is no loss of generality to assume $x_{0}=0$ and $r=1$. Due to \cite[Thm. 3.13, Ex. 3.15]{Cianchi}, we have for all $u\in\E^{1,A_{\alpha}}(\ball_{1}(0))$ with a constant $c=c(n,\alpha)>0$
\begin{align}\label{eq:CianchiKorn}
\inf_{\mathbf{r}\in\mathscr{R}(\R^{n})}\dashint_{\ball_{1}(0)}A_{\alpha-1}(|\nabla(u-\mathbf{r})|)\dif x \leq c\bigg(1 + \dashint_{\ball_{1}(0)}A_{\alpha}(|\sg(u)|)\dif x \bigg). 
\end{align}
In the following, let $\mathbb{P}$ denote the projection on $\ball_{1}(0)$ from Section~\ref{sec:bd} to the space~$\mathscr{R}(\R^{n})$ of rigid deformations. By the equivalence of all norms on finite dimensional spaces and since $\dim(\mathscr{R}(\R^{n}))<\infty$, there exists a constant $c(n)>0$ such that 
\begin{align}\label{eq:littlehelper}
\sup_{\ball_{1}(0)}|\nabla\mathbb{P}v|\leq c \dashint_{\ball_{1}(0)}|v|\,\dif x
\end{align}
holds for all $v\in\lebe^{1}(\ball_{1}(0);\R^{n})$. Using that $A_{\alpha-1}$ is convex and doubling in the first two steps, we have for any $\mathbf{r}\in\mathscr{R}(\R^{n})$ with $(u-\mathbf{r})_{\ball_{1}(0)}=0$ 
\begin{align*}
\lefteqn{\dashint_{\ball_{1}(0)}A_{\alpha-1}(|\nabla u|)\,\dif x }\\
& \stackrel{\phantom{\eqref{eq:littlehelper}}}{\leq} c\,\dashint_{\ball_{1}(0)}A_{\alpha-1}(|\nabla(u-\mathbf{r})|)\,\dif x + c\,\dashint_{\ball_{1}(0)}A_{\alpha-1}(|\nabla(\mathbb{P}(u-\mathbf{r})|)\,\dif x + c\,\dashint_{\ball_{1}(0)}A_{\alpha-1}(|\nabla\mathbb{P}u|)\,\dif x \\    
& \stackrel{\eqref{eq:littlehelper}}{\leq} c\,\dashint_{\ball_{1}(0)}A_{\alpha-1}(|\nabla(u-\mathbf{r})|)\,\dif x + c A_{\alpha-1}\bigg(\dashint_{\ball_{1}(0)}|u-\mathbf{r}|\dif x \bigg) + c\,A_{\alpha-1}\bigg(\dashint_{\ball_{1}(0)}|u|\,\dif x \bigg) \\
& \stackrel{\phantom{\eqref{eq:littlehelper}}}{\leq} c\,\dashint_{\ball_{1}(0)}A_{\alpha-1}(|\nabla(u-\mathbf{r})|)\,\dif x + c\,\bigg(\dashint_{\ball_{1}(0)}A_{\alpha-1}(|u|)\,\dif x \bigg), 
\end{align*}
where the last inequality is valid due to the classical Poincar\'{e} inequality and the convexity of~$A_{\alpha-1}$. Since it does not matter whether the infimum in~\eqref{eq:CianchiKorn} is taken over all $\mathbf{r}\in\mathscr{R}(\R^{n})$ or over all $\mathbf{r}\in\mathscr{R}(\R^{n})$ with $(u-\mathbf{r})_{\ball_{1}(0)}=0$, we then may pass to the infimum over all $\mathbf{r}\in\mathscr{R}(\R^{n})$ with $(u-\mathbf{r})_{\ball_{1}(0)}=0$ in the previous inequality to conclude~\eqref{eq:kornllogl} for $x_{0}=0$ and $r=1$. Since every function in $\lebe^{A_{\alpha-1}}(\ball_{1}(0);\R^{n})$ automatically belongs to $\lebe^{A_{\alpha}}(\ball_{1}(0);\R^{n})$, this completes the proof. 
\end{proof} 
\subsubsection{Negative Sobolev spaces}
For $k \in \N$ we define $\W^{-k,1}(\Omega;\R^n)$ as the collection of all distributions~$T$ on~$\Omega$ of the form 
\begin{equation}\label{eq:Trex}
  T = \sum_{|\alpha| \leq k} \partial^\alpha w_\alpha
\end{equation}
with functions $w_\alpha \in \L^1(\Omega;\R^n)$ for all multi-indices $\alpha \in \N^n_0$ of length $|\alpha| \leq k$. This linear space is canonically endowed with the norm
\begin{equation*}
 \|T\|_{\W^{-k,1}(\Omega;\R^n)} \coloneqq \inf \sum_{|\alpha| \leq k} \| w_\alpha \|_{\L^1(\Omega;\R^n)},
\end{equation*}
where the infimum is taken over all functions $(w_\alpha)_{|\alpha| \leq k}$ such that the representation~\eqref{eq:Trex} holds. Endowed with this norm, $\W^{-k,1}(\Omega;\R^n)$ is a Banach space. In this paper, we only consider the cases $k\in \{1,2\}$, for which we notice the trivial embedding $\W^{-1,1}(\Omega;\R^n) \subset \W^{-2,1}(\Omega;\R^n)$ with 
\begin{equation}
\label{eq:neg_Sob_0}
\|T\|_{\W^{-2,1}(\Omega;\R^n)} \leq  \|T\|_{\W^{-1,1}(\Omega;\R^n)} \quad \text{for all } T \in \W^{-1,1}(\Omega;\R^n). 
\end{equation}
Let us further observe that for a function $w \in \L^1(\Omega;\R^n)$ we can estimate the $\W^{-k,1}$-norm of its derivative and finite difference quotient, and for $s \in \{1,\ldots,n\}$ and $h > 0$ there holds
\begin{align}
\label{eq:neg_Sob_1}
 & \| \del_s w\|_{\W^{-2,1}(\Omega;\R^n)} \leq \| w\|_{\W^{-1,1}(\Omega;\R^n)} \leq \| w\|_{\L^{1}(\Omega;\R^n)} , \\
\label{eq:neg_Sob_2}
 & \| \del_s w\|_{\W^{-1,1}(\Omega;\R^n)} \leq  \| w\|_{\L^{1}(\Omega;\R^n)} , \\
 \label{eq:neg_Sob_3} 
 & \| \Delta_{s,h} w\|_{\W^{-1,1}(\{ x \in \Omega \colon \dist (x,\partial \Omega) > h\};\R^n)} \leq  \| w\|_{\L^{1}(\Omega;\R^n)}.
\end{align}
The inequalities in~\eqref{eq:neg_Sob_1} and~\eqref{eq:neg_Sob_2} directly follow from the definition of the norm in the spaces $\W^{-2,1}(\Omega;\R^n)$ and $\W^{-1,1}(\Omega;\R^n)$, respectively, while inequality~\eqref{eq:neg_Sob_3} is obtained for a function $w \in \hold^1(\Omega;\R^n)$ from the representation 
\begin{equation*}
\Delta_{s,h} w(x) = \del_s \int_0^1 w(x+the_s) \dx{t}, 
\end{equation*}
and then follows for a general function $w \in \L^1(\Omega;\R^n)$ by approximation.  

\subsubsection{Weighted Lebesgue spaces}
Given a Radon measure $\mu$ on $\Omega$ and a finite dimensional inner product space $V$, we denote for $1\leq p \leq \infty$ the $V$-valued $\L^{p}$-space with respect to $\mu$ by $\L_{\mu}^{p}(\Omega;V)$, equipped with the canonical norm. The following lemma on the identification of pointwise and weak limits might be well-known, but we have not found a precise reference and thus state it here for the reader's convenience.

\begin{lemma}\label{lem:identification}
Let $m\in\N$ and consider $\mu=\theta\mathscr{L}^{n}\Lcorner\Omega$ for some $\theta\in\L^{1}(\Omega;\R_{\geq 1})$. Suppose that $(u_{j})_{j \in \N}$ in $\L_{\mu}^{2}(\Omega;\R^{m})$ converges 
	\begin{itemize}
		\item[\emph{(i)}] weakly in $\L_{\mu}^{2}(\Omega;\R^{m})$ to some function $u\in\L_{\mu}^{2}(\Omega;\R^{m})$, and 
		\item[\emph{(ii)}] pointwisely $\mathscr{L}^{n}$-a.e. to some measurable function $v\colon\Omega\to\R^{m}$. 
	\end{itemize}
	Then $u=v$. 
\end{lemma}

\begin{proof}
Since $\L_{\mu}^{2}(\Omega;\R^{m})$ is a Hilbert space, the Banach--Saks theorem \cite[Ex. 5.24]{Brezis} implies the existence of a subsequence $(u_{j(k)})$ of $(u_j)$ such that the corresponding C\'{e}saro means $v_{N}\coloneqq N^{-1}\sum_{k=1}^{N}u_{j(k)}$ converge strongly in $\L_{\mu}^{2}(\Omega;\R^{m})$ to $u$. On the one hand, passing to another subsequence $(v_{N(\ell)})$ of $(v_{N})$, we can assume that $v_{N(\ell)}\to u$ pointwisely $\mu$-a.e.\ in~$\Omega$ and therefore, by our assumption $\theta \geq 1$ on the density of~$\mu$, $\mathscr{L}^{n}$-a.e.\ in~$\Omega$, as $\ell\to\infty$. On the other hand, as $u_{j}\to v$ pointwisely $\mathscr{L}^{n}$-a.e.\ in~$\Omega$, we necessarily have $v_{N(\ell)}\to v$ pointwisely $\mathscr{L}^{n}$-a.e.\ in~$\Omega$, and therefore, by uniqueness ($\mathscr{L}^{n}$-a.e.\ in~$\Omega$) of pointwise limits, we end up with $u=v$. 
\end{proof}

\subsection{(Lower semi-)continuity results}
In this section, we collect some (semi-)continuity results which will prove instrumental in the proof of Theorem~\ref{thm:main}. 

First, let $f\colon\rsym \to \R$ be convex and of linear growth from below, meaning that the first inequality in~\eqref{eq:lingrowth1} holds. Given an open and bounded set $\Omega\subset\R^{n}$, we consider $\mu\in\mathrm{RM}_{\mathrm{fin}}(\Omega;\rsym)$ and then define, based on the Lebesgue--Radon--Nikod\'{y}m decomposition~\eqref{eq:LRNdecomp} of~$\mu$, a measure $f(\mu)$ via
\begin{align}\label{eq:funofmeas}
f(\mu)(U)\coloneqq \int_{U}f\Big(\frac{\dif\mu}{\dif\mathscr{L}^{n}}\Big)\dif\mathscr{L}^{n} + \int_{U}f^{\infty}\Big(\frac{\dif\mu}{\dif|\mu^{s}|}\Big)\dif|\mu^{s}| \qquad \text{for Borel subsets } U\subset\Omega.
\end{align}
If $f$ is of linear growth also from above, meaning that also the second inequality in~\eqref{eq:lingrowth1} is satisfied, then the recession function~$f^{\infty}$ of~$f$ defined as $f^\infty(z) \coloneqq \lim_{s \to \infty} s^{-1} f(sz)$ for all $z \in \rsym$ is well-defined, $1$-homogeneous, lower-semicontinuous and convex on~$\rsym$ with values in $[\gamma,\infty)$. As a consequence, there holds $f(\mu)\in\mathrm{RM}_{\mathrm{fin}}(\Omega)$. However, we note that linear growth from above is not required in~\eqref{eq:funofmeas}, so in general we only have $f^{\infty}\colon\rsym\to\R\cup\{\infty\}$. The (semi-)continuity of functionals of the form~\eqref{eq:funofmeas} is ensured by

\begin{theorem}[Reshetnyak (lower semi-)continuity theorem]\label{lem:reshetnyak}
Let $\Omega\subset\R^{n}$ be an open and bounded set and let   $f\colon\rsym\to\R$ be a convex function of linear growth from below. For functions  $u,u_{1},u_{2},\ldots \in\bd(\Omega)$ the following statements hold: 
\begin{enumerate} 
\item If $u_{j}\stackrel{*}{\rightharpoonup}u$ in the {\textup{(}symmetric\textup{)}} weak*-sense in $\bd(\Omega)$, then we have 
\begin{align*}
f(\E u)(\Omega)\leq \liminf_{j\to\infty}f(\E u_{j})(\Omega).
\end{align*} 
\item If $u_{j}\stackrel{\langle\cdot\rangle}{\to}u$ in the {\textup{(}symmetric\textup{)}} area-strict sense in $\bd(\Omega)$, then we have 
\begin{align*}
f(\E u)(\Omega)=\lim_{j\to\infty}f(\E u_{j})(\Omega).
\end{align*} 
\end{enumerate} 
\end{theorem}

This theorem appears as a special case of the results of \textsc{Reshetnyak}~\cite{Reshetynak} (also see~\cite{BeckSchmidt}). If the boundary~$\partial\Omega$ is moreover Lipschitz, one may extend functions $v\in\bd(\Omega)$ by a fixed map $\widetilde{v}\in\bd(\R^{n})$ to all of~$\R^{n}$ to obtain an element of~$\bd(\R^{n})$. In conjunction with~\eqref{eq:BDLNR}, Theorem~\ref{lem:reshetnyak} then yields the following result: 

\begin{corollary}\label{cor:reshetnyak}
Let $\Omega\subset\R^{n}$ be an open and bounded set with Lipschitz boundary. Moreover, let $f\colon\rsym\to\R$ be a convex function of linear growth satisfying~\eqref{eq:lingrowth1}, and define its recession function by~\eqref{eq:recfunc}. Then for every $u_{0}\in\bd(\Omega)$, the functional 
\begin{align*}
\overline{F}_{u_{0}}[u;\Omega] \coloneqq \int_{\Omega} f(\mathscr{E}u)\,\dif x + \int_{\Omega} f^\infty\left(\frac{\mathrm{dE}u}{\mathrm{d}\abs{\mathrm{E}^su}}\right)\!\dif{\abs{\mathrm{E}^su}} + \int_{\partial\Omega} f^\infty(\mathrm{tr}_{\partial\Omega}(u_{0}-u)\odot\nu_{\del\Omega})\dx{\Hd^{n-1}}.
\end{align*}
as in~\eqref{eq:WeakStarRelaxation} is lower semicontinuous with respect to {\textup{(}symmetric\textup{)}} weak*-convergence in $\bd(\Omega)$ and continuous with respect to \textup{(}symmetric\textup{)} area-strict convergence in $\bd(\Omega)$. 
\end{corollary}

\begin{remark}[$\bd$-minimizers and consistency]
\label{rem:consistency_proof} Based on Corollary~\ref{cor:reshetnyak}, we  now comment on the existence of $\bd$-minimizers and the consistency result~\eqref{eq:consistency}. Recalling the compactness result that on a bounded domain~$\Omega$ with Lipschitz boundary every bounded sequence in~$\bd(\Omega)$ contains a weakly* convergent sequence in $\bd(\Omega)$, we obtain the existence of a minimizer of $\overline{F}_{u_{0}}[\cdot;\Omega]$ in $\bd(\Omega)$ by the lower semicontinuity part of Corollary~\ref{cor:reshetnyak}. Concerning~\eqref{eq:consistency}, the inequality ``$\leq$'' is obvious, while the reverse inequality ``$\geq$'' follows from the continuity part of Corollary~\ref{cor:reshetnyak} in conjunction with the (symmetric) area-strict approximation result from Lemma~\ref{lem:SAP}.
\end{remark}



Secondly, we record a lower semicontinuity result for functionals extended from Dirichlet classes to negative Sobolev spaces by infinity. Up to obvious modifications, so the treatment of lower order terms, the proof is identical to {\cite[Lem. 2.6]{Gmeineder}}.

\begin{lemma}\label{lem:EkelandLSC}
Let $\Omega\subset\R^{n}$ be an open and bounded set with Lipschitz boundary, $1<q<\infty$ and $k\in\N$. Moreover, assume that $\mathbf{f}_{1}\colon\rsym\to\R_{\geq 0}$ is a convex function that satisfies $c^{-1}|z|^{q}\leq \mathbf{f}_{1}(z)\leq c(1+|z|^{q})$ for some $c>0$ and all $z\in\rsym$, and that $\mathbf{f}_{2}\colon\R^{n}\to\R_{\geq 0}\cup\{+\infty\}$ is a lower semicontinuous function. Then, for every $u_{0}\in\W^{1,q}(\Omega;\R^{n})$, the functional 
	\begin{align*}
	\mathcal{F}[u] \coloneqq \begin{cases} \displaystyle \int_{\Omega} \big( \mathbf{f}_{1}(\eps(u))+\mathbf{f}_{2}(u) \big)\dx{x}&\;\text{if}\;u\in \mathscr{D}_{u_{0}}\coloneqq u_{0}+\ld_{0}(\Omega),\\
	+\infty&\;\text{if}\;u\in\W^{-k,1}(\Omega;\R^{n})\setminus\mathscr{D}_{u_{0}}
	\end{cases}
	\end{align*}
	is lower semicontinuous with respect to the norm topology on $\W^{-k,1}(\Omega;\R^{n})$. 
\end{lemma}

\subsection{The Ekeland variational principle} 
As alluded to in the introduction, in order to deal with possible non-uniqueness phenomena of $\bd$-minimizers we use the Ekeland variational principle~\cite{Ekeland} in the proof of Theorem~\ref{thm:main}. We give here a version that is tailored to our purposes in Section~\ref{sec:main} below:

\begin{proposition}[Ekeland variational principle, {\cite{Ekeland}, \cite[Thm. 5.6, Rem. 5.5]{Giusti}}]\label{prop:Ekeland}
	Let $(X,d)$ be a complete metric space and let $\mathcal{F}\colon X\to\R\cup\{\infty\}$ be a lower semicontinuous function with respect to the metric topology which is bounded from below and not identically $+\infty$. Suppose that, for some $u\in X$ and some $\varepsilon>0$, there holds $\mathcal{F}[u]\leq \inf_{X} \mathcal{F}+\varepsilon$. Then there exists $v\in X$ such that 
	\begin{enumerate}
		\item $d(u,v)\leq \sqrt{\varepsilon}$,
		\item $\mathcal{F}[v]\leq \mathcal{F}[u]$, 
		\item for all $w\in X$ there holds $\mathcal{F}[v]\leq \mathcal{F}[w] + \sqrt{\varepsilon}d(v,w)$. 
	\end{enumerate}
\end{proposition}
\subsection{Miscellaneous bounds} We conclude this section by recording two estimates as follows:

\begin{lemma}[{\cite[Lem. 5.2]{Giusti}, \cite[Lem. 2.8]{BeckSchmidt}}]\label{lem:boundbelow}
	Suppose that $f\colon\rsym\to\R$ is a convex $\hold^1$-function satisfying~\eqref{eq:lingrowth1}. Then we have the following statements:
	\begin{enumerate}
        \item\label{item:boundbelow2} For all $z\in\rsym$ there holds $|\nabla f(z)|\leq \Gamma$. In particular, we have $\Lip(f)\leq \Gamma$.
		\item\label{item:boundbelow1} For all $z\in\rsym$ there holds $\langle \nabla f(z),z\rangle \geq \gamma |z|-\Gamma$. 
	\end{enumerate}
\end{lemma}

\section{Gradient integrability and the proof of Theorem~\ref{thm:main}}\label{sec:main}
We now give the proof of Theorem~\ref{thm:main}. To this end, it is convenient to use the notation 
\begin{align*}
\lebe_{\leq t}^{\infty}(\Omega;\R^{n}) \coloneqq \{w\in\lebe^{\infty}(\Omega;\R^{n}) \colon \|w\|_{\lebe^{\infty}(\Omega;\R^{n})}\leq t\} \qquad \text{for }  t >0.
\end{align*}
\subsection{The Ekeland-type viscosity approximation}
\label{section:viscosity_approximation}
Towards the claim of Theorem~\ref{thm:main}, which is a local result, let $u\in\bd_{\locc}(\Omega)\cap\lebe_{\locc}^{\infty}(\Omega;\R^{n})$ be a local $\bd$-minimizer of $F$. In particular, on each relatively compact subset $U\Subset\Omega$ with Lipschitz boundary $\partial U$, $u|_{U}\in\bd(U)\cap\lebe^{\infty}(U;\R^{n})$ is a $\bd$-minimizer of $F$ on $U$ with respect to its own boundary values. 

In view of Theorem \ref{thm:main}, we may therefore directly assume that 
\begin{align}\label{eq:uAss}
	u = u_{0} \in\L^{\infty}(\Omega;\R^{n})\cap\bd(\Omega)\;\;\;\text{and define}\;\;\; m \coloneqq \|u\|_{\L^{\infty}(\Omega;\R^{n})}.
\end{align}
We then proceed in three steps. As outlined in the introduction, we must simultaneously keep track of the $\L^{\infty}$-constraint within the Ekeland approximation scheme, and get access to the corresponding algebraic manipulations which let us solely work with symmetric gradients. This necessitates a very careful approximation procedure to be displayed in Step 1. In the subsequent  Step 2, we introduce the corresponding stabilised and $\lebe^{\infty}$-penalising integrands, finally letting us come up with the requisite Ekeland competitors for the main part of the proof in Step 3. 



\emph{Step 1. Construction of regular minimising sequences to regularised boundary values.} In view of Lemma~\ref{lem:SAP}, there exist a constant $C_A=C_A(\Omega,n)>0$ and a sequence $(u_{j})_{j\in\N}$ in $\L_{\leq C_A m}^{\infty}(\Omega;\R^{n})\cap (u_{0}+\hold_{c}^{\infty}(\Omega;\R^{n}))$  such that 
	\begin{align}\label{eq:SASTconv}
		u_j \to u \quad\mbox{in the (symmetric) area-strict topology on }\BD(\Omega). 
	\end{align}
For later purposes, we record that  
\begin{align}\label{eq:beautifulworld}
	  \|u_{j}-u_{0}\|_{\L^{\infty}(\Omega;\R^{n})}\leq (C_A+1)m
\end{align} 	
with~$m$ being introduced in~\eqref{eq:uAss} above.
By Corollary~\ref{cor:reshetnyak}, $\overline{F}_{u_{0}}[-;\Omega]$ is continuous with respect to the (symmetric) area-strict topology in $\bd(\Omega)$. Therefore, the consistency relation~\eqref{eq:consistency} implies that 
\begin{align*}
\lim_{j \to \infty} F[u_{j};\Omega] = \lim_{j \to \infty} \overline{F}_{u_{0}}[u_{j};\Omega] = \overline{F}_{u_{0}}[u;\Omega]=\min_{\bd(\Omega)}\overline{F}_{u_{0}}[-;\Omega] \stackrel{\eqref{eq:consistency}}{=}\inf_{\mathscr{D}_{u_{0}}}F[-;\Omega],
\end{align*}
meaning that $(u_j)_{j \in \N}$ is a minimizing sequence for $F[-;\Omega]$ in $\mathscr{D}_{u_{0}}$.
Hence, passing to a non-relabelled subsequence if required, we may assume 
\begin{align}
\label{eq:minizing_sequence_initial}
F[u_j;\Omega] \leq \inf_{\Di_{u_0}} F[-;\Omega] + \frac{1}{8j^2}\qquad\text{for all}\;j\in\N.
\end{align}
For an open and bounded set~$\Omega_0$ with $\Omega\Subset \Omega_0$, we recall the extension operator $\mathfrak{J}\colon\LD(\Omega)\to\LD_{0}(\Omega_0)$ from Lemma~\ref{lem:ext}. In particular, $\mathfrak{J}\colon \LD(\Omega)\cap\L^{\infty}(\Omega;\R^{n})\to\L^{\infty}(\Omega_0;\R^{n})$ is a bounded linear operator with respect to the $\L^{\infty}$-norm. We denote the corresponding operator norm by 
	\begin{align*}
		\|\mathfrak{J}\|_\infty \coloneqq \sup\{\|\mathfrak{J}v\|_{\L^{\infty}(\Omega_0;\R^{n})}\colon v\in\LD(\Omega)\cap \L^{\infty}_{\leq 1}(\Omega;\R^{n})\}
	\end{align*} 
	and define $\overline{u}_{0}\coloneqq \mathfrak{J}u_{0}$. We next consider a radial standard mollifier $\varrho\in\C_{c}^{\infty}(\B_{1}(0);[0,1])$  with $\|\varrho\|_{\L^{1}(\B_{1}(0))}=1$ and set, for $\varepsilon>0$, $\varrho_{\varepsilon}(x)\coloneqq \varepsilon^{-n}\varrho(\frac{x}{\varepsilon})$. For each $j\in\N$, we then choose 
	$\varepsilon_{j}>0$ such that the mollification $u_{j}^{\partial\Omega}\coloneqq (\varrho_{\varepsilon_{j}}*\overline{u}_{0})|_{\Omega}$ via convolution satisfies
	\begin{align}\label{eq:verycarefulapproximation}
	\begin{split}
		&\|u_{j}^{\partial\Omega}-u_{0}\|_{\LD(\Omega)}\leq \frac{1}{8\Lip(f)j^{2}},\\
		&\|u_{j}^{\partial\Omega}\|_{\L^{\infty}(\Omega;\R^{n})}\leq \|\mathfrak{J}\|_{\infty} \|u_{0}\|_{\L^{\infty}(\Omega;\R^{n})}=\|\mathfrak{J}\|_\infty m
		\end{split}
	\end{align} 
	seriatim. Having constructed the sequence $(u_{j}^{\partial\Omega})_{j\in\N}$ in $ (\W^{1,n+1}\cap{\L}{_{\leq \|\mathfrak{J}\|_\infty m}^{\infty}})(\Omega;\R^{n})$, we then define approximate Dirichlet classes by 
	\begin{equation*}
	\mathscr{D}_{j}\coloneqq u_{j}^{\partial\Omega}+ \ld_0(\Omega).
	\end{equation*}
	For $\widetilde{u}_{j}\coloneqq u_j-u_0 + u_{j}^{\partial\Omega} \in\mathscr{D}_{j}$ we then conclude on the one hand via~\eqref{eq:verycarefulapproximation}$_1$
	\begin{equation}\label{eq:LDdistanceMollification}
			\|\widetilde{u}_j - u_j\|_{\LD(\Omega)} 
			=
			\|u_j^{\del\Omega} - u_0\|_{\LD(\Omega)}
			\leq \frac{1}{8\Lip(f)j^2}, 
    \end{equation}
	and on the other hand via~\eqref{eq:beautifulworld} and~\eqref{eq:verycarefulapproximation}$_{2}$  
		\begin{align}\label{eq:veryimportantLinftybounds}
			\|\widetilde{u}_{j}\|_{\L^{\infty}(\Omega;\R^{n})}\leq \|u_j-u_0\|_{\L^{\infty}(\Omega;\R^{n})} + \|u_{j}^{\partial\Omega}\|_{\L^{\infty}(\Omega;\R^{n})} \leq (1+C_A+\|\mathfrak{J}\|_\infty )m \eqqcolon M
		\end{align}
	for all $j\in\N$. We next show that the infimum of $F[-;\Omega]$ in the Dirichlet class $\mathscr{D}_{u_0}$ is approximated by the corresponding one in the Dirichlet class~$\mathscr{D}_{j}$, which in turn is almost attained for the function~$\widetilde{u}_{j}$. To this end, we first notice 
	\begin{equation*}
    \label{eq:F_difference}
     |F[v;\Omega] - F[w;\Omega] | \leq \Lip(f)\|\eps(v-w)\|_{\L^1(\Omega;\R^{n \times n})} \qquad \text{for all } v,w \in \ld(\Omega).
	\end{equation*}
    For any function $\varphi\in \ld_{0}(\Omega)$ we hence obtain the inequalities
    \begin{align*}
     & F[u_0+\varphi;\Omega] \leq F[u_j^{\del\Omega}+\varphi;\Omega] + \Lip(f) \|\eps(u_0-u_j^{\del\Omega})\|_{\L^1(\Omega;\R^{n \times n})}, \\
     & F[u_j^{\del\Omega}+\varphi;\Omega] \leq F[u_0+\varphi;\Omega]  + \Lip(f) \|\eps(u_0-u_j^{\del\Omega})\|_{\L^1(\Omega;\R^{n \times n})}.
    \end{align*}
    Taking into account~\eqref{eq:LDdistanceMollification}, we now infimise on the right-hand sides over all functions $\varphi\in \ld_0(\Omega)$ and arrive at
    \begin{equation}
    \label{eq:intermediateest1}
     \Big| \inf_{\Di_{u_0}} F[-;\Omega] - \inf_{\Di_{j}}F[-;\Omega] \Big|  \leq \frac{1}{8j^2},
    \end{equation}
    while the specific choice $\varphi = u_j - u_0$ in the second inequality in combination with~\eqref{eq:minizing_sequence_initial} yields
    \begin{equation}
    \label{eq:anothermolli}
     F[\widetilde{u}_{j};\Omega] \leq F[u_j ;\Omega] + \frac{1}{8j^2} \leq \inf_{\Di_{u_0}} F[-;\Omega] +  \frac{1}{4j^2}
    \end{equation}
    for all $j \in \N$. Note that the sequence $(\widetilde{u}_{j})_{j \in \N}$ has still the relevant properties of being smooth in~$\Omega$ with uniform $\lebe^{\infty}$-bound and approximating~$u$ in the (symmetric) area-strict topology, but has more regular boundary values.
    
\emph{Step 2. Definition of the $\W^{1,n+1}$-stabilised and $\lebe^{\infty}$-penalizing integrands.}	We next introduce suitable stabilised variational integrals defined in terms of $\widetilde{u}_{j}$. This will allows us to simultaneously keep control of the $\L^{\infty}$-bounds in the Ekeland variational principle in Step 3. To do so, we denote by $C_{\mathrm{M}}=C_{\rm M}(n,\Omega)>0$ a constant for the Morrey embedding $\W^{1,n+1}(\Omega;\R^{n})\hookrightarrow\C^{0,\frac{1}{n+1}}(\overline{\Omega};\R^{n})$ such that every $w\in\W^{1,n+1}(\Omega;\R^{n})$ satisfies the inequality
		\begin{align}\label{eq:optimalHoelderbound}
			|w(x)-w(y)|\leq C_{\rm M}\|w\|_{\W^{1,n+1}(\Omega;\R^{n})}|x-y|^{\frac{1}{n+1}}\qquad\text{for all } x,y \in  \overline{\Omega}.
		\end{align}
	Equally, we denote by $C_{\mathrm{K},n+1}=C_{{\rm K},n+1}(n,\Omega)\geq 1$ a constant such that every $w\in\W^{1,n+1}(\Omega;\R^{n})$ satisfies the Korn-type inequality 
		\begin{align}\label{eq:Kornn+1inequality}
			\|w\|_{\W^{1,n+1}(\Omega;\R^{n})}\leq C_{{\rm K},n+1}(\|w\|_{\L^{n+1}(\Omega;\R^{n})}+\|\eps(w)\|_{\L^{n+1}(\Omega;\rsym)}).
		\end{align}
	We may assume the estimate 
		\begin{equation}\label{eq:introUpsilon}
			\|\eps(\widetilde{u}_{j})\|_{\L^{n+1}(\Omega;\rsym)} \leq \Upsilon(j)
		\end{equation}
    for the blow-up rate of the norms $\|\eps(\widetilde{u}_{j})\|_{\L^{n+1}(\Omega;\rsym)}$ with a convex and increasing $\hold^2$-function $\Upsilon\colon\R_{\geq 0}\to\R_{\geq 0}$ with $\Upsilon(t)\to \infty$ as $t\to\infty$. We then consider a convex and increasing function $h\colon [0,2)\to [0,\infty)$ which satisfies $h=0$ on $[0,1]$ and, for $\frac{3}{2}\leq t < 2$, is even strictly increasing with
		\begin{align}\label{eq:subtlehconstruction}
			\frac{1}{\omega_{n}}\Big(\frac{4C_{\rm M}}{M (2-t)^{4}}\Big(1+2^{\frac{n-1}{2}}\Big(\mathscr{L}^{n}(\Omega)+\Big(\Upsilon\Big(\frac{1}{2-t}\Big)\Big)^{n+1} \Big)\Big)\Big)^{n(n+1)}<h(t).
		\end{align}
    The construction of such functions~$\Upsilon$ and~$h$ is elementary and is briefly addressed for the reader's convenience in Appendix~\ref{sec:Ups}. 
    
    Now let $g\colon\R^n \to \R\cup\{+\infty\}$ be given by $g(\cdot)=\widetilde{g}(|\cdot|)$, where
		\begin{align}\label{eq:gdefLinfty}
			\widetilde{g}(t)\coloneqq
			\begin{cases} 
			\displaystyle h(t) &\;\text{if}\;0\leq t < 2,\\
			+\infty &\;\text{if}\;t\geq 2.
			\end{cases}
		\end{align}
    We record that the function $g$ is convex, lower semicontinuous,  and its restriction to $\ball_{2}(0)$ is of class $\C^{2}$.  We now define the perturbed integrands $f_{j}\colon \R^{n\times n}_\sym\to \R$ by 
		\begin{align}\label{eq:AjdefLinfty}
		f_{j}(z)\coloneqq f(z)+\frac{1}{2A_{j}j^{2}}(1+|z|^{2})^{\frac{n+1}{2}} \quad \text{with} \quad A_{j}\coloneqq 1+\int_{\Omega}(1+|\eps(\widetilde{u}_{j})|^{2})^{\frac{n+1}{2}}\dx{x}. 
		\end{align}
    For future reference, we compute 
		\begin{align}
		\label{eq:nabla_f_j}
		\nabla f_j(z) & = \nabla f(z) + \frac{n+1}{2A_{j}j^{2}} (1+|z|^{2})^{\frac{n-1}{2}} z, \\
		\label{eq:nabla_2_f_j}
		\nabla^2 f_j(z) & = \nabla^2 f(z) + \frac{n+1}{2A_{j}j^{2}} (1+|z|^{2})^{\frac{n-3}{2}} \big((1+|z|^{2}) \1_{(n \times n) \times (n \times n)}+(n-1) z \otimes z\big),
		\end{align}
    for all $z \in \R^{n\times n}_\sym$. We note that by a simple convexity argument and~\eqref{eq:introUpsilon} we have 
		\begin{align*}
			A_{j} \leq 1+ 2^{\frac{n-1}{2}}(\mathscr{L}^{n}(\Omega)+(\Upsilon(j))^{n+1}),
		\end{align*}
	and thus, choosing $t=2-\frac{1}{j}$ in~\eqref{eq:subtlehconstruction} for $j\geq 2$, we find the estimate
		\begin{align}\label{eq:goodgtwiddleinequality}
			\frac{1}{\omega_{n}}\Big(\frac{4C_{\rm M}A_{j}j^{4}}{M}\Big)^{n(n+1)}<\widetilde{g}\Big(2-\frac{1}{j}\Big),
		\end{align}
	which shall turn out important below. We finally introduce the \emph{$\W^{1,n+1}$-stabilised, $\lebe^{\infty}$-penalizing integrals} by
		\begin{align}\label{eq:FjLinfty}
			F_{j}[w;\Omega]\coloneqq
			\begin{cases} 
			\displaystyle\int_{\Omega}f_{j}(\eps(w))\dx{x} + \int_{\Omega}g\Big(\frac{w}{M}\Big)\dx{x} &\;\text{if}\;w\in\mathscr{D}_{j},\\
			+\infty &\;\text{if}\;w\in \W^{-2,1}(\Omega;\R^{n})\setminus\mathscr{D}_{j}. 
			\end{cases} 
		\end{align}
	We emphasize that, by construction and Korn's inequality, $F_{j}[w;\Omega]$ can only be finite if $w \in (\W^{1,n+1}\cap{\L}{_{\leq 2M}^{\infty}})(\Omega;\R^{n})$ (and thus in particular in $\hold(\overline{\Omega};\R^n)$).	
	
\emph{Step 3. The Ekeland-type approximations $v_{j}$ and the Euler--Lagrange inequality.} We note that, for every $j \in \N$, $F_{j}[-;\Omega]$ is lower semicontinuous with respect to the norm topology on $\W^{-2,1}(\Omega;\R^{n})$, cf. Lemma~\ref{lem:EkelandLSC} with $\mathbf{f}_{1}=f_{j}$, $\mathbf{f}_{2}=g(\cdot/M)$, $q=n+1$ and $k=2$, and obviously, $F_{j}[-;\Omega]$ is not identically $+\infty$ on $\W^{-2,1}(\Omega;\R^{n})$. By virtue of $\eqref{eq:veryimportantLinftybounds}$ and the definition of~$g$ and~$A_j$ in the first step, we then find 
	\begin{align}\label{eq:Ekelandestimatecase2a}
	F_{j}[\widetilde{u}_{j};\Omega] \leq F[\widetilde{u}_{j};\Omega] + \frac{1}{2j^{2}}
	& \stackrel{\eqref{eq:anothermolli}}{\leq}  \inf_{\Di_{u_0}} F[-;\Omega] +  \frac{3}{4j^2} \\
	& \stackrel{\eqref{eq:intermediateest1}}{\leq} \inf_{\mathscr{D}_{j}} F[-;\Omega] + \frac{1}{j^{2}} \leq  \inf_{\W^{-2,1}(\Omega;\R^{n})} F_{j}[-;\Omega] + \frac{1}{j^{2}}. \nonumber
	\end{align}
	Via the Ekeland variational principle from Proposition~\ref{prop:Ekeland} we then find, for each $j\in\N$, a function $v_{j}\in\W^{-2,1}(\Omega;\R^{n})$ such that 
	\begin{align}\label{eq:almostoptimal4Linfty}
	\begin{split}
	&\|v_{j}-\widetilde{u}_{j}\|_{\W^{-2,1}(\Omega;\R^{n})}\leq \frac{1}{j},\\ 
	&F_{j}[v_{j};\Omega] \leq F_{j}[w;\Omega] + \frac{1}{j}\|v_{j}-w\|_{\W^{-2,1}(\Omega;\R^{n})}\qquad\text{for all}\;w\in \W^{-2,1}(\Omega;\R^{n}). 
	\end{split}
	\end{align}
We now proceed to draw several conclusions from~\eqref{eq:almostoptimal4Linfty}, which are in particular crucial for the derivation of the perturbed Euler--Lagrange inequality.

\begin{proposition}\label{lem:ELperutrbed}
Let $(v_{j})_{j\in\N}$ be the Ekeland-type approximation sequence from above. Then for every $j \in \N$ the following estimates hold:  
\begin{align}
\label{v_j_estimate_1}
  & \int_{\Omega}|\eps(v_{j})|\dx{x} \leq \frac{1}{\gamma} \Big(\inf_{\mathscr{D}_{u_{0}}}F[-;\Omega] + \frac{2}{j^{2}}\Big), \\
\label{v_j_estimate_n}  
  & \frac{1}{2A_{j}j^{2}}\int_{\Omega}(1+|\eps(v_{j})|^{2})^{\frac{n+1}{2}}\dx{x} \leq \frac{2}{j^{2}}, \\
\label{v_j_estimate_bounded} 
  & \sup_{j\in\N} \|v_{j}\|_{\lebe^\infty(\Omega;\R^n)}\leq 2M,
\end{align}
where $\gamma$ is as in the linear growth condition~\eqref{eq:lingrowth1} and~$M$ as  in~\eqref{eq:veryimportantLinftybounds}. Moreover, there exists $j_{0}\in\N_{\geq 2}$ depending only on $n$,~$\Omega$ and~$M$ such that 
\begin{equation}
\label{v_j_estimate_bounded_strict} 
  \|v_{j}\|_{\lebe^\infty(\Omega;\R^n)} \leq \Big(2- \frac{1}{j^2}\Big) M <2M \quad \text{for every }j\geq j_{0}.
\end{equation}
\end{proposition}

\begin{proof}
We initially consider $j \in \N$ arbitrary. We start by testing $\eqref{eq:almostoptimal4Linfty}_{2}$ with $w=\widetilde{u}_{j}$. Using~\eqref{eq:almostoptimal4Linfty}$_{1}$ in conjunction with~\eqref{eq:Ekelandestimatecase2a} we then obtain
\begin{align}\label{eq:usefulfjboundsLinfty}
	F_{j}[v_{j};\Omega]\leq F_{j}[\widetilde{u}_{j};\Omega] + \frac{1}{j^2} \leq \inf_{\mathscr{D}_{u_{0}}} F[-;\Omega] + \frac{7}{4j^{2}},
	\end{align} 
so that, in particular, $F_{j}[v_{j};\Omega]$ is finite. Now, on the one hand, inequality~\eqref{v_j_estimate_1} follows directly from the linear growth condition~\eqref{eq:lingrowth1}. On the other hand, in view of the definition of the functional~$F_j$ in~\eqref{eq:FjLinfty}, we deduce $v_{j}\in\mathscr{D}_{j} \cap \sobo^{1,n+1}(\Omega;\R^n)$ and, with $v_{j}\in \C(\overline{\Omega};\R^{n})$, 
\begin{equation*}
	\sup_{\Omega} |v_{j}|\leq 2M \quad \text{and} \quad \mathscr{L}^{n}(\{x \in \Omega \colon |v_{j}(x)|=2M\})=0,
\end{equation*}
which proves~\eqref{v_j_estimate_bounded}. Since this bound is not good enough for the subsequent derivation of the Euler--Lagrange inequality, see Remark~\ref{rem:W1n1regularisations} further below, we next need to improve this estimate to a strict inequality for sufficiently large indices. Towards this aim, we utilise $v_{j} \in \mathscr{D}_{j}$ and estimate via~\eqref{eq:usefulfjboundsLinfty} 
\begin{equation*}
 \inf_{\mathscr{D}_{j}} F[-;\Omega] \leq F[v_{j};\Omega] \leq  F_j[v_{j};\Omega] \leq \inf_{\mathscr{D}_{u_{0}}} F[-;\Omega] + \frac{7}{4j^{2}}.
\end{equation*}
In view of~\eqref{eq:intermediateest1} we then obtain
\begin{equation}
\label{eq:perturbation_v_j_small}
 \frac{1}{2A_{j}j^{2}} \int_{\Omega}(1+|\eps(v_{j})|^{2})^{\frac{n+1}{2}}\dx{x} + \int_{\Omega}g\Big(\frac{v_{j}}{M}\Big)\dx{x} = F_{j}[v_{j};\Omega] - F[v_{j};\Omega] \leq \frac{2}{j^2},
\end{equation}
which in particular proves~\eqref{v_j_estimate_n} and the estimate $\|\eps(v_{j})\|_{\L^{n+1}(\Omega;\R^{n \times n})} \leq 2A_j$. Recalling the Korn-type inequality from~\eqref{eq:Kornn+1inequality}, we may now choose $j_{0}\in\N_{\geq 2}$ depending only on $n$,~$\Omega$ and~$M$ sufficiently large such that the inequalities
\begin{equation*}
 C_{{\rm K},n+1} \leq j_0^2 \quad \text{and} \quad  C_{{\rm K},n+1} M \mathscr{L}^{n}(\Omega)^{\frac{1}{n+1}}\leq j_0^{2}
\end{equation*}
are satisfied. In view of~\eqref{eq:Kornn+1inequality} and~\eqref{v_j_estimate_bounded} we then estimate 
\begin{align*}
	\|v_{j}\|_{\W^{1,n+1}(\Omega;\R^{n})} & \leq C_{{\rm K},n+1}(\|v_{j}\|_{\L^{n+1}(\Omega;\R^{n})}+\|\eps(v_{j})\|_{\L^{n+1}(\Omega;\R^{n \times n})}) \\ & \leq C_{{\rm K},n+1}\big(2M\mathscr{L}^{n}(\Omega)^{\frac{1}{n+1}} + 2A_{j}\big)\leq 4 A_{j}j^{2}
\end{align*}
for all $ j \geq j_0$.  Thus, as a consequence of~\eqref{eq:optimalHoelderbound}, we find 
\begin{equation}
	\label{eq:v_j_Hoelder}
	|v_{j}(x)-v_{j}(y)|\leq 4C_{\rm M}A_{j}j^{2}|x-y|^{\frac{1}{n+1}}\qquad\text{for all }x,y\in \overline{\Omega}. 
\end{equation}
By the strict monotonicity of $\widetilde{g}$ on $[\frac{3}{2},2)$ and by~\eqref{eq:perturbation_v_j_small}, we then infer 
\begin{align*}
	\mathscr{L}^{n}\Big(\Big\{ x \in \Omega \colon \frac{|v_{j}(x)|}{M}>2-\frac{1}{j} \Big\} \Big)  & =\mathscr{L}^{n}\Big(\Big\{ x \in \Omega \colon \widetilde{g}\Big(\frac{|v_{j}(x)|}{M}\Big)>\widetilde{g}\Big(2-\frac{1}{j} \Big)\Big\} \Big)\\
	& \!\!\!\!\leq \frac{1}{\widetilde{g}\Big(2-\frac{1}{j} \Big)}\int_{\Omega}g\Big(\frac{v_{j}}{M}\Big)\dx{x}\leq \frac{1}{\widetilde{g}\Big(2-\frac{1}{j} \Big)} < \omega_{n}\Big(\frac{M}{4C_{\rm M}A_{j}j^{4}}\Big)^{n(n+1)}, 
\end{align*}
	where the last inequality is valid by construction of $g$, see~\eqref{eq:goodgtwiddleinequality}. This inequality implies that the set $\Omega_{j}^{M} \coloneqq \{x \in \Omega \colon |v_{j}(x)|/M>2-\frac{1}{j}\}$ cannot contain a ball of radius $r_j \coloneqq (M/(4C_{\rm M}A_{j}j^{4}))^{n+1}$. We now argue that $|v_j(x)|$ is strictly below $2M$ with a quantified ($j$-dependent) estimate. For this purpose, we distinguish points away from and close to the boundary. 
	
	\begin{figure}
	\begin{center}
	\begin{tikzpicture}[scale=1.5]
	\draw[-, blue!60!white, fill=blue!10!white, opacity=0.4] (0,0) [out = -30, in =240] to (3,1) [out=60, in = 0] to (2.5,2.5) [out=180, in =-30] to (0.5,3) [out = 150, in = 40] to (-0.5,1.5) [out = 220, in =150] to (0,0); 
			\draw[-,blue!30!white,fill = blue!30!white,opacity=0.5] (0,0) [out=70, in =0] to (-0.05,2.45) -- (0.02,2.7) [out=-40, in = 220] to (0.8,2.84) --  (1.1,2.71) [out = 250, in = 180] to (0.9,1.3) [out=0, in =250] to (1.3,2.64) -- (1.5,2.59) [out=270, in = 180] to (1.9,1.0) [out = 0, in =120] to (2.3,0.2) -- (2.1,0.07) [out = 110, in = 0] to (1.75,0.75) [out=180, in =90] to (1.6,-0.152) -- (1.5,-0.19) [out=90, in = 0] to (1,1) [out=180, in =70] to (0.25,-0.12) -- (0,0); 
				\draw[-,blue!30!white,fill = blue!30!white,opacity=0.5] (2.15,1.25) [out=-70, in =180] to (2.5,0.75) [out = 0, in =280] to (2.75,1.5) [out = 100, in = -20] to (2.3,2.25) [out =160, in = 110] to (2.15,1.25);
				\node at (-0.25,0.5) {\footnotesize\textbullet};
					\node[right] at (-0.25,0.5) {\large $\widetilde{x}$};
					\node[right] at (1.65,1.28) {\large $x$};
				\draw [dotted,thick] (-0.25,0.5) circle [radius=1.225];
					\node at (1.65,1.28) {\footnotesize\textbullet};
						\draw [dotted,thick] (1.65,1.28) circle [radius=1.225];
				\draw[blue!40!white,-] (1.35,2.25) -- (2.95, 3.05);
					\draw[blue!40!white,-] (2.35,1.85) -- (3.15, 2.95);
					\node[blue!80!white] at (3.25,3.1) {$\Omega_{j}^{M}$};
					\node[blue!50!white] at (0.075,2) {\LARGE $\Omega$};
						\draw[<->]  (-0.25,0.5) -- (-1.0,-0.5);
						\node at (-0.8,-0) {$r_{j}$};
	\end{tikzpicture}
	\end{center}
	\caption{The geometric situation in the globalisation of the $\lebe^{\infty}$-threshold from $\Omega_{j}^{M} \coloneqq \{x\in\Omega\colon\;|v_{j}(x)|/M>2-\frac{1}{j}\}$ to $\Omega$ by use of H\"{o}lder continuity. Here, $x$ corresponds to the first case, see~\eqref{eq:caso1}, and $\widetilde{x}$ to the second case, see~\eqref{eq:caso2}. In both cases, the balls with radius $r_{j}$ have non-trivial intersection with $\Omega_{j}^{M}$, and the H\"{o}lder continuity allows to amplify the region where $|v_{j}|$ is at least $M(\frac{1}{j}-\frac{1}{j^{2}})$ away from $2M$ to the entire $\overline{\Omega}$.}
	\end{figure}
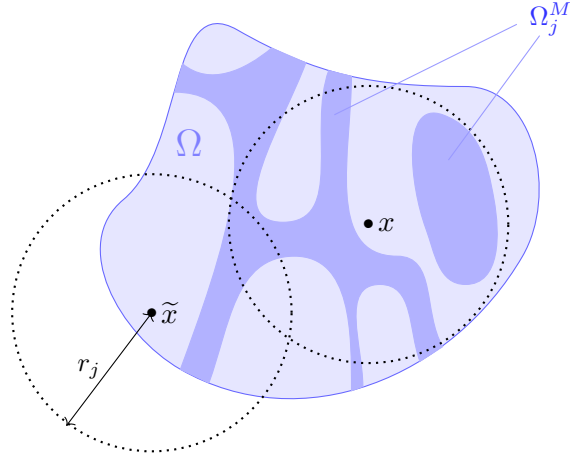
	
	Firstly, for $x \in \Omega$ with $\dist(x,\partial \Omega) \geq r_j$ (implying $\ball_{r_j}(x) \subset \Omega$) and $|v_{j}(x)|/M >2-\frac{1}{j}$, there exists $y\in\Omega$ with $|x-y| < r_j$ and $|v_{j}(y)|/M\leq 2-\frac{1}{j}$. Since $v_{j}\in\C^{0,\frac{1}{n+1}}(\overline{\Omega};\R^{n})$ with~\eqref{eq:v_j_Hoelder}, the definition of~$r_j$ leads to
	\begin{align}\label{eq:caso1}
	\frac{|v_{j}(x)|}{M} \leq \frac{|v_{j}(y)|}{M} + \frac{4C_{\rm M}A_{j}j^{2}}{M}|x-y|^{\frac{1}{n+1}}\leq 2-\frac{1}{j}+\frac{1}{j^{2}} < 2. 
	\end{align}
	Secondly, for $x \in \overline{\Omega}$ with $\dist(x,\partial \Omega) < r_j$, there exists a boundary point $y \in \partial \Omega$ with $|x-y| < r_j$. Taking now advantage of $v_j(y) = u_{j}^{\partial\Omega}(y)$ combined with the bound~\eqref{eq:verycarefulapproximation}$_2$, we find in a similar way 
	\begin{equation}\label{eq:caso2}
	 \frac{|v_{j}(x)|}{M} \leq \frac{|u_{j}^{\partial\Omega}(y)|}{M} + \frac{4C_{\rm M}A_{j}j^{2}}{M}|x-y|^{\frac{1}{n+1}}\leq 1 + \frac{1}{j^{2}} < 2.
	\end{equation}
	This case distinction implies $ \|v_{j}\|_{\lebe^\infty(\Omega;\R^n)} \leq (2- \frac{1}{j^2}) M$ and hence the final claim~\eqref{v_j_estimate_bounded_strict}. This completes the proof of the proposition.
\end{proof} 	

\begin{lemma}[Euler--Lagrange inequality] \label{lem:Euler_Lagrange_inequality}
Let $(v_{j})_{j\in\N}$ be the Ekeland-type approximation sequence from above and $j_{0}\in\N$ as in the previous Proposition~\ref{lem:ELperutrbed}. Then the following \emph{Euler--Lagrange inequality} is valid for all $j\geq j_{0}$ and $\varphi\in\sobo_{0}^{1,n+1}(\Omega;\R^{n})$:
\begin{align}\label{eq:EL2}
	\left\vert \int_{\Omega}\big\langle \nabla f _{j}(\eps(v_{j})),\eps(\varphi)\big\rangle \dx{x} + \int_{\Omega} \Big\langle \nabla g\Big(\frac{v_{j}}{M}\Big),\frac{\varphi}{M}\Big\rangle\dx{x}\right\vert \leq \frac{1}{j}\|\varphi\|_{\W^{-2,1}(\Omega;\R^{n})}. 
	\end{align}
\end{lemma} 

\begin{proof}
Let $\varphi\in\W_{0}^{1,n+1}(\Omega;\R^{n})$ be arbitrary. Then, since $\varphi\in\C_{0}(\Omega;\R^{n})$ by Morrey's embedding, we infer from~\eqref{v_j_estimate_bounded_strict} that for $\theta>0$ sufficiently small there holds $|v_{j}+\theta\varphi|<2M$ too. Hence we have $F_{j}[v_{j}+\theta\varphi]<\infty$, and since $g$ is differentiable on $\ball_{2}(0)$, we can test $\eqref{eq:almostoptimal4Linfty}_{2}$ with the functions $w = v_j \pm \theta \varphi$. Letting $\theta\searrow 0$, we then obtain the claim~\eqref{eq:EL2}. 
\end{proof}

We conclude this section with some comments on our construction of the Ekeland-type approximation sequence:
\begin{remark}[$\lebe^{\infty}$-threshold]\label{rem:W1n1regularisations}
The specific strategy as displayed above is motivated by Proposition~\ref{lem:ELperutrbed} and Lemma~\ref{lem:Euler_Lagrange_inequality}. Here, the key point is the quantified $\lebe^{\infty}$-bound~\eqref{v_j_estimate_bounded_strict}: If we only had the slightly weaker estimate  $\|v_{j}\|_{\lebe^{\infty}(\Omega;\R^{n})}\leq 2 M$, the Euler--Lagrange inequality could not be derived in the requisite form. Indeed, even for $\varphi\in \hold_{c}^{\infty}(\Omega;\R^{n})$, $F_{j}[v_{j}+t\varphi]=\infty$ then could potentially happen howsoever small $|t|>0$ might be. Compared with the BV-setting considered in~\cite{BeckSchmidt,SchmidtHabil}, the bound~\eqref{v_j_estimate_bounded_strict} moreover controls the minimal distance of~$\|v_{j}\|_{\lebe^\infty(\Omega;\R^n)}$ to the threshold $2M$ in a precisely quantified way. This is also the chief reason for considering $\sobo^{1,n+1}$-stabilisations (instead of $\sobo^{1,2}$-stabilisations considered in~\cite{Gmeineder}), letting us work with bounded \emph{and} continuous approximations.
\end{remark}
	
	\begin{remark}[On bounded minimising sequences]\label{rem:cutoff}In the $\mathrm{BV}$-setting of~\cite{BeckSchmidt}, the specific minimising sequence $(u_{j})_{j\in\N}$ as chosen at the very beginning of the approximation scheme in Step~1 \emph{can a priori be taken to belong to $\L_{\leq m}^{\infty}(\Omega;\R^{n})$}. In the situation of~\cite{BeckSchmidt}, this is achieved by possibly passing to the truncations
		\begin{align*}
			\widehat{u}_j\coloneqq \begin{cases} u_j &\;\text{if}\;|u_j|\leq m,\\
			\frac{u_j}{|u_j|}m&\;\text{if}\;|u_j|>m,
			\end{cases}
		\end{align*}
	which satisfies $|\nabla \widehat{u}_j|\leq |\nabla u_j|$. It is \emph{not} clear to us how to show that the same function satisfies a similar pointwise bound exclusively involving symmetric gradients (possibly up to a constant), and so we are bound to stick to Lemma~\ref{lem:SAP}. 
	\end{remark}
\subsection{Preliminary estimates}
We begin with the investigation of the regularity properties of the sequence $(v_j)_{j \in \N}$ by proving the existence of second order derivatives (while uniform estimates are postponed to the next Section~\ref{subsection:second_order_uniform}) and a differentiated version of the Euler--Lagrange inequality~\eqref{eq:EL2}. To this end, it is convenient to define
\[
\sigma_j \coloneqq \nabla f_j(\eps(v_j)),
\]
where $f_{j}$ is given by~\eqref{eq:AjdefLinfty} and $v_{j}\in (\W^{1,n+1} \cap\L_{\leq 2M}^{\infty})(\Omega;\R^{n})$ is the function obtained from the application of the Ekeland variational principle, for every $j \in \N$.

\begin{lemma}[Non-uniform second order estimates]\label{lem:higherregularityapproximate}
Let $f\in \C^{2}(\rsym)$ satisfy~\eqref{eq:lingrowth1} and, for some $\Lambda\in (0,\infty)$, the bound 
\begin{align}\label{eq:generalisedaelliptic}
0 \leq \langle \nabla^2f(z)\xi,\xi\rangle \leq \Lambda\frac{|\xi|^{2}}{(1+|z|^{2})^{\frac{1}{2}}}\qquad\text{for all}\;z,\xi\in\rsym. 
\end{align}
For every $j \geq j_0$ with $j_0\in\mathds{N}$ as in Proposition~\ref{lem:ELperutrbed} there holds $v_{j}\in \W_{\loc}^{2,2}(\Omega;\R^{n})$ with
	\begin{align}\label{eq:nonunifreg2}
	(1+|\eps(v_{j})|^{2})^{\frac{n-1}{2}}|\partial_{s}\eps(v_{j})|^{2}\in \L_{\loc}^{1}(\Omega)\qquad\text{for all } s\in\{1,\ldots,n\}.
	\end{align} 
\end{lemma}
\begin{proof} 
We take a point $x_0\in \Omega $, two radii $0<r<R<\dist(x_0,\del\Omega)$, and a localization function $\varrho\in \C^\infty_c(\Omega; [0,1])$ with $\1_{\ball_{r}(x_{0})}\leq \varrho\leq\1_{\ball_{R}(x_{0})}$. Furthermore, let $s\in\{1,\dots, n\}$ and  $0<h<\frac{1}{2}(\dist(x_0,\del\Omega)-R)$. We recall that by construction we have $v_{j}\in\W^{1,n+1}(\Omega;\R^{n})$ and we now employ the function $\varphi\coloneqq\Delta_{s,-h}(\varrho^{2}\Delta_{s,h}v_{j})\in \W_{0}^{1,n+1}(\Omega;\R^{n})$ in the Euler--Lagrange inequality~\eqref{eq:EL2} from Lemma~\ref{lem:Euler_Lagrange_inequality}. With the integration by parts formula for finite difference quotients and with $\eps(\varrho^2\Delta_{s,h}v_j) = \varrho^2 \eps(\Delta_{s,h}v_j) + 2\varrho\nabla\varrho \odot \Delta_{s,h} v_j $ this gives
\begin{align}
	\left\vert \int_{\Omega}\big\langle \Delta_{s,h}\nabla f_{j}(\eps(v_{j})),\varrho^2 \eps(\Delta_{s,h}v_j) + 2\varrho\nabla\varrho \odot \Delta_{s,h} v_j \big\rangle \dx{x} + \int_{\Omega} \Big\langle \Delta_{s,h}\nabla g\Big(\frac{v_{j}}{M}\Big),\frac{\varrho^{2}\Delta_{s,h}v_{j}}{M} \Big\rangle  \dx{x} \right\vert \notag \\
	\!\!\!\!\!\!\!\!\!\leq \frac{1}{j}\|\Delta_{s,-h}(\varrho^{2}\Delta_{s,h}v_{j}) \|_{\W^{-2,1}(\Omega;\R^{n})}.\label{EL_inequality_nonuniform}
\end{align}
In order to be able to conveniently use the growth conditions on $\nabla^2 f_j$, we rewrite 
\begin{align*}
		\Delta_{s,h}\nabla f_j(\eps(v_j))(x) 
		=
		\int_0^1 \nabla^2 f_j(\eps(v_j)(x) + th \Delta_{s,h} \eps(v_j)(x))\dx{t}\,\Delta_{s,h} \eps(v_j)(x)
\end{align*}
via the fundamental theorem of calculus, and we then define, for $\mathscr{L}^{n}$-a.e.\ $x\in\ball_{r}(x_{0})$, the bilinear forms $\mathscr{B}_{j,s,h}(x)\colon\rsym\times\rsym\to\R$ by
\begin{align*}
	\mathscr{B}_{j,s,h}(x)[\eta,\xi] & \coloneqq \int_{0}^{1}\big\langle \nabla^2f_{j}\big(\eps(v_{j})(x)+th\Delta_{s,h}\eps(v_{j})(x)\big)\eta,\xi\big\rangle\dx{t}\qquad \text{for all } \xi,\eta\in\rsym.
\end{align*}
We next show that each of these bilinear forms is positive definite (which in particular allows for the application of the Cauchy--Schwarz inequality), with suitable lower and upper bounds. As an upper bound we infer from the formula~\eqref{eq:nabla_2_f_j} for $\nabla^2 f_j(z)$ combined with assumption~\eqref{eq:generalisedaelliptic}
	\begin{align}
	\label{eq:bilinear_form_upper_bound}
	\mathscr{B}_{j,s,h}(x)[\xi,\xi] \nonumber
	&\leq 
	\Lambda \int_0^1  \frac{\abs{\xi}^2}{(1+\abs{\eps(v_j)(x)+th\Delta_{s,h}\eps(v_j)(x)}^2)^{\frac{1}{2}}}\dx{t}  \\
	&\qquad\qquad\qquad\qquad+
	\frac{n(n+1)}{2A_j j^2}  \int_0^1(1+\abs{\eps(v_j)(x)+th\Delta_{s,h}\eps(v_j)(x)}^2)^{\frac{n-1}{2}}\abs{\xi}^2\dx{t} \nonumber \\
	& \leq \Theta(n, \Lambda) \big(1+\abs{\eps(v_j)(x)}^2 + \abs{\eps(v_j)(x+he_s)}^2\big)^{\frac{n-1}{2}} \abs{\xi}^2 
	\end{align}
for a constant~$\Theta$ depending only on~$n$ and~$\Lambda$. Furthermore, using the inequality
\begin{align*}
	\tilde{\theta}(n)(1+|\xi|^{2}+|\eta|^{2})^{\frac{n-1}{2}} \leq \int_{0}^{1}(1+|\xi + t\eta|^{2})^{\frac{n-1}{2}}\dx{t}\qquad\text{for all}\;\xi,\eta\in\rsym
\end{align*}
from \cite[Lemma 2.VI]{CAMPANATO82a}, we get as a lower bound
\begin{equation*}
 \mathscr{B}_{j,s,h}(x)[\xi,\xi] \geq  \frac{\theta(n)}{A_{j} j^2}(1+|\eps(v_{j})(x)|^{2}+|\eps(v_{j})(x+he_{s})|^{2})^{\frac{n-1}{2}}  |\xi|^{2},
\end{equation*}
for a constant~$\theta$ depending only on~$n$. As a direct consequence of this lower bound we find 
\begin{align*}
 \mathrm{I}
	\coloneqq  \int_{\Omega}\big\langle \Delta_{s,h}\nabla f_{j}(\eps(v_{j})),\varrho^{2} \eps(\Delta_{s,h}v_{j})\big\rangle \dx{x}
	& = \int_\Omega \mathscr{B}_{j,s,h}(x)	[\varrho \eps(\Delta_{s,h} v_j), \varrho \eps(\Delta_{s,h}v_j)]\dx{x} \\
	& \geq \frac{\theta(n)}{A_{j} j^2} \int_\Omega (1+|\eps(v_{j})(x)|^{2})^{\frac{n-1}{2}} |\varrho\eps(\Delta_{s,h}v_{j})|^{2}\dx{x} .
\end{align*}
In a very similar way, we obtain from the convexity of the auxiliary function~$g$ 
\begin{align*}
	\int_{\Omega}\Big\langle \Delta_{s,h}\nabla g\Big(\frac{v_{j}}{M}\Big),\frac{\varrho^{2}\Delta_{s,h}v_{j}}{M}\Big\rangle\dx{x} \geq 0. 
\end{align*}
This inequality allows us to take advantage of~\eqref{EL_inequality_nonuniform} in order to get
\begin{equation*}
 \mathrm{I}
	\leq - \int_{\Omega} \mathscr{B}_{j,s,h}(x)[\varrho \eps(\Delta_{s,h}v_j), 2\nabla \varrho\odot \Delta_{s,h}v_j]\dx{x} 
	+
	\frac{1}{j}\|\Delta_{s,-h} (\varrho^2 \Delta_{s,h}v_j)\|_{\W^{-2,1}(\Omega;\R^n)}.
\end{equation*}
Via the Cauchy--Schwarz inequality we can reproduce term~$\mathrm{I}$ (with prefactor~$\tfrac{1}{2}$) on the right-hand side, and after absorbing it on the left-hand side, we arrive at 
\begin{equation*}
 \frac{1}{2} \mathrm{I} \leq \frac{1}{2} \int_\Omega \mathscr{B}_{j,s,h}(x)[2\nabla \varrho \odot\Delta_{s,h}v_j, 2\nabla \varrho \odot\Delta_{s,h}v_j]\dx{x}
	+\frac{1}{j}\|\Delta_{s,-h} (\varrho^2 \Delta_{s,h}v_j)\|_{\W^{-2,1}(\Omega;\R^n)} \coloneqq \mathrm{II} + \mathrm{III}.
\end{equation*}
Via Hölder's inequality, the upper bound~\eqref{eq:bilinear_form_upper_bound} and standard properties of finite difference quotients (see e.g.~\cite[Chapter 7.11]{GILTRU77}) we next estimate
\begin{align*}
 \mathrm{II} & \leq 2 \Theta(n,\Lambda) \int_{\Omega} \big(1+\abs{\eps(v_j)(x)}^2 + \abs{\eps(v_j)(x+he_s)}^2)^{\frac{n-1}{2}} |\nabla\varrho\odot\Delta_{s,h}v_{j}|^{2} \dx{x} \\
 & \leq C(n,\Lambda) \|\nabla \varrho \|_{\L^\infty(\Omega;\R^n)}^2 \| 1 + \abs{\eps(v_j)} \|_{\L^{n+1}(\Omega)}^{n-1} \| \partial_s v_j \|_{\L^{n+1}(\Omega;\R^{n})}^2,
\end{align*}
which, in view of $v_{j}\in\W^{1,n+1}(\Omega;\R^{n})$, provides a bound for~$\mathrm{II}$ that is independent of~$h$. Finally, by inequality~\eqref{eq:neg_Sob_3} for negative Sobolev spaces, we have 
\begin{equation*}
 \mathrm{III} \leq \frac{1}{j} \| \varrho^2 \Delta_{s,h} v_{j} \|_{\L^1(\Omega;\R^n)} \leq \| \partial_s v_j \|_{\L^{1}(\Omega;\R^{n})}.
\end{equation*}
Collecting the estimates for $\mathrm{I}$ from below and above, we obtain  
	\begin{align}
	\label{eq:diff-quot-bounded-H}
	\int_{\Omega}|\varrho\eps(\Delta_{s,h}v_{j})|^{2}\dx{x} \leq \int_{\Omega}(1+|\eps(v_{j})(x)|^{2})^{\frac{n-1}{2}}|\varrho\eps(\Delta_{s,h}v_{j})|^{2}\dx{x} \leq C, 
	\end{align} 
where the constant~$C$ depends on $n, \Omega, \Lambda, j, A_j$ and $\|v_j\|_{\W^{1,n+1}(\Omega;\R^n)}$, but not on~$h$. Therefore, the family $(\Delta_{s,h}\eps(v_{j}))_{h}$ is bounded uniformly in $\L^2(\ball_r(x_0);\R^{n \times n})$, which implies that $\partial_s \eps(v_{j})$ exists in $\L^2(\ball_r(x_0);\R^{n \times n})$.
By arbitrariness of $s \in \{1,\ldots,n\}$, $x_0 \in \Omega$ and $0<r<R<\dist(x_0,\del\Omega)$, Korn's inequality shows $v_{j}\in \W_{\loc}^{2,2}(\Omega;\R^{n})$. Furthermore, with the strong convergence $\Delta_{s,h}\eps(v_{j})\to \partial_{s}\eps(v_{j})$ in $\L^{2}(K;\rsym)$ for a given compactly supported set $K \subset \Omega$, we can select a suitable sequence $(h_{i})_{i\in\N}$ in $\R_{>0}$ with $h_{i}\searrow 0$ and $\Delta_{s,h_{i}}\eps(v_{j})\to \partial_{s}\eps(v_{j})$ $\mathscr{L}^{n}$-a.e. in $K$. Then, by Fatou's lemma, we have 
	\begin{align*}
	\int_{K}(1+|\eps(v_{j})|^{2})^{\frac{n-1}{2}}|\partial_{s}\eps(v_{j})|^{2}\dx{x} \leq \liminf_{i\to\infty} \int_{K}(1+|\eps(v_{j})|^{2})^{\frac{n-1}{2}}|\Delta_{s,h_{i}}\eps(v_{j})|^{2}\dx{x}.
	\end{align*}
Since the right-hand side is bounded by the previous reasoning, we have shown also the second assertion $(1+|\eps(v_{j})|^{2})^{\frac{n-1}{2}}|\partial_{s}\eps(v_{j})|^{2}\in \L_{\loc}^{1}(\Omega)$, which completes the proof of the lemma.
\end{proof}

As a direct consequence of the second order estimates from Lemma~\ref{lem:higherregularityapproximate}, we next show that also a differentiated version of the Euler--Lagrange inequality~\eqref{eq:EL2} is at our disposal.

\begin{lemma}[Differentiated Euler--Lagrange inequality]\label{lem:ELbetter}
	Let $f\in\C^{2}(\rsym)$ satisfy~\eqref{eq:lingrowth1} and, for some $\Lambda\in (0,\infty)$, the bound~\eqref{eq:generalisedaelliptic}. For every $j \geq j_0$ with $j_0\in\mathds{N}$ as in Proposition~\ref{lem:ELperutrbed} we have $\sigma_j\in \W^{1,\frac{n+1}{n}}_\loc(\Omega; \R^{n\times n}_\sym)$ and, for all $s\in\{1,\ldots,n\}$ and $\varphi\in\W_{c}^{1,n+1}(\Omega;\R^{n})$, there holds
		\begin{align}\label{eq:betterEL2}
		\left\vert\int_{\Omega} \langle\partial_{s}\sigma_{j},\eps(\varphi) \rangle\dx{x} + \int_{\Omega}\Big\langle \partial_{s}\Big(\nabla g\Big(\frac{v_{j}}{M} \Big)\Big) ,\frac{\varphi}{M}\Big\rangle\dx{x} \right\vert \leq \frac{1}{j}\|\varphi\|_{\W^{-1,1}(\Omega;\R^{n})}. 
		\end{align}
	\end{lemma}
\begin{proof}
    We first show $\sigma_j\in \W^{1,\frac{n+1}{n}}_\loc(\Omega;\R^{n\times n}_\sym)$ (which in particular proves that the first integral in~\eqref{eq:betterEL2} is well-defined). With $v_j\in \W^{2,2}_\loc(\Omega;\R^n)$ due to Lemma~\ref{lem:higherregularityapproximate}, we observe that~$\sigma_j$ is weakly differentiable with $\del_s\sigma_j = \nabla^2f_j(\eps(v_j))\del_s \eps(v_j)$ for  $s\in\{1,\ldots,n\}$. Since $\abs{\nabla^2f_j(z))}\leq C(n,\Lambda)(1+\abs{z}^2)^{\frac{n-1}{2}}$ holds in view of~\eqref{eq:nabla_2_f_j} for all $z\in \R^{n\times n}_\sym$, we obtain from Hölder's inequality, applied with exponents $\tfrac{2n}{n-1}$ and $\tfrac{2n}{n+1}$, the estimate
	\begin{align*}
	\|\partial_s \sigma_j \|^{\frac{n+1}{n}}_{\L^{\frac{n+1}{n}}(K;\rsym)}
	& \leq \int_K (1+\abs{\eps(v_j)}^2)^{\frac{n-1}{2}\frac{n+1}{n}}\abs{\del_s\eps(v_j)}^{\frac{n+1}{n}}\dx{x}\\
	&= \int_K (1+\abs{\eps(v_j)}^2)^{\frac{n-1}{4}\frac{n+1}{n}}(1+\abs{\eps(v_j)}^2)^{\frac{n-1}{4}\frac{n+1}{n}}\abs{\del_s\eps(v_j)}^{\frac{n+1}{n}}\dx{x}\\
	&\leq \left( \int_K (1+\abs{\eps(v_j)}^2)^{\frac{n+1}{2}}\dx{x}\right)^{\frac{n-1}{2n}}
	\left(\int_K (1+\abs{\eps(v_j)^2})^{\frac{n-1}{2}}\abs{\del_s \eps(v_j)}^2\dx{x} \right)^{\frac{n+1}{2n}}
	\end{align*}
	for every compact subset $K \subset \Omega$. In view of  $v_{j}\in\W^{1,n+1}(\Omega;\R^{n})$ and with the local integrability~\eqref{eq:nonunifreg2} from Lemma~\ref{lem:higherregularityapproximate}, all terms on the right-hand side are finite, which proves the first claim $\sigma_j\in \W^{1,\frac{n+1}{n}}_\loc(\Omega; \R^{n\times n}_\sym)$. We next test the Euler--Lagrange inequality~\eqref{eq:EL2} from Lemma~\ref{lem:Euler_Lagrange_inequality} with $\partial_s \varphi$ for an arbitrary function $\varphi \in \hold^\infty_c(\Omega;\R^n)$. Via the integration by parts formula and the first  inequality in~\eqref{eq:neg_Sob_1}, we then find 
	\begin{multline*}
	\left\vert\int_{\Omega}\langle\partial_{s}\sigma_{j},\eps(\varphi)\rangle\dx{x} + \int_{\Omega}\Big\langle \partial_{s}\Big(\nabla g\Big(\frac{v_{j}}{M} \Big)\Big) ,\frac{\varphi}{M}\Big\rangle\dx{x} \right\vert \\
	= \left\vert \int_{\Omega} \big\langle \nabla f _{j}(\eps(v_{j})),\eps(\partial_s \varphi) \big\rangle + \Big\langle \nabla g\Big(\frac{v_{j}}{M}\Big),\frac{\partial_s \varphi}{M}\Big\rangle\dx{x}\right\vert \leq \frac{1}{j}\|\partial_s \varphi\|_{\W^{-2,1}(\Omega;\R^{n})} \leq \frac{1}{j}\|\varphi\|_{\W^{-1,1}(\Omega;\R^{n})}.
	\end{multline*}
    We finally notice that this inequality holds by an approximation argument also for all functions $\varphi\in\W_{c}^{1,n+1}(\Omega;\R^{n})$, where we take advantage of $\sigma_j\in \W{_{\locc}^{1,\frac{n+1}{n}}}(\Omega; \R^{n\times n}_\sym)$, 
	$g\in \C^2(\ball_2(0))$, \eqref{eq:gdefLinfty} and~\eqref{v_j_estimate_bounded_strict} for the left-hand side, and of the embedding $\W^{1,n+1}(\Omega;\R^n)\hookrightarrow \W^{-1,1}(\Omega;\R^n)$ for the right-hand side. This finishes the proof of the lemma.
\end{proof}
\subsection{Uniform degenerate second order estimates}
\label{subsection:second_order_uniform}

We next improve on the second order estimates of the approximating sequence $(v_j)_{j \in \N}$ and derive estimates, which are uniform in $j \in \N$. These estimates will be the essential ingredient for proving the superlinear estimates on the gradients $(\nabla v_j)_{j \in \N}$ in the next Section~\ref{section:proof_main_theorem}.

\begin{theorem}[Uniform second order estimates]\label{thm:uniformSecondOrder}
	Let $f\in\C^{2}(\rsym)$ satisfy~\eqref{eq:lingrowth1} and, for some $\Lambda\in (0,\infty)$, the bound~\eqref{eq:generalisedaelliptic}. Then there exists a constant $c=c(n,\Gamma,\Lambda,M)$ such that for every ball $\ball_{2r}(x_{0})\Subset\Omega$ and every function $\varrho\in\C_{c}^{\infty}(\B_{2r}(x_{0});[0,1])$ with $\1_{\B_{r}(x_{0})}\leq \varrho \leq \1_{\B_{2r}(x_{0})}$  and $\abs{\nabla^s\varrho}\leq\left(\frac{2}{r}\right)^s$ for $s\in\{1,2,3\}$ there holds  
	\begin{multline}\label{eq:dualauxest}
	\sum_{k=1}^{n}\int_{\ball_{2r}(x_{0})} \varrho^4 \big\langle \nabla^2f_{j}(\eps(v_{j}))\partial_{k}\eps(v_{j}),\partial_{k}\eps(v_{j})\big\rangle\dx{x} \\
    \leq 
	\frac{c}{r^2} \bigg[ \int_{\ball_{2r}(x_{0})} \Big(\frac{1}{r}+\frac{1}{r^2} + \frac{r^2}{j} + \abs{\eps(v_j)} \Big) \dx{x} + \frac{1}{A_j j^2}\int_{\ball_{2r}(x_{0})}\Big(1+\frac{1}{r^2} + \abs{\eps(v_j)}^2\Big)^{\frac{n+1}{2}} \dx{x} \bigg] 
	\end{multline}
	for all $j\geq j_0$ with $j_0 \in \N$ as in Proposition~\ref{lem:ELperutrbed}. 
\end{theorem}

\begin{remark}
\label{remark:uniform_second_order_estimates}
We notice that, via the estimates~\eqref{v_j_estimate_1} and~\eqref{v_j_estimate_n} in the construction of the Ekeland-type approximation sequence $(v_j)_{j \in \N}$, all terms on the right-hand side of the estimate~\eqref{eq:dualauxest} of Theorem~\ref{thm:uniformSecondOrder} are bounded uniformly in $j \in \N$. If, in addition, the condition~\eqref{eq:MuEllipticity} of $\mu$-ellipticity is assumed, we deduce from the formula~\eqref{eq:nabla_2_f_j} for $\nabla^2 f_j(z)$ in particular the uniform weighted second order estimate
\begin{align*}
 \lefteqn{\int_{\ball_{2r}(x_{0})} \varrho^4 (1+\abs{\eps(v_j)}^2)^{-\frac{\mu}{2}} |\nabla \eps(v_j)|^2 \dx{x} + \frac{1}{2A_{j}j^{2}}  \int_{\ball_{2r}(x_{0})} \varrho^4 (1+\abs{\eps(v_j)}^2)^{\frac{n-1}{2}} |\nabla \eps(v_j)|^2 \dx{x}}  \\
  & \leq \frac{c(n,\Gamma,\lambda,\Lambda,M)}{r^{2}}  \bigg[ \int_{\ball_{2r}(x_{0})} \abs{\eps(v_j)} \dx{x} + \Big( \frac{1}{r} +\frac{1}{r^2} + \frac{r^2}{j} \Big) \mathscr{L}^{n}(\ball_{2r}(x_0)) \\ 
  & \hspace{7cm} + \frac{1}{A_j j^2}\int_{\ball_{2r}(x_{0})}(1+ \abs{\eps(v_j)}^2)^{\frac{n+1}{2}} \dx{x} + \frac{1}{j^2} \frac{1}{r} \bigg] \\
  & \leq \frac{c(n,\gamma,\Gamma,\lambda,\Lambda,M)}{r^{2}}\bigg[\inf_{\mathscr{D}_{u_{0}}}F[-;\Omega] + \Big( 1 + \frac{r^2}{j} \Big) \mathscr{L}^{n}(\Omega)+ 1 +  \frac{1}{j^{2}} \Big( 1 + \frac{1}{r} \Big) \bigg].
\end{align*}
We further observe that the seemingly unnatural scaling in~$r$ for the right-hand side vanishes in the limit $j \to \infty$.
\end{remark} 

In the subsequent proof of Theorem~\ref{thm:uniformSecondOrder} we will use several times estimates of the same type. To shorten the arguments within that proof we combine them in the following technical lemma. 

\begin{lemma}\label{lem:TechnicalEstimate}
	 Let $f\in\C^{2}(\rsym)$ satisfy, for some $\Lambda\in (0,\infty)$, the bound~\eqref{eq:generalisedaelliptic}. Then, for vectors $w, e \in \R^n$, a matrix $M \in \R^{n\times n}$ and every $z\in\rsym$ there holds 
	\begin{align*}
		\langle \nabla^2f_j(z)((M w) \odot e), ((Mw)\odot e)\rangle 
		\leq \abs{M}^2 & \bigg[\Lambda \frac{\abs{w}^2 \abs{e}^2}{(1+\abs{z}^2)^{\frac{1}{2}}}\bigg. \\ & \bigg. \;\;\;\;\;\;\;+
		\frac{c(n)}{A_j j^2}  
		\left((1+\abs{z}^2)^{\frac{n+1}{2}} +\abs{w}^{n+1} \abs{e}^{n+1}\right)\bigg].
	\end{align*}
\end{lemma}

\begin{proof} 
This follows from a direct computation, using the special form of the perturbed integrand~$f_j$ from~\eqref{eq:AjdefLinfty}. Indeed, via the formula for~$\nabla^2 f_j$ from~\eqref{eq:nabla_2_f_j} and the upper bound~\eqref{eq:generalisedaelliptic} we have 
	\begin{equation*}
		\langle \nabla^2 f_j(z)(((M w) \odot e), ((Mw)\odot e)\rangle 
		\leq \Lambda \abs{M}^2 \frac{\abs{w}^2 \abs{e}^2}{(1+\abs{z}^2)^{\frac{1}{2}}}  
		+ \frac{c(n)}{A_j j^2}(1+\abs{z}^2)^{\frac{n-1}{2}}\abs{M}^2 \abs{w}^2 \abs{e}^2 .
	\end{equation*}
	The claim of the lemma now follows from Young's inequality with exponents $\tfrac{n+1}{n-1}$ and $\tfrac{n+1}{2}$.
\end{proof}
We now come to the proof of Theorem~\ref{thm:uniformSecondOrder}: 

\begin{proof}[Proof of Theorem~\ref{thm:uniformSecondOrder}] 
	We structure the proof into several steps. 
	
	\emph{Step 1: Preliminary estimate.}
	We observe that the product formula yields
	\begin{equation*}
	   \varrho^4 \del_k \eps(v_j) = \varrho^4 \eps(\del_k v_j) = - \nabla \varrho^4 \odot \del_k v_j + \eps (\varrho^4 \del_k v_j)
	\end{equation*}
	for $k \in \{1,\ldots,n\}$. This consequently allows to rewrite the left-hand side of~\eqref{eq:dualauxest} as
	 \begin{align}
     \lefteqn{\sum_{k=1}^n \int_{\Omega} \big\langle \nabla^2f_j(\eps(v_j)) \varrho^2 \del_k\eps(v_j),  \varrho^2  \del_k\eps(v_j) \big\rangle\dx{x} } \nonumber \\
	 & = \sum_{k,i,m=1}^n \int_\Omega \del_k \sigma_j^{(im)} \varrho^4 \del_k \eps^{(im)}(v_j)\dx{x} \notag \\
	 & = - \frac{1}{2} \sum_{k,i,m=1}^n \int_\Omega \del_k\sigma_j^{(im)}   \left[ \del_i\varrho^4 \del_k v_j^{(m)} + \del_m\varrho^4 \del_k v_j^{(i)}\right]\dx{x} \notag \\
	 & \qquad + \sum_{k,i,m=1}^n \int_\Omega \del_k\sigma_j^{(im)} \eps^{(im)}(\varrho^4\del_k v_j)\dx{x} \notag \\
	 \label{eq:FirstEstimateUniformSecondOrder}
     &\eqqcolon A+B. 
    \end{align}
	Notice that all integrals are well-defined, by the explicit form of $\del_k \sigma_j$ and the (weighted) second order estimate~\eqref{eq:nonunifreg2} for~$v_j$ from Lemma~\ref{lem:higherregularityapproximate}. 
   
\emph{Step 2: Estimate for~$A$.}
We observe that $A$ can be decomposed into the three parts
	\begin{align*}
		A 
		& = - \frac{1}{2}\sum_{k,i,m=1}^n \int_\Omega \del_k\sigma_j^{(im)}\left[ \del_i\varrho^4 \del_k v_j^{(m)} + \del_i\varrho^4 \del_m v_j^{(k)}\right]\dx{x} \\
		& \qquad + \frac{1}{2} \sum_{k,i,m=1}^n \int_\Omega \del_k\sigma_j^{(im)}\left[\del_i\varrho^4\, \del_m v_j^{(k)} + \del_m\varrho^4 \del_i v_j^{(k)}\right]\dx{x} \\
		& \qquad - \frac{1}{2} \sum_{k,i,m=1}^n \int_\Omega \del_k \sigma_j^{(im)}\left[\del_m\varrho^4 \del_i v_j^{(k)} + \del_m \varrho^4 \del_k v_j^{(i)}\right]\dx{x} \\
		& \eqqcolon \mathrm{I} + \mathrm{II} + \mathrm{III}.
		\end{align*}
\emph{On $\mathrm{I}$ and $\mathrm{III}$.} Using the symmetry of $\sigma_j$, i.e.\ $\sigma_j(x)\in \rsym$ for all $x\in\Omega$, we see that $\mathrm{I}=\mathrm{III}$. For future purposes, we rewrite the term in square brackets in term~$\mathrm{I}$ as 
	\begin{equation*}
	  \frac{1}{2} \left[ \del_i\varrho^4 \del_k v_j^{(m)} + \del_i\varrho^4 \del_m v_j^{(k)} \right] = 4 \varrho^3 \big( \nabla \varrho \otimes (\eps(v_j) e_k) \big)^{(im)} \qquad \text{for all } k,i,m \in \{1,\ldots,n\}.
	\end{equation*}
Using once again the symmetry of~$\sigma_j$ together with the orthogonal sum decomposition~\eqref{eq:orthsum}, the definition of~$\sigma_j$ and the Cauchy--Schwarz inequality for the bilinear form associated with $\nabla^2f_j(\eps(v_j))$, we find
	\begin{align*}
		\abs{\mathrm{I} + \mathrm{III}}
		\leq 2\abs{\mathrm{I}} 
		& = 8\bigg|\sum_{k=1}^n\int_\Omega \langle \del_k\sigma_j,\varrho^3 \nabla \varrho \otimes (\eps(v_j) e_k) \rangle \dx{x}\bigg| \\
		& = 8\bigg|\sum_{k=1}^n\int_\Omega \langle \del_k\sigma_j,\varrho^3 \nabla \varrho \odot (\eps(v_j) e_k) \rangle \dx{x}\bigg| \\
		& = 8\bigg|\sum_{k=1}^n\int_\Omega \big\langle \nabla^2f_j(\eps(v_j)) \varrho^2 \del_k \eps(v_j), \varrho \nabla \varrho \odot (\eps(v_j) e_k) \big\rangle \dx{x}\bigg| \\
		& \leq \frac{1}{4} \sum_{k=1}^n \int_\Omega \big\langle \nabla^2f_j(\eps(v_j))\varrho^2 \del_k\eps(v_j), \varrho^2 \del_k \eps(v_j)\big\rangle \dx{x} \\
		& \qquad + 64 \sum_{k=1}^n \int_\Omega \big\langle \nabla^2f_j(\eps(v_j)) \varrho \nabla \varrho \odot (\eps(v_j) e_k) ,\varrho \nabla \varrho \odot (\eps(v_j) e_k) \big\rangle \dx{x}.
	\end{align*}
	We note that, even though $\sigma_{j}$ and hence also $\del_k \sigma_j$ take values in $\rsym$, the inner product in the first line is well-defined; see~\eqref{eq:orthsum}.
	The first term on the right-hand side can be absorbed into the left-hand side of~\eqref{eq:FirstEstimateUniformSecondOrder}. The second term is bounded due to the estimate $|\nabla \varrho| \leq \tfrac{2}{r}$ and Lemma~\ref{lem:TechnicalEstimate}, applied with $M = r^{-1} \varrho \1_{(n \times n) \times (n \times n)}$, $w = r \nabla \varrho$ and $e = \eps(v_j) e_k$ for $k \in \{1,\ldots,n\}$, via 
	\begin{align}\label{eq:helpfulissimo1}
	\begin{split}
		64 \sum_{k=1}^n \int_\Omega \big\langle \nabla^2f_j(\eps(v_j))& \varrho \nabla \varrho \odot (\eps(v_j) e_k) ,\varrho \nabla \varrho \odot (\eps(v_j) e_k) \big\rangle \dx{x}\\
		&\leq \frac{c \Lambda}{r^2}\int_{\ball_{2r}(x_{0})}\abs{\eps(v_j)}\dx{x} + \frac{c(n)}{A_j j^2 r^2}\int_{\ball_{2r}(x_{0})}(1+\abs{\eps(v_j)}^2)^{\frac{n+1}{2}}\dx{x}.
		\end{split}
	\end{align}
\emph{On $\mathrm{II}$.} Again employing the symmetry of $\sigma_j$, we have
	\begin{align*}
		\mathrm{II} = \sum_{k,i,m=1}^n \int_\Omega  \del_k\sigma_j^{(im)}\del_i\varrho^4 \del_m v_j^{(k)} \dx{x}.
	\end{align*}
	Since the derivative $\del_m v_j^{(k)}$ appearing in the integral is not estimated in terms of the symmetric gradient~$\eps(v_j)$, we integrate by parts twice (where all computations are justified due to the regularity estimate for~$v_j$ in Lemma~\ref{lem:higherregularityapproximate}). In this way, we obtain
	\begin{align*}
	\mathrm{II} 
	&= -\sum_{k,i,m=1}^n \int_\Omega \sigma_j^{(im)} \left[\del_{ik}\varrho^4 \del_m v_j^{(k)} + \del_i\varrho^4\del_{mk}v_j^{(k)}\right]\dx{x}\\
	& = \sum_{k,i,m=1}^n \int_\Omega \Big(\del_m \left[\sigma_j^{(im)}\del_{ik}\varrho^4\right] v_j^{(k)} + \del_m\left[\sigma_j^{(im)}\del_{i}\varrho^4\right]\del_k v_j^{(k)} \Big)\dx{x}\\
	&=\sum_{k,i,m=1}^n\int_\Omega \Big(\del_m\sigma_j^{(im)} \del_{ik}\varrho^4 v_j^{(k)} + \sigma_j^{(im)} \del_{ikm}\varrho^4 v_j^{(k)} \\
	& \hspace{6cm} + \del_m\sigma_j^{(im)}\del_i\varrho^4\del_kv_j^{(k)} + \sigma_j^{(im)} \del_{im}\varrho^4 \del_kv_j^{(k)}\Big)\dx{x} \\
	&\eqqcolon \mathrm{II}_1+\ldots+\mathrm{II}_4,
	\end{align*}
	where now the only derivatives of~$v_j$ appearing in $\mathrm{II}_1, \ldots, \mathrm{II}_4$ are of the form $\del_kv_j^{(k)}$, which after summation in~$k \in \{1,\ldots,n\}$ gives the divergence of~$v_j$ and is hence estimated by~$|\eps(v_j)|$.
	
	\begin{itemize}
	\item  
		\emph{On $\mathrm{II}_1$.} We first express the integrand of $\mathrm{II}_1$ (summed over $i,k,m \in \{1,\ldots ,n\}$) via the symmetry of~$\sigma_j$ as 
		\begin{align*}
		 \sum_{k,i,m=1}^n \del_m\sigma_j^{(im)} \del_{ik}\varrho^4 v_j^{(k)} 
		  & = \sum_{k,i,m,\ell=1}^n \del_\ell \sigma_j^{(im)} \del_{ik}\varrho^4 v_j^{(k)} \delta_{\ell m} \\
		  & = \sum_{\ell=1}^n  \langle  \del_\ell \sigma_j, (\nabla^2 \varrho^4 v_j) \otimes e_\ell \rangle \\
		  & = \sum_{\ell=1}^n  \langle  \del_\ell \sigma_j, (\nabla^2 \varrho^4 v_j) \odot e_\ell \rangle \\
		  & = \sum_{\ell=1}^n  \big\langle \nabla^2 f_j(\eps(v_j)) \partial_\ell \eps(v_j), (\nabla^2 \varrho^4 v_j) \odot e_\ell \big\rangle.
		\end{align*}
        By the product rule we next observe 
		\[
		\nabla^2 \varrho^4 
			= \nabla \left(4\varrho^3\nabla \varrho\right) 
			= 12 \varrho^2 \nabla \varrho \otimes \nabla \varrho + 4\varrho^3 \nabla^2 \varrho,
		\]
		and via the Cauchy--Schwarz inequality for the associated bilinear form of $\nabla^2f_j(\eps(v_j))$ we then find for $\mathrm{II}_1$ the estimate
		\begin{align}
		\abs{\mathrm{II}_1} 
		&= 12 \bigg| \sum_{\ell=1}^n  \int_\Omega \big\langle \nabla^2 f_j(\eps(v_j)) \varrho^2 \partial_\ell \eps(v_j), (\nabla \varrho \cdot v_j) \nabla \varrho \odot e_\ell \big\rangle \dx{x} \nonumber \\
		& \qquad + 4 \sum_{\ell=1}^n \int_\Omega \big\langle \nabla^2 f_j(\eps(v_j)) \varrho^2 \partial_\ell \eps(v_j), \varrho (\nabla^2 \varrho v_j) \odot e_\ell \big\rangle \dx{x} \bigg| \nonumber \\
		& \leq \frac{1}{4} \sum_{\ell=1}^n \int_\Omega \big\langle \nabla^2f_j(\eps(v_j))\varrho^2 \del_\ell\eps(v_j), \varrho^2 \del_\ell \eps(v_j) \big\rangle \dx{x} \label{eq:pollute} \\
		& \qquad + 160 \sum_{\ell=1}^n \int_\Omega  \big\langle \nabla^2 f_j(\eps(v_j)) (\nabla \varrho \cdot v_j) \nabla \varrho \odot e_\ell, (\nabla \varrho \cdot v_j) \nabla \varrho \odot e_\ell \big\rangle \dx{x} \nonumber \\
		& \qquad + 160 \sum_{\ell=1}^n \int_\Omega \big\langle \nabla^2 f_j(\eps(v_j)) \varrho (\nabla^2 \varrho v_j) \odot e_\ell , \varrho (\nabla^2 \varrho v_j) \odot e_\ell \big\rangle \dx{x}. \nonumber
		\end{align}
		For the second and the third integral on the right-hand side, we apply Lemma~\ref{lem:TechnicalEstimate}, for the choices $M = r \nabla \varrho \otimes \nabla \varrho$ and $M = r \varrho  \nabla^2\varrho$, $w= r^{-1} v_j$ and $e = e_\ell$ for $\ell \in \{1,\ldots,n\}$. With $|M| \leq \tfrac{4}{r}$ (and dropping the factor $(1+\abs{z}^2)^{-1/2}$ in the first term) we thus arrive at 
		\begin{multline*}
		\hspace{1cm} \abs{\mathrm{II}_1} 
		\leq \frac{1}{4} \sum_{\ell=1}^n \int_\Omega \big\langle \nabla^2f_j(\eps(v_j))\varrho^2 \del_\ell\eps(v_j), \varrho^2 \del_\ell \eps(v_j) \big\rangle \dx{x} \\
		+ \frac{c(n,\Lambda)}{r^2} \int_{\ball_{2r}(x_{0})} \frac{\abs{v_j}^2}{r^2}\dx{x} + \frac{c(n)}{ A_j j^2 r^2} \int_{\ball_{2r}(x_{0})} \bigg( (1+\abs{\eps(v_j)}^2)^{\frac{n+1}{2}} + \frac{\abs{v_j}^{n+1}}{r^{n+1}} \bigg) \dx{x}.
		\end{multline*}
		
		\item \emph{On $\mathrm{II}_3$.} We can initially proceed exactly as for~$\mathrm{II}_1$ and estimate by the Cauchy--Schwarz inequality 
		\begin{align*}
			\abs{\mathrm{II}_3} & = \bigg| \sum_{\ell=1}^n \int_\Omega \langle  \del_\ell \sigma_j, (\nabla \varrho^4 \Div(v_j)) \odot e_\ell \rangle \dx{x} \bigg| \\
            & = 4 \bigg| \sum_{\ell=1}^n  \int_\Omega \big\langle \nabla^2 f_j(\eps(v_j)) \varrho^2 \partial_\ell \eps(v_j), \varrho (\nabla \varrho \Div(v_j)) \odot e_\ell \big\rangle \dx{x} \bigg| \\
            & \leq \frac{1}{4} \sum_{\ell=1}^n \int_\Omega \big\langle \nabla^2f_j(\eps(v_j))\varrho^2 \del_\ell\eps(v_j), \varrho^2 \del_\ell \eps(v_j) \big\rangle \dx{x} \\ 
            & \qquad + 16 \sum_{\ell=1}^n \int_\Omega \big\langle \nabla^2 f_j(\eps(v_j)) \varrho (\nabla \varrho \Div(v_j)) \odot e_\ell, \varrho (\nabla \varrho \Div(v_j)) \odot e_\ell \big\rangle \dx{x} .
		\end{align*}
		We again use Lemma~\ref{lem:TechnicalEstimate}, with $M = r^{-1} \1_{n \times n}$, $w = r \varrho \nabla \varrho \Div(v_j)$ and $e = e_\ell$ for $\ell \in \{1,\ldots,n\}$. With $\abs{\Div(v_j)}\leq \abs{\eps(v_j)}$ this gives
		\begin{align*}
			\abs{\mathrm{II}_3}
            & \leq \frac{1}{4} \sum_{\ell=1}^n \int_\Omega \big\langle \nabla^2f_j(\eps(v_j))\varrho^2 \del_\ell\eps(v_j), \varrho^2 \del_\ell \eps(v_j) \big\rangle \dx{x} \\ 
            & \qquad + \frac{c(n,\Lambda)}{r^2} \int_{\ball_{2r}(x_{0})}\abs{\eps(v_j)}\dx{x} + \frac{c(n)}{ A_j j^2 r^2} \int_{\ball_{2r}(x_{0})} (1+\abs{\eps(v_j)}^2)^{\frac{n+1}{2}} \dx{x}.
		\end{align*}
		
		\item \emph{On $\mathrm{II}_{2}$.} With the formula~\eqref{eq:nabla_f_j} for $\nabla f_j$, the bound~$\Gamma$ for the Lipschitz constant for $\nabla f$ in view of Lemma~\ref{lem:boundbelow}\ref{item:boundbelow2}, $|\nabla^3 \varrho| \leq \tfrac{8}{r^3}$ and Young's inequality with exponents $\tfrac{n+1}{n}$ and~$n+1$ we obtain
		\begin{align*}
			\abs{\mathrm{II}_2} 
			&= \bigg| \sum_{k,i,m=1}^n \int_\Omega \nabla f_j(\eps(v_j))^{(im)}(\del_{ikm}\varrho^4) v_j^{(k)}\dx{x}\bigg|\\
			&\leq \frac{c(n,\Gamma)}{r^3}\int_{\ball_{2r}(x_{0})}\abs{v_j}\dx{x}
			+
			\frac{c(n)}{A_j j^2 r^3}\int_{\ball_{2r}(x_{0})}(1+|\eps(v_j)|^2)^{\frac{n}{2}} \abs{v_j}\dx{x} \\
			& \leq \frac{c(n,\Gamma)}{r^2}\int_{\ball_{2r}(x_{0})} \frac{\abs{v_j}}{r}\dx{x}
			+
			\frac{c(n)}{A_j j^2 r^2}\int_{\ball_{2r}(x_{0})}\bigg( (1+\abs{\eps(v_j)}^2)^{\frac{n+1}{2}} + \frac{\abs{v_j}^{n+1}}{r^{n+1}} \bigg) \dx{x} .
		\end{align*}
		
		\item \emph{On $\mathrm{II}_{4}$.} We proceed similarly as in the estimation of $\mathrm{II}_2$, now additionally using the pointwise estimate $\abs{\Div(v_j)}\leq \abs{\eps(v_j)}$, and Young's inequality with exponents $\frac{n+1}{n-1}$ and $\frac{n+1}{2}$. In this way, we find
		\begin{align*}
		\abs{\mathrm{II}_4} 
		&= \bigg| \sum_{k,i,m=1}^n \int_\Omega \nabla f_j(\eps(v_j))^{(im)}  \del_{im}\varrho^4 \del_k 
		v_j^{(k)}\dx{x}\bigg|\\
		&\leq \frac{c(n,\Gamma)}{r^2}\int_{\ball_{2r}(x_{0})}\abs{\Div(v_j)}\dx{x} + \frac{c(n)}{A_j j^2 r^2} \int_{\ball_{2r}(x_{0})}(1+\abs{\eps(v_j)}^2)^{\frac{n}{2}}\abs{\Div(v_j)}\dx{x}\\
		&\leq   \frac{c(n,\Gamma)}{r^2}\int_{\ball_{2r}(x_{0})}\abs{\eps(v_j)}\dx{x} 
		+ 
		\frac{c(n)}{A_j j^2 r^2}\int_{\ball_{2r}(x_{0})}(1+\abs{\eps(v_j)}^2)^{\frac{n+1}{2}}\dx{x}.
		\end{align*}
	\end{itemize}Combining all the estimates from Step 2, we arrive at
	\begin{align}\label{eq:centralAestimate}
	\begin{split}
	  A & \leq \frac{3}{4} \sum_{\ell=1}^n \int_\Omega \big\langle \nabla^2f_j(\eps(v_j))\varrho^2 \del_\ell\eps(v_j), \varrho^2 \del_\ell \eps(v_j) \big\rangle \dx{x} \\
	  & \qquad + \frac{c(n,\Gamma,\Lambda)}{r^2}\int_{\ball_{2r}(x_{0})}\bigg( \abs{\eps(v_j)} + \frac{\abs{v_j}}{r} + \frac{\abs{v_j}^2}{r^2}\bigg) \dx{x} \\
	  & \qquad + \frac{c(n)}{A_j j^2 r^2}\int_{\ball_{2r}(x_{0})}\bigg( (1+\abs{\eps(v_j)}^2)^{\frac{n+1}{2}} + \frac{\abs{v_j}^{n+1}}{r^{n+1}} \bigg)\dx{x}.
	  \end{split}
	\end{align}
\emph{Step 3: Estimating $B$.}	
	The estimate for the second term~$B$, which is given as
    \begin{equation*}
		B = \sum_{k=1}^n \int_\Omega \langle \del_k \sigma_j, \eps(\varrho^4\del_k v_j)\rangle\dx{x},
	\end{equation*}
	can be obtained by exploiting the differentiated Euler--Lagrange inequality~\eqref{eq:betterEL2} as stated in Lemma~\ref{lem:ELbetter}, with the test function $\varphi = \varrho^4\del_k v_j$. Here, we briefly pause to justify that this choice of $\varphi$ (even though it does not belong to $\W^{1,n+1}(\Omega;\R^{n})$) is admissible in~\eqref{eq:betterEL2} for every fixed $j \geq j_0$. Since we have $v_{j}\in (\W^{1,n+1}\cap\W_{\loc}^{2,2})(\Omega;\R^{n})$ by construction and Lemma~\ref{lem:higherregularityapproximate}, we can choose a sequence $(h_{i})_{i\in\mathds{N}}$ in $ \R_{>0}$ with $h_{i}\searrow 0$ such that for $k \in \{1,\ldots,n\}$ we have
	\begin{alignat}{2}
	\label{eq:makeitworkconvergences1}
	\varrho^{4}\Delta_{k,h_{i}}v_{j} & \to \varrho^{4}\partial_{k}v_{j}\qquad&&\text{strongly in }\L^{n+1}(\Omega;\R^{n}),\\ 
	\label{eq:makeitworkconvergences2}
	\eps(\varrho^{4}\Delta_{k,h_{i}}v_{j}) & \to \eps(\varrho^{4}\partial_{k}v_{j})\qquad &&\text{$\mathscr{L}^{n}$-a.e. in $\Omega$},
	\end{alignat}
	as $i\to\infty$. We consider the (weighted) Hilbert space $\mathcal{H}\coloneqq\L_{\mu_{j}}^{2}(\ball_{2r}(x_{0});\rsym)$, where 
	\begin{align*}
		\mu_{j} \coloneqq (1+|\eps(v_{j})|^{2})^{\frac{n-1}{2}}\mathscr{L}^{n}\Lcorner\ball_{2r}(x_{0}). 
	\end{align*}
	Since $(\eps(\varrho^{4}\Delta_{k,h_{i}}v_{j}))_{i \in \N}$ is uniformly bounded in $\mathcal{H}$ (cf.~\eqref{eq:diff-quot-bounded-H}), a (non-relabeled) subsequence converges weakly in $\mathcal{H}$, and its limit can be identified by the pointwise convergence $\eqref{eq:makeitworkconvergences2}$ in combination with Lemma~\ref{lem:identification} as $\eps(\varrho^{4}\partial_{k}v_{j})$ (which, in view of~\eqref{eq:nonunifreg2}, also belongs to~$\mathcal{H}$). By the estimate $|\nabla^2f_{j}(z)|\leq C(n,\Lambda) (1+|z|^{2})^{\frac{n-1}{2}}$ for all $z\in\rsym$ from~\eqref{eq:nabla_2_f_j} combined with~\eqref{eq:nonunifreg2} and by the Cauchy--Schwarz inequality, we notice that the linear functional
	\begin{equation*}
	\Psi_k\colon \mathcal{H}\ni \psi \mapsto \int_{\ball_{2r}(x_{0})} \langle \del_k \sigma_j, \psi \rangle\dx{x} = \int_{\ball_{2r}(x_{0})}\big\langle \nabla^2f_{j}(\eps(v_{j}))\partial_{k}\eps(v_{j}),\psi \big\rangle\dx{x}\in \R
	\end{equation*}
	is well-defined and belongs to the dual space $\mathcal{H}'$.
	Therefore, we have that $\Psi_k(\eps(\varrho^{4}\Delta_{k,h_{i}}v_{j}))\to\Psi_k(\eps(\varrho^{4}\partial_{k}v_{j}))$, i.e.,
	\begin{equation*}
	  \int_\Omega \langle \del_k \sigma_j, \eps(\varrho^{4}\Delta_{k,h_{i}}v_{j})\rangle\dx{x} \to \int_\Omega \langle \del_k \sigma_j, \eps(\varrho^4\del_k v_j)\rangle\dx{x} \quad \text{as } i\to\infty.
	\end{equation*}
	Concerning the second term in the Euler--Lagrange inequality~\eqref{eq:betterEL2} we notice that $|\nabla^2g(v_{j}/M)|$ is uniformly bounded in $\ball_{2r}(x_0)$ by the strict inequality~\eqref{v_j_estimate_bounded_strict}  together with $g\in \C^2(\ball_2(0))$, and with $\partial_k v_j \in \L^{n+1}(\Omega;\R^{n})\subset\L^{(n+1)/n}(\Omega;\R^{n})$ and the convergence~\eqref{eq:makeitworkconvergences1}, we then infer 
    \begin{equation*}
      \int_{\Omega}\Big\langle \partial_{k}\Big(\nabla g\Big(\frac{v_{j}}{M} \Big)\Big) ,\frac{\varrho^{4}\Delta_{k,h_{i}}v_{j}}{M}\Big\rangle\dx{x} \to \int_{\Omega}\Big\langle \partial_{k}\Big(\nabla g\Big(\frac{v_{j}}{M} \Big)\Big)  ,\frac{\varrho^{4}\partial_{k}v_{j}}{M}\Big\rangle\dx{x}  \quad \text{as } i\to\infty.
    \end{equation*}  
	Noticing that the limit on the right-hand side is non-negative by the strict convexity of~$g$ on $\B_{2}(0)$, we can now take advantage of the Euler--Lagrange inequality~\eqref{eq:betterEL2}, applied with the test function $\varphi = \varrho^{4}\Delta_{k,h_{i}}v_{j}\in\W_{c}^{1,n+1}(\Omega;\R^{n})$ for every $k \in \{1,\ldots,n\}$. With the limit of the two integrals as established above, this yields
	\begin{align*}
		B & \leq \sum_{k=1}^n \bigg( \int_\Omega \langle \del_k \sigma_j, \eps(\varrho^4\del_k v_j)\rangle\dx{x} + \int_{\Omega}\Big\langle \partial_{k}\Big(\nabla g\Big(\frac{v_{j}}{M} \Big)\Big)  ,\frac{\varrho^{4}\partial_{k}v_{j}}{M}\Big\rangle\dx{x} \bigg) \\
        & = \lim_{i \to \infty} \sum_{k=1}^n   \int_\Omega \langle \del_k \sigma_j, \eps(\varrho^{4}\Delta_{k,h_{i}}v_{j})\rangle\dx{x} + \int_{\Omega}\Big\langle \partial_{k}\Big(\nabla g\Big(\frac{v_{j}}{M} \Big)\Big) ,\frac{\varrho^{4}\Delta_{k,h_{i}}v_{j}}{M}\Big\rangle\dx{x} \\
        & \leq  \frac{1}{j} \lim_{i \to \infty} \sum_{k=1}^n \|\varrho^4 \Delta_{k,h_{i}}v_j\|_{\W^{-1,1}(\Omega; \R^n)}.
	\end{align*}
	With the convergence~\eqref{eq:makeitworkconvergences1} combined with the embedding $\lebe^{n+1}(\Omega;\R^n)\hookrightarrow \W^{-1,1}(\Omega;\R^n)$, the pro\-duct rule and the inequalities in~\eqref{eq:neg_Sob_1} and~\eqref{eq:neg_Sob_2}, we obtain for the left-hand side of the previous inequality
	\begin{equation*}
	   \lim_{i \to \infty}\|\varrho^4 \Delta_{k,h_{i}}v_j\|_{\W^{-1,1}(\Omega;\R^n)} \leq \|\varrho^4 \del_k v_j\|_{\W^{-1,1}(\Omega; \R^n)} \leq \Big( 1+ \frac{4}{r} \Big) \|v_j\|_{\L^1(\ball_{2r}(x_{0});\R^n)}
	\end{equation*}
    for $k \in \{1,\ldots,n\}$. Therefore, we end up with
    \begin{equation}\label{eq:centralBestimate}
      B \leq \frac{c(n)}{j} \Big( 1+ \frac{1}{r} \Big) \|v_j\|_{\L^1(\ball_{2r}(x_{0});\R^n)}.
    \end{equation}
	
	\emph{Step 4: Conclusion.} We now return to the estimate~\eqref{eq:FirstEstimateUniformSecondOrder} from Step~1. Taking advantage of the estimate~\eqref{eq:centralAestimate} for~$A$ from Step~2, where we can absorb the first term on the left-hand side of~\eqref{eq:FirstEstimateUniformSecondOrder}, and of the estimate~\eqref{eq:centralBestimate} for~$B$ from Step~3, we arrive at 
	\begin{align*}
	 \lefteqn{\sum_{k=1}^n \int_{\ball_{2r}(x_{0})} \varrho^4 \big\langle \nabla^2f_j(\eps(v_j))\del_k\eps(v_j), \del_k\eps(v_j) \big\rangle\dx{x}} \\
		& \leq \frac{c(n,\Gamma,\Lambda)}{r^2}\int_{\ball_{2r}(x_{0})}\bigg( \abs{\eps(v_j)} + \frac{\abs{v_j}}{r} + \frac{\abs{v_j}^2}{r^2} + \frac{r^2 \abs{v_j}}{j} +   \frac{r \abs{v_j}}{j}  \bigg) \dx{x} \\
	  & \qquad + \frac{c(n)}{A_j j^2 r^2}\int_{\ball_{2r}(x_{0})}\bigg( (1+\abs{\eps(v_j)}^2)^{\frac{n+1}{2}} + \frac{\abs{v_j}^{n+1}}{r^{n+1}} \bigg)\dx{x},
	\end{align*}
    where we have also employed the fact that the localization function~$\varrho$ is supported only inside of the ball $\ball_{2r}(x_0)$. At this stage, we employ on the right-hand side the uniform bound~\eqref{v_j_estimate_bounded} on the $\lebe^\infty$-norm of~$v_j$. Taking into account the trivial inequality $\tfrac{r}{j} \leq \tfrac{1}{r} + \tfrac{r^2}{j}$, we then arrive at the estimate stated in the theorem, with the claimed dependence of the constant~$c$.
\end{proof}

We conclude this subsection by commenting on a detail of the structure of the above proof:
\begin{remark} \label{rem:strucproof}
For the estimation of the terms $\mathrm{II}_{1}$ and $\mathrm{II}_{3}$ in the above proof, it might seem more natural and conceptually easier to employ the non-differentiated Euler--Lagrange inequality~\eqref{eq:EL2} from Lemma~\ref{lem:Euler_Lagrange_inequality} rather than its differentiated analogue~\eqref{eq:betterEL2} from Lemma~\ref{lem:ELbetter}. Specifically, one might be inclined to write  
\begin{align}\label{eq:problematic1}
\mathrm{II}_{1} = \sum_{k,i,m=1}^{n}\int_{\Omega}\partial_{m}\sigma_{j}^{(im)}\partial_{ik}\varrho^{4}v_{j}^{(k)}\dif x = \int_{\Omega}\langle \mathrm{div}(\sigma_{j}),\tau_{j}\rangle\dif x = -\int_{\Omega}\langle\sigma_{j},\sg(\tau_{j})\rangle\dif x, 
\end{align}
where $\tau_{j}$ is defined in the obvious way. Then, if we aim to use~\eqref{eq:EL2}, we are bound to re-introduce the corresponding $\lebe^{\infty}$-penalization term, leading us to 
\begin{align}\label{eq:problematic2}
|\mathrm{II}_{1}| \leq \left\vert \int_{\Omega}\langle\sigma_{j},\sg(\tau_{j})\rangle\dif x + \int_{\Omega}\Big\langle \nabla g\Big(\frac{v_{j}}{M}\Big),\frac{\tau_{j}}{M}\Big\rangle \dif x  \right\vert + \left\vert \int_{\Omega}\Big\langle\nabla g\Big(\frac{v_{j}}{M}\Big),\frac{\tau_{j}}{M}\Big\rangle\dif x \right\vert = \mathrm{II}_{1}^{a} + \mathrm{II}_{1}^{b}. 
\end{align}
Whereas $\mathrm{II}_{1}^{a}$ is conveniently controlled by~\eqref{eq:EL2}, the available a-priori bounds do not exclude a potential blow-up of $\mathrm{II}_{1}^{b}$ as $j\to\infty$. This is so because the strict bound in~\eqref{v_j_estimate_bounded_strict} does not rule out the possibility that $\limsup_{j\to\infty}\|v_{j}\|_{\lebe^\infty(K;\R^n)}=2M$ holds on an open subset $K\Subset\Omega$, and in this case uniform $\lebe^{1}$-bounds on $\sg(\tau_{j})$ are rendered useless. In turn, this is the main reason for the more involved algebraic transformations employed in the above proof. This principal advantage however is only usable through the passage to the weighted second order estimates: Proceeding in this way, the difficulties inherent in \eqref{eq:problematic1}--\eqref{eq:problematic2} do not vanish completely but rather transfer to the appearance of second order quantities and two pollution terms in~\eqref{eq:pollute}. It is thus a key point that our approach in the above proof of Theorem~\ref{thm:uniformSecondOrder} results in terms which are either absorbable or conveniently controllable, which would not be the case for \eqref{eq:problematic1}--\eqref{eq:problematic2}. 
\end{remark}
\subsection{$\sobo_{\locc}^{1,\lebe\log\lebe}$-regularity}
\label{section:proof_main_theorem}
Based on the weighted second order estimates gathered in the previous paragraph, we now proceed to the 

\begin{proof}[Proof of Theorem~\ref{thm:main}]
We split the proof into four steps. Throughout, let $\varrho\in\hold_{c}^{\infty}(\Omega;[0,1])$ be a localization function with $\mathbbm{1}_{\ball_{r}(x_{0})}\leq \varrho\leq \mathbbm{1}_{\ball_{2r}(x_{0})}$ and $\abs{\nabla^s\varrho}\leq\left(\frac{2}{r}\right)^s$ for $s\in\{1,2,3\}$, as in the statement of Theorem~\ref{thm:uniformSecondOrder}. Furthermore, the inequality $\log(1+t) \leq t$ for $t\geq 0$ implies that there exists a constant $c>0$ such that 
\begin{equation}
\label{eq:logkey}
 \log(1+t) \leq c t^{\frac{1}{4}} \quad \text{for all } t \geq 0.
\end{equation}

\emph{Step 1: Integration by parts in the Euler--Lagrange inequality and choice of a suitable test function.} We first notice from $\nabla f(\eps(v_{j})) \in \W^{1,2}_{\loc}(\Omega;\rsym)$, see Lemma~\ref{lem:higherregularityapproximate}, and $\sigma_{j}\coloneqq \nabla f_{j}(\varepsilon(v_{j}))\in{\sobo}{_{\locc}^{1,(n+1)/n}}(\Omega;\rsym)$, see  Lemma~\ref{lem:ELbetter},  that
\begin{equation}\label{eq:toomanytempi}
 (\nabla f_{j} - \nabla f)(\eps(v_{j})) \in \W^{1,\frac{n+1}{n}}_{\loc}(\Omega;\rsym). 
\end{equation}
This allows to use the integration by parts formula in the first integral of the Euler--Lagrange inequality~\eqref{eq:EL2}. Thus, also employing the symmetry of $\nabla f_{j}$, we infer that 
\begin{multline}
\label{eq:ELpartint}
	\Bigg\vert \int_{\Omega} \big\langle \nabla f(\eps(v_{j})),\eps(\varphi) \big\rangle \dx{x} - \sum_{i,m=1}^n \int_{\Omega} \partial_m ((\nabla f_{j} - \nabla f)(\eps(v_{j}))^{(im)}) \varphi^{(i)} \dx{x} \\ + \int_{\Omega} \Big\langle \nabla g\Big(\frac{v_{j}}{M}\Big),\frac{\varphi}{M}\Big\rangle\dx{x}\Bigg\vert
	\leq \frac{1}{j}\|\varphi\|_{\W^{-2,1}(\Omega;\R^{n})}
\end{multline}
holds for all $\varphi\in\sobo_{c}^{1,n+1}(\Omega;\R^{n})$. Because of $\nabla f(\sg(v_{j}))\in\lebe^{\infty}(\Omega;\rsym)$, \eqref{eq:toomanytempi} and 
\begin{align*}
\left\Vert \nabla g\Big(\frac{v_{j}}{M}\Big)\right\Vert_{\lebe^{\infty}(\Omega;\R^{n})}<\infty, 
\end{align*}
see~\eqref{v_j_estimate_bounded_strict}, we find by approximation of general maps   $\varphi\in(\sobo^{1,1}_c \cap \lebe^{n+1})(\Omega;\R^{n})$ by $\hold_{c}^{\infty}(\Omega;\R^{n})$-maps in the norm topology of $(\sobo^{1,1}\cap\lebe^{n+1})(\Omega;\R^{n})$ that the inequality~\eqref{eq:ELpartint} holds true for all competitors $\varphi\in(\sobo^{1,1}_c \cap \lebe^{n+1})(\Omega;\R^{n})$. We note that the use of the integration by parts formula has allowed to weaken the regularity requirement for the test functions compared with the Euler--Lagrange inequality~\eqref{eq:EL2}, where $\varphi \in \sobo^{1,n+1}_c(\Omega;\R^{n})$ is required. This circumvents an additional approximation argument to justify our choice of a test function
\begin{equation}
\label{eq:log_estimate_test_function}
\varphi\coloneqq \varrho^{4} \log^{2}(1+\abs{\eps(v_j)}^2)v_{j},
\end{equation}
which obviously belongs to the space $(\sobo^{1,1}_c \cap \lebe^{n+1})(\Omega;\R^{n})$ and where, for convenience of notation, we have suppressed the explicit notation of the dependence of~$\varphi$ on~$j$. We further observe that 
\begin{align}\label{eq:symgradsplit}
\begin{split}
\sg(\varphi)  = \varrho^{4}\log^{2}(1+\abs{\eps(v_j)}^2)\sg(v_{j}) & + 4\varrho^3 \nabla \varrho \odot \log^{2}(1+\abs{\eps(v_j)}^2)v_{j}  \\ & + \varrho^{4} \nabla ( \log^{2}(1+\abs{\eps(v_j)}^2)) \odot v_{j} . 
\end{split}
\end{align}

\emph{Step 2: Uniform local higher integrability estimates.} 
With the test function~$\varphi$ from~\eqref{eq:log_estimate_test_function} in the previous step and by integration by parts, the  Euler--Lagrange inequality~\eqref{eq:ELpartint}  implies 
	\begin{align*}
		\mathrm{I}  
		&\coloneqq 
			\int_\Omega \varrho^4 \log^2(1+\abs{\eps(v_j)}^2) \big\langle \nabla f(\eps(v_j)), \eps(v_j) \big\rangle \dx{x} 
			\\
			&\leq 
			- \int_\Omega \big\langle \nabla f(\eps(v_j)),4\varrho^3 \nabla \varrho\odot \log^2(1+\abs{\eps(v_j)}^2)v_j \big\rangle \dx{x}\\
			& \qquad -
			\int_\Omega \big\langle \nabla f(\eps(v_j)), \varrho^4 \nabla (\log^2(1+\abs{\eps(v_j)}^2)) \odot v_j \big\rangle \dx{x} \\
			& \qquad + \sum_{i,m=1}^n \int_{\Omega} \partial_m ((\nabla f_{j} - \nabla f)(\eps(v_{j}))^{(im)}) \varrho^{4} \log^{2}(1+\abs{\eps(v_j)}^2)v_{j}^{(i)} \dx{x} \\
			& \qquad -
			\int_\Omega \varrho^4 \log^2(1+\abs{\eps(v_j)}^2)\Big\langle \nabla g\left(\frac{v_j}{M}\right),  \frac{v_j}{M} \Big\rangle \dx{x}\\ 
			& \qquad + 
			\frac{1}{j}\|\varrho^4 \log^2(1+\abs{\eps(v_j)}^2)v_j\|_{\W^{-2,1}(\Omega;\R^n)}\\
			&\eqqcolon \mathrm{II} + \mathrm{III} + \mathrm{IV} + \mathrm{V} + \mathrm{VI}. 
	\end{align*}
With Lemma~\ref{lem:boundbelow}\ref{item:boundbelow1} and~\eqref{eq:logkey}, we first observe 
    \begin{equation}
    \label{eq:finalIestimate}
	\mathrm{I} \geq \gamma \int_{\Omega}\varrho^{4}\abs{\eps(v_j)} \log^2(1+\abs{\eps(v_j)}^2)\dx{x} - c(\Gamma) \int_{\Omega} \varrho^4 |\sg(v_{j})| \dx{x}.
	\end{equation}
    In order to estimate next the terms $\mathrm{II}$ and $\mathrm{III}$, we recall from Lemma~\ref{lem:boundbelow}\ref{item:boundbelow2} that~$f$ is Lipschitz continuous with constant~$\Gamma$ and that $v_j$ satisfies the $\lebe^\infty$-bound from~\eqref{v_j_estimate_bounded}. 
	Using~\eqref{eq:logkey} we then obtain 
	\begin{equation*}
    \mathrm{II} \leq 4
		\int_{\Omega} \varrho^3 \abs{\nabla f(\eps(v_j))}\abs{\nabla \varrho} \log^2(1+\abs{\eps(v_j)}^2)\abs{v_j}\dx{x}\\
        \leq \frac{c(\Gamma,M)}{r} 
		\int_{\B_{2r}(x_{0})} \abs{\eps(v_j)}\dx{x}.
	\end{equation*}
	We next observe the estimate
	\begin{align*}
		\big| \nabla (\log^2(1+\abs{\eps(v_j)}^2)) \big| 
		\leq 
		4 (1+\abs{\eps(v_j)}^2)^{-1} \abs{\eps(v_j)} \abs{\nabla\eps(v_j)}\log(1+\abs{\eps(v_j)}^2).
	\end{align*}
	Via Young's inequality, we then find for term $\mathrm{III}$ 
	\begin{align*}
		\mathrm{III} & \leq 
		c(\Gamma,M) \int_{\Omega} \varrho^4 (1+\abs{\eps(v_j)}^2)^{-1} \abs{\eps(v_j)} \abs{\nabla\eps(v_j)} \log(1+\abs{\eps(v_j)}^2)\dx{x}\\
		& \leq \frac{\gamma}{2} \int_{\Omega}\varrho^{4}\abs{\eps(v_j)} \log^2(1+\abs{\eps(v_j)}^2)\dx{x} \\
		& \qquad + c(\gamma,\Gamma,M)  \int_{\Omega}\varrho^{4}(1+|\eps(v_{j})|^{2})^{-\frac{3}{2}} |\nabla \eps(v_j)|^2 \dx{x},
	\end{align*}
	and for term $\mathrm{IV}$ 
    \begin{align*}
		\mathrm{IV} & \leq 
		c(n,M) \frac{1}{2A_j j^2}\int_{\Omega} \varrho^4 (1+\abs{\eps(v_j)}^2)^{\frac{n-1}{2}}\abs{\nabla\eps(v_j)} \log^2(1+\abs{\eps(v_j)}^2)\dx{x} \\
		& \leq \frac{c(n,M)}{A_j j^2} \int_\Omega \varrho^4 (1+\abs{\eps(v_j)}^2)^{\frac{n-1}{2}}\log^4(1+\abs{\eps(v_j)}^2)\dx{x} \\
		&\qquad + \frac{c(n, M)}{A_j j^2} \int_\Omega \varrho^4 (1+\abs{\eps(v_j)}^2)^{\frac{n-1}{2}}\abs{\nabla\eps(v_j)}^2\dx{x}  .
	\end{align*}	
Under the ellipticity condition~\eqref{eq:MuEllipticity} with $\mu \leq 3$, we may now take advantage of the second order estimates from Theorem~\ref{thm:uniformSecondOrder} as stated in Remark~\ref{remark:uniform_second_order_estimates} in order to deal with the second terms on the right-hand side of the estimate for $\mathrm{III}$ and $\mathrm{IV}$. Using~\eqref{eq:logkey}, we find in this way
\begin{align*}
 \mathrm{III} + \mathrm{IV} & \leq \frac{\gamma}{2} \int_{\Omega}\varrho^{4}\abs{\eps(v_j)} \log^2(1+\abs{\eps(v_j)}^2)\dx{x} \\
		& \qquad + \frac{c(n,\gamma,\Gamma,\lambda,\Lambda,M)}{r^{2}}  \bigg[ \int_{\ball_{2r}(x_{0})} \abs{\eps(v_j)} \dx{x} + \Big( \frac{1}{r} +\frac{1}{r^2} + \frac{r^2}{j} \Big) \mathscr{L}^{n}(\ball_{2r}(x_0)) \\ 
  & \hspace{5cm} + \frac{1}{A_j j^2} \Big( 1 + r^2 \Big) \int_{\ball_{2r}(x_{0})}(1+ \abs{\eps(v_j)}^2)^{\frac{n+1}{2}} \dx{x} + \frac{1}{j^2} \frac{1}{r} \bigg].
\end{align*}
We next discuss term $\mathrm{V}$. Since $g|_{\B_{2}(0)}$ is finite, of class $\C^{2}$ and convex, we have the monotonicity inequality 
\begin{equation*}
0\leq \langle \nabla g(y_{1})-\nabla g(y_{2}),y_{1}-y_{2}\rangle<\infty \quad \text{for all } y_{1},y_{2}\in\B_2(0). 
\end{equation*}
We recall that $\nabla g(0)=0$ and, by~\eqref{v_j_estimate_bounded_strict}, $\|v_{j}\|_{\lebe^\infty(\Omega;\R^n)}<2M$. Hence, applying the monotonicity inequality pointwisely with $y_{1}=v_{j}(x)$ and $y_{2}=0$ yields $\langle \nabla g(\tfrac{v_{j}}{M}),\tfrac{v_{j}}{M}\rangle \geq 0$ $\mathscr{L}^{n}$-a.e. in $\Omega$. Therefore, we have
	\begin{equation*}
	\mathrm{V} 
	= - \int_{\Omega}\varrho^{4}\log^2(1+|\eps(v_{j})|^{2})\Big\langle \nabla g\Big(\frac{v_{j}}{M}\Big),\frac{v_{j}}{M}\Big\rangle\dx{x}\leq 0. 
	\end{equation*}
Finally, we can estimate term $\mathrm{VI}$ via~\eqref{eq:neg_Sob_0},~\eqref{eq:neg_Sob_1} and~\eqref{eq:logkey}
	\begin{equation*}
	\mathrm{VI} = \frac{1}{j}\|\varrho^4 \log^2(1+\abs{\eps(v_j)}^2)v_j\|_{\W^{-2,1}(\Omega;\R^n)} \leq \frac{c(M)}{j}\|\eps(v_j)\|_{\L^1(\ball_{2r}(x_{0}),\R^{n \times n})}.
	\end{equation*}
We are now ready to synthesise all estimates for the terms $\mathrm{I}$--$\mathrm{VI}$. Absorbing the first term on the right-hand side in the estimate for $\mathrm{III} + \mathrm{IV}$ in the lower bound~\eqref{eq:finalIestimate} for $\mathrm{I}$ and then using also~\eqref{v_j_estimate_n}, we obtain  
\begin{align}\label{eq:ultimate}
\lefteqn{\int_{\ball_{r}(x_{0})}|\sg(v_{j}) |\log^{2}(1+|\sg(v_{j})|^{2})\dif x} \nonumber \\
& \leq \int_{\Omega}\varrho^{4}|\sg(v_{j})  \log^2(1+\abs{\eps(v_j)}^2)\dx{x} \nonumber \\
& \leq c\Big(1+\frac{1}{r}+\frac{1}{r^{2}} + \frac{1}{j}\Big) \int_{\ball_{2r}(x_{0})}|\sg(v_{j})|\dif x  + \Big( \frac{1}{r^3} +\frac{1}{r^4} + \frac{1}{j} \Big) \mathscr{L}^{n}(\ball_{2r}(x_0)) \nonumber \\
& \quad +  \frac{c}{A_j j^2} \Big( 1 + \frac{1}{r} + \frac{1}{r^2} \Big) \int_{\ball_{2r}(x_{0})}(1+ \abs{\eps(v_j)}^2)^{\frac{n+1}{2}} \dx{x}  + \frac{1}{j^2} \frac{1}{r} \nonumber \\
& \leq c\Big(1+\frac{1}{r^{2}} + \frac{1}{j}\Big) \int_{\ball_{2r}(x_{0})}|\sg(v_{j})|\dif x   + \Big( \frac{1}{r^3} +\frac{1}{r^4} + \frac{1}{j} \Big) \mathscr{L}^{n}(\ball_{2r}(x_0)) +  \frac{c}{j^2} \Big( 1 + \frac{1}{r^2} \Big)
\end{align}
with a constant $c=c(n,\gamma,\Gamma,\lambda,\Lambda,M)>0$. 

	
	\emph{Step 3: $\LD_{\locc}$-regularity.} We now aim to employ Reshetnyak's (lower semi-) continuity theorem to conclude that $\E^{s}u\equiv 0$ in $\ball_{r}(x_{0})$. To this end, we first note that \eqref{eq:SASTconv}, \eqref{eq:LDdistanceMollification} and  $\eqref{eq:almostoptimal4Linfty}_{1}$ yield $v_{j}\to u$ in $\sobo^{-2,1}(\Omega;\R^{n})$. Furthermore, \eqref{v_j_estimate_1} and the compact embedding $\ld(\Omega)\hookrightarrow \lebe^{1}(\Omega;\R^{n})$ yield the existence of a (non-relabelled) subsequence of $(v_{j})_{j\in\N}$ which converges in the (symmetric) weak*-sense on $\bd(\Omega)$, to the same limit~$u$.  In particular, going back to~\eqref{eq:ultimate}, we find 
\begin{align}\label{eq:ultimate1}
\liminf_{j\to\infty}\int_{\ball_{r}(x_{0})}|\sg(v_{j}) |\log^{2}(1+|\sg(v_{j})|^{2})\dif x \leq c\Big(1+\frac{1}{r^{2}}\Big)\Big(r^{n-2}+|\E u|(\overline{\ball}_{2r}(x_{0}))\Big)
\end{align}
with a constant $c=c(n,\gamma,\Gamma,\lambda,\Lambda,M)$. With the notation $A_{\alpha}(t)\coloneqq t \log^{\alpha}(1+t^{2})$ introduced in Section~\ref{sec:Orlicz}, we next define a function $\Phi \colon \rsym \to [0,\infty)$ via $\Phi(z)\coloneqq A_{2}(|z|)$ for $z\in\rsym$. The corresponding recession function is given by 
	\begin{align}\label{eq:recPhi} 
	\Phi^\infty(z) 
	\coloneqq \lim_{t\to\infty}\frac{\Phi(tz)}{t} 
	= \lim_{t\to\infty} \abs{z}\log^2(1+t^{2}\abs{z}^{2}) 
	= \begin{cases}
	0\quad&\mbox{if}\quad \abs{z}=0\\
	+\infty \quad &\mbox{if}\quad \abs{z}>0.
	\end{cases}
	\end{align}
    With $v_{j}\stackrel{*}{\rightharpoonup}u$ in $\bd(\ball_{r}(x_{0}))$ the lower semicontinuity part of Theorem~\ref{lem:reshetnyak} (with the notation introduced in~\eqref{eq:funofmeas}) in combination with~\eqref{eq:ultimate1} implies
    \begin{align}\label{eq:LSCultimate1}
    \begin{split}
    \int_{\ball_{r}(x_{0})}\Phi&(\mathscr{E}u)\dif x + \int_{\ball_{r}(x_{0})}\Phi^{\infty}\Big(\frac{\dif\E u}{\dif|\E^{s}u|}\Big)\dif|\E^{s}u|  = \Phi(\E u)(\ball_{r}(x_{0})) \\ & \leq \liminf_{j\to\infty}\Phi(\E v_{j})(\ball_{r}(x_{0})) \leq c\Big(1+\frac{1}{r^{2}}\Big)\Big(r^{n-2}+|\E u|(\overline{\ball}_{2r}(x_{0}))\Big) 
    \end{split}
    \end{align}
with a constant $c=c(n,\gamma,\Gamma,\lambda,\Lambda,M)$. In view of~\eqref{eq:recPhi}, this implies $|\E^{s}u|\equiv 0$ in $\ball_{r}(x_{0})$. By arbitrariness of the ball $\ball_r(x_0)$ with $\ball_{2r}(x_0) \Subset \Omega$, this in turn yields that $\E^{s}u\equiv 0$ in $\Omega$. Hence, in particular, we arrive at $u\in\ld_{\locc}(\Omega)$ as well as $\mathscr{E}u=\sg(u)$, and estimate~\eqref{eq:LSCultimate1} then entails that $u\in \E^{1,A_{2}}(\ball_{r}(x_{0}))$ with the corresponding estimates. 

\emph{Step 4. $\sobo_{\locc}^{1,\lebe\log\lebe}$-regularity.} To conclude the proof, we now invoke the Korn-type inequality~\eqref{eq:kornllogl} from Lemma~\ref{lem:Korn} to find via~\eqref{eq:LSCultimate1} 
\begin{align*}
\int_{\ball_{r}(x_{0})}|\nabla u|\log(1+|\nabla u|^{2})\dif x & = \int_{\ball_{r}(x_{0})}A_{1}(|\nabla u|)\,\dif x \\
 & \leq c\bigg( r^n + \int_{\ball_{r}(x_{0})}A_{1}\Big(\frac{|u|}{r}\Big)\dif x + \int_{\ball_{r}(x_{0})}A_{2}(|\sg(u)|)\,\dif x \bigg) \\ 
& \leq c\Big(1+\frac{1}{r^{2}}\Big)\Big(r^{n} + r^{n-2} +|\E u|(\overline{\ball}_{2r}(x_{0}))\Big)
\end{align*}
for each ball $\ball_r(x_0)$ with $\ball_{2r}(x_0) \Subset \Omega$, with a constant $c=c(n,\gamma,\Gamma,\lambda,\Lambda,M)$. Since~$M$ depends only on $n$,~$\Omega$ and~$m$, this is the desired estimate~\eqref{eq:mainestimate} and the proof of the theorem is complete. 
\end{proof}
We conclude this subsection with two remarks on the above proof, both concerning potential improvements and its application to full gradient scenarios. 

\begin{remark}[$\sobo^{1,\lebe\log\lebe}$-regularity for $\mu = 3$]\label{rem:LlogL}
The above proof yields that every locally bounded $\bd$-minimizer~$u$ of~$F$ belongs to $\sobo{_{\locc}^{1,\lebe\log\lebe}}(\Omega;\R^{n})$ with $\sg(u)\in\lebe\log^{2}\lebe_{\locc}(\Omega;\rsym)$ provided that $\mu\leq 3$. It is not clear to us whether the above strategy can be improved for $\mu =3$ to obtain 
\begin{equation*}
 \sg(u)\in\lebe\log^{q}\lebe_{\locc}(\Omega;\rsym) \quad \text{and in turn} \quad u \in \sobo{_{\locc}^{1,\lebe\log^{q-1}\lebe}}(\Omega;\R^{n})
\end{equation*}
by Lemma~\ref{lem:Korn} (which is sharp) for some $q >2$. To arrive at this conclusion, one might be inclined to employ the test function $\varphi=\varrho^{4}\log^{q}(1+|\sg(v_{j})|^{2})v_{j}$ in the Euler--Lagrange inequality~\eqref{eq:ELpartint}. When estimating term~$\mathrm{III}$ by means of Young's inequality as above, we then have to control    
\begin{equation*}
\frac{\gamma}{2} \int_{\Omega}\varrho^{4}(1+|\sg(v_{j})|^{2})^{\frac{1}{2}}\log^{2(q-1)}(1+\abs{\eps(v_j)}^2)\dx{x}.
\end{equation*}
By the uniform $\lebe\log^{2}\lebe_{\locc}$-integrability of $\sg(v_{j})$ or by absorption, this is possible only for $q \leq 2$. In the full gradient case and subject to additional structure conditions on the integrands $f$, this can be overcome by use of stronger weights in Theorem~\ref{thm:uniformSecondOrder}  (see e.g. \cite[Lem. 4.2]{BeckSchmidt1}, \cite[Lem. 3.2]{BildhauerMu}). Here, however, the appearance of the symmetric gradients seems to destroy any benefits of such additional assumptions on $f$, whereby the local $\L\log^{2}\L$-integrability of $\eps(u)$ might be optimal; see Section~\ref{sec:improvedmu} below for improvements for $\mu$-elliptic integrands with $1<\mu<3$. 
\end{remark}
\begin{remark}[Admissibility of competitors]\label{rem:simple}
The integration by parts in Step 1 of the preceding proof is motivated by the fact that the Euler--Lagrange inequality from Lemma~\ref{lem:Euler_Lagrange_inequality} requires competitors $\varphi\in\sobo_{0}^{1,n+1}(\Omega;\R^{n})$. Aiming to test with $\varphi$ given by~\eqref{eq:log_estimate_test_function}, we however cannot argue by analogous means as in~\eqref{eq:makeitworkconvergences1} ff.. More precisely, we put $\mathcal{H}=\lebe_{\mu_{j}}^{2}(\ball_{2r}(x_{0});\rsym)$,  where $\mu_{j}=(1+|\sg(v_{j})|^{2})^{\frac{n-1}{2}}\mathscr{L}^{n}$ is as in Step~3 of the proof of Theorem~\ref{thm:uniformSecondOrder}. Considering approximations
	\begin{align}\label{eq:approxdiscrete}
	\varphi_h \coloneqq \varrho^{4}\log^{\gamma}({1+}\beta_{j}^{h})v_{j} \quad \text{with }\;\;\;{\beta_{j}^{h}} \coloneqq  \frac{1}{4}\sum_{i,m=1}^n \big|\Delta_{i,h}v_j^{(m)}+\Delta_{m,h}v_j^{(i)} \big|^2 \;\;\;\text{and}\;\;\;\gamma=2,
\end{align}
the desired Euler--Lagrange inequality satisfied by $\varphi$ then would follow from Lemma~\ref{lem:identification} and 
\begin{align*}
\Psi\in\mathcal{H}',\;\;\;\text{where}\;\;\;	\Psi\colon\mathcal{H}\ni\psi\mapsto\int_{\ball_{2r}(x_{0})}\langle \nabla f_{j}(\sg(v_{j})),\psi\rangle\,\dif x,
\end{align*}
provided that $(\varphi_{h})_{h>0}$ is bounded in $\mathcal{H}$. Expanding the symmetric gradients $\varepsilon(\varphi_{h})$ as 
\begin{align*}
\varrho^{4}\log^{2}(1+\beta_{j}^{h})\sg(v_{j}) + 2\varrho^{4}v_{j}\odot \frac{\log(1+\beta_{j}^{h})}{1+\beta_{j}^{h}}\nabla\beta_{j}^{h} + (\nabla\varrho^{4})\odot \log^{2}(1+\beta_{j}^{h})v_{j} \eqqcolon \mathrm{J}_{1}^{h}+\mathrm{J}_{2}^{h} + \mathrm{J}_{3}^{h}, 
\end{align*}
Lemma~\ref{lem:higherregularityapproximate} is too weak to conclude that $(\mathrm{J}_{2}^{h})_{h>0}$, and so $(\varepsilon(\varphi_{h}))_{h>0}$, is bounded in $\mathcal{H}$. This could be resolved by setting $\gamma=1$ in~\eqref{eq:approxdiscrete}, but then comes at the cost of the substantially weaker regularity conclusion $u\in\bv_{\locc}(\Omega;\R^{n})$ by the above proof of Theorem~\ref{thm:main}.  The integration by parts circumvents this issue, and also provides a simplification in the full gradient case, see \cite[Lem. 5.3]{BeckSchmidt},  where the $\lebe\log^{2}\lebe_{\locc}$-gradient integrability of local $\bv$-minimizers only follows by use of a two-step argument based on the analogue of~\eqref{eq:approxdiscrete} with $\gamma=1$. 
\end{remark}
\subsection{Improved $\sobo^{1,4-\mu}$-regularity for $1<\mu < 3$.}\label{sec:improvedmu}
As explained in the above  Remark~\ref{rem:LlogL} the Sobolev regularity stated in Theorem~\ref{thm:main} seems to be optimal in the borderline case $\mu=3$, but can still be improved under stronger ellipticity conditions with $\mu \in (1,3)$. In fact, a modification of the proof of Theorem~\ref{thm:main} yields the following result:
\begin{theorem}[$\sobo_{\locc}^{1,4-\mu}$-regularity, $1<\mu<3$]\label{thm:improreg}
Let $\Omega\subset\R^{n}$ be open and bounded, and let $f\in\hold^{2}(\rsym)$ be a variational integrand which satisfies~\eqref{eq:lingrowth1} and~\eqref{eq:MuEllipticity} with $1<\mu<3$. Then any local $\bd$-minimizer $u\in\bd_{\loc}(\Omega)\cap\lebe_{\loc}^{\infty}(\Omega;\R^{n})$ of $F$ is of class $\sobo_{\locc}^{1,4-\mu}(\Omega;\R^{n})$. 
\end{theorem}
\begin{proof}
We only comment on the required changes in the proof of Theorem~\ref{thm:main}. 
The suitable test function for the Euler-Lagrange inequality~\eqref{eq:ELpartint} is in this case given by $\varphi=\varrho^{4} (1 + |\sg(v_{j})|^{2})^{\frac{q}{2}} v_{j}$ with $q = \min\{1,3-\mu\}$. In Step~$2$, term~$\mathrm{I}$ then produces the desired expression 
\begin{equation*}
 \int_{\Omega}\varrho^{4}(1+|\sg(v_{j})|)^{\min\{2,4-\mu\}}\dx{x}
\end{equation*}
up to one term of at most linear growth, while terms~$\mathrm{II}$--$\mathrm{VI}$ are treated in analogy with the above proof (where the restriction~$3-\mu$ of the exponent in~$\varphi$  is needed for the estimation of term~$\mathrm{III}$ and the restriction~$1$ for the estimation of the terms~$\mathrm{IV}$ and~$\mathrm{VI}$ coming from the regularization). Passing to the limit $j \to \infty$ as above and employing the standard Korn inequality, we then arrive initially at an improved $\sobo^{1,\min\{2,4-\mu\}}_{\locc}$-regularity of the bounded local $\bd$-minimizer~$u$. For the case $\mu < 2$, this regularity result can further be improved via the following observations: using the weighted second order estimates for the sequence $(v_j)_{j \in \N}$, we infer a corresponding estimate for~$u$, which in particular implies local ${\sobo}^{2,4/(2+\mu)}$-regularity. At that stage, we may then use the standard Euler-Lagrange \emph{equation}
\begin{equation*}
 \int_\Omega \big\langle \nabla f(\eps(u)),\eps(\varphi)\big\rangle \dx{x} = 0 \qquad\text{for all } \varphi \in \sobo^{1,1}_c(\Omega;\R^n),
\end{equation*}
for the test function $\varphi=\varrho^{4} (T_K(1 + |\sg(u)|^{2}))^{\frac{3-\mu}{2}} u$ with truncation $T_K(t) = \min\{K,t\}$ for $t \geq 0$ at level~$K>0$ (needed for its admissibility). This allows to bound the expression
\begin{equation*}
 \int_{\Omega}\varrho^{4}\big(T_K(1+|\sg(u)|)\big)^{3-\mu} |\sg(u)| \dx{x}.
\end{equation*}
By monotone convergence and once again the standard Korn inequality  we then end up with the desired ${\sobo}{^{1,4-\mu}_{\locc}}$-regularity. 
\end{proof}

\begin{remark}[The borderline case $\mu=1$ and simplifications]\label{rem:LlogL1}
If one admits $\mu=1$ as the \emph{lower} growth exponent in~\eqref{eq:MuEllipticity}, one may integrate twice to see that integrands $f$ with this property are at least of $\lebe\log\lebe$-growth, but not of linear growth. In particular, this leads to the variant 
\begin{align}\label{eq:LlogLMuEllipticity}
\lambda \frac{|\xi|^{2}}{(1+|z|^{2})^{\frac{1}{2}}} \leq \langle \nabla^{2}f(z)\xi,\xi\rangle \leq \Lambda \frac{\mathrm{L}(|z|)}{(1+|z|^{2})^{\frac{1}{2}}}|\xi|^{2}\qquad\text{for all}\;z,\xi\in\rsym, 
\end{align}
where $\mathrm{L}\colon \R_{\geq 0}\to\R_{\geq 0}$ is monotonously increasing with 
\begin{align}\label{eq:LlogLgrowthmu=1}
\liminf_{t\to\infty}\frac{\mathrm{L}(t)}{t\log(1+t)} \in (0,\infty].
\end{align}
Growth conditions of this type have recently been considered by \textsc{De Filippis} et al. \cite{DeFilippisDeFilippisPiccinini1,DeFilippisMingione,DeFilippisPiccinini2}. The smallest choice of $\mathrm{L}$ with~\eqref{eq:LlogLgrowthmu=1} is $\mathrm{L}(t)=\psi(t)\coloneqq \ell t\log(1+t)$ with $\ell>0$. In the situation of variational problems~\eqref{eq:varprin1} over Dirichlet subclasses of Orlicz--Sobolev-type space
\begin{align*}
\mathrm{E}^{\psi}(\Omega)\coloneqq \{u\in\lebe\log\lebe(\Omega;\R^{n})\colon\; \sg(u)\in\lebe\log\lebe(\Omega;\rsym)\}, 
\end{align*}
the Korn-type inequality from Lemma~\ref{lem:Korn} on open and bounded sets $\Omega\subset\R^{n}$ with Lipschitz boundary yields that minimizers exist in $\sobo^{1,1}(\Omega;\R^{n})$. The slightly improved coercivity embodied by the additional $\log$-term in~\eqref{eq:LlogLMuEllipticity} then moreover implies that minimizers are unique. Problems of this sort have been studied in \cite{FrehseSeregin,FuchsSereginLog,Fuchs1999VariationalMF,FuchsSeregin} in the context of Prandtl-Eyring fluids and materials with logarithmic hardening. In the case where such materials are confined to a bounded spatial region, whereby one is in the situation of $\lebe^{\infty}$-constrained minimizers as considered in the present paper, the a priori $\sobo^{1,1}(\Omega;\R^{n})$-regularity of minimizers implies that neither the involved algebraic manipulations from Section~\ref{subsection:second_order_uniform} nor the use of Ekeland's variational principle are then necessary. Specifically, an adaptation of the arguments of the proof of Theorem \ref{thm:improreg} yields that locally bounded minimizers are of class  $\sobo_{\locc}^{1,3}(\Omega;\R^{n})$.
\end{remark}
\section{Appendix}\label{sec:appendix}

In this appendix we briefly address the extension results and the construction of auxiliary maps as used in the main part. To keep our exposition at a reasonable length, we focus on the key points throughout.

\subsection{Proof of Lemma~\ref{lem:ext}} \label{subsec:construction_extension}
To keep our presentation self-contained, we briefly address the linear extension operator from Lemma~\ref{lem:ext}. This operator is a variant of \textsc{Jones}' extension operator from~\cite{Jones} and has been employed first in~\cite{GmeinederRaita2} in the context of $\mathbb{C}$-elliptic differential operators, cf.~\cite{BDG}. We note that a (non-linear) operator with similar properties can be constructed by using that $\bd(\Omega)$ and $\sobo^{1,1}(\Omega)$ have the same trace spaces along $\partial\Omega$. However, we would then need to carefully discuss the construction of its right-inverse to obtain the corresponding weaker, yet sufficient version of Lemma~\ref{lem:ext}. 

We start by recording an estimate for projections: Given a non-degenerate cube $Q\subset\R^{n}$ and $v\in\lebe^{1}(Q;\R^{n})$, we denote by $\mathbb{P}_{Q}v$ the projection from Section~\ref{sec:bd} on~$Q$ onto $\mathscr{R}(\R^{n})$. Since $\mathscr{R}(\R^{n})$ is finite dimensional and hence the $\lebe^{1}$- and $\lebe^{\infty}$-norms on~$Q$ are equivalent, for every $v \in \lebe^{1}(Q;\R^{n})$ there holds 
\begin{align}\label{eq:inverse}
\|\mathbb{P}_{Q}v\|_{\lebe^{\infty}(Q;\R^{n})}\leq c \dashint_{Q}|\mathbb{P}_{Q}v|\,\dif x \leq c \dashint_{Q}|v|\,\dif x
\end{align}
for some constant $c>0$ depending only on~$n$.

Let $\Omega \subset \R^n$ be an open and bounded set with Lipschitz boundary. In order to construct the linear extension operator $\mathfrak{J}$, we pick a dyadic Whitney decomposition of~$\Omega$, i.e., a countable family~$\mathcal{W}_{1}$ of open, dyadic cubes~$Q$ (whose length is denoted by $\ell(Q)$) such that 
\begin{enumerate}[label={(W\arabic*)}]
\item\label{item:Whitney1} $\cup_{Q\in\mathcal{W}_{1}}\overline{Q}=\Omega$, and the cubes from $\mathcal{W}_{1}$ are pairwise disjoint, 
\item\label{item:Whitney2} $\sqrt{n}\ell(Q)\leq \dista(Q,\partial\Omega)\leq 4\sqrt{n}\ell(Q)$ for all $Q\in\mathcal{W}_{1}$,
\item\label{item:Whitney3} for $Q,Q'\in\mathcal{W}_{1}$ with $\overline{Q}\cap\overline{Q'}\neq\emptyset$ there holds 
\begin{align*}
\frac{1}{4} \leq \frac{\ell(Q)}{\ell(Q')}\leq 4,
\end{align*}
\item\label{item:Whitney4} and for every $Q \in \mathcal{W}_{1}$ there exist at most $12^{n}$ cubes $Q'\in\mathcal{W}_{1}$ with $Q \cap Q' = \emptyset$ and $\overline{Q}\cap\overline{Q'}\neq\emptyset$.
\end{enumerate}
Analogously, we choose a dyadic Whitney decomposition~$\mathcal{W}_{2}$ of $\R^{n}\setminus\overline{\Omega}$. Next, we denote by $\mathcal{W}_{3}$ the set of all cubes $Q\in\mathcal{W}_{2}$ with $\ell(Q)\leq \frac{3\diam(\Omega)}{16n}$. Most importantly, for each $Q\in\mathcal{W}_{3}$, there exists a \emph{reflected} cube $Q^*\in\mathcal{W}_{1}$ such that for some constant $c=c(\Omega)>0$ there hold
\begin{align}\label{eq:reflecto}
\frac{1}{c}\leq \frac{\ell(Q)}{\ell(Q^*)}\leq c \quad \text{and} \quad \dista(Q,Q^*)\leq c\ell(Q).
\end{align}

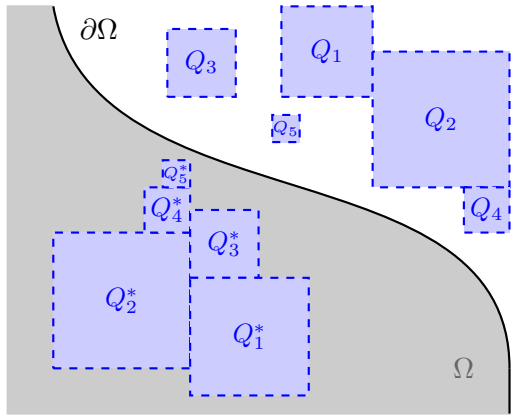
\begin{figure}[t]
\begin{center}
\begin{tikzpicture}[scale=1.2]
\draw[-,thick,black!20!white,fill=black!20!white,opacity=0.2] (-2,3) to [out = -80, in = 90] (3,-1) -- (3,-1.5) -- (-2.5,-1.5) -- (-2.5,3) -- (-2,3);
\draw[-,thick,black] (-2,3) to [out = -80, in = 90] (3,-1);
\draw[-,thick,black] (3,-1) -- (3,-1.5);
\draw[thick,blue,dashed,fill=blue!20!white,opacity=0.3] (-2,-1)--(-0.5,-1) -- (-0.5,0.5) -- (-2,0.5) -- (-2,-1);
\draw[thick,blue,dashed,fill=blue!20!white,opacity=0.3] (-0.5,0.5)--(-1,0.5) -- (-1,1) -- (-0.5,1) -- (-0.5,0.5);
\draw[thick,blue,dashed,fill=blue!20!white,opacity=0.3] (-0.5,1)--(-0.5,1.3) -- (-0.8,1.3) -- (-0.8,1) -- (-0.5,1);
\draw[thick,blue,dashed,fill=blue!20!white,opacity=0.3] (-0.5,0)--(0.25,0) -- (0.25,0.75) -- (-0.5,0.75)--(-0.5,0);
\draw[thick,blue,dashed,fill=blue!20!white,opacity=0.3] (-0.5,0)--(0.8,0) -- (0.8,-1.3) -- (-0.5,-1.3) -- (-0.5,0);
\node[black!60!white] at (2.5,-1) {{\large $\Omega$}};
\node[blue] at (0.15,-0.65) {$Q^*_{1}$};
\node[blue] at (-1.25,-0.25) {$Q^*_{2}$};
\node[blue] at (-0.125,0.375) { $Q^*_{3}$};
\node[blue] at (-0.75,0.75) {$Q^*_{4}$};
\node[blue] at (-0.65,1.15) {\tiny $Q^*_{5}$};

\node[blue] at (1,2.5) {$Q_{1}$};
\node[blue] at (2.25,1.75) {$Q_{2}$};
\node[blue] at (-0.375,2.375) { $Q_{3}$};
\node[blue] at (2.75,0.75) {$Q_{4}$};
\node[blue] at (0.55,1.65) {\tiny $Q_{5}$};
\draw[thick,blue,dashed,fill=blue!20!white,opacity=0.3] (-0.75,2)--(0,2) -- (0,2.75) -- (-0.75,2.75)--(-0.75,2.75) -- (-0.75,2);
\draw[thick,blue,dashed,fill=blue!20!white,opacity=0.3] (0.5,2)--(1.5,2) -- (1.5,3) -- (0.5,3) -- (0.5,2);
\draw[thick,blue,dashed,fill=blue!20!white,opacity=0.3] (1.5,2.5)--(1.5,1) -- (3,1) -- (3,2.5) -- (1.5,2.5);
\draw[thick,blue,dashed,fill=blue!20!white,opacity=0.3] (2.5,1) -- (3,1) -- (3,0.5) -- (2.5,0.5) -- (2.5,1);
\draw[thick,blue,dashed,fill=blue!20!white,opacity=0.3] (0.4,1.5)--(0.7,1.5) -- (0.7,1.8) -- (0.7,1.8) -- (0.4,1.8) -- (0.4,1.5);
\node[black] at (-1.5,2.75) {{\large $\partial\Omega$}};
\node[blue] at (1,2.5) {$Q_{1}$};
\node[blue] at (2.25,1.75) {$Q_{2}$};
\node[blue] at (-0.375,2.375) { $Q_{3}$};
\node[blue] at (2.75,0.75) {$Q_{4}$};
\node[blue] at (0.55,1.65) {\tiny $Q_{5}$};
\end{tikzpicture}
\end{center}
\caption{Based on a one-to-one correspondence between the cubes of $\mathcal{W}_{1}$ in $\Omega$ and the cubes of $\mathcal{W}_{3}$ in $\R^n \setminus \overline{\Omega}$, the requisite extension operator $\mathfrak{J}$ from Lemma~\ref{lem:ext} is constructed via a reflection principle indicated above in the picture.}
\end{figure}

All these properties can be traced back to~\cite{Jones}, also see \cite[Prop.~8.5.3, Lem.~8.5.4]{DHHR} for a modern treatment. We may then choose $\theta>1$ sufficiently close to $1$ such that the family $(\theta Q)_{Q \in \mathcal{W}_{3}}$ (with $\theta Q$ having the same center as $Q$ and $\ell(\theta Q) = \theta \ell (Q)$) satisfies $\theta Q\subset\R^{n}\setminus\overline{\Omega}$ for all $Q\in\mathcal{W}_{3}$ and yields a locally uniformly finite cover of $\cup_{Q\in\mathcal{W}_{3}}Q$ with \ref{item:Whitney2}--\ref{item:Whitney4} and~\eqref{eq:reflecto} still being in action, possibly with worse constants. We take a partition of unity~$(\varphi_{\theta Q})$ in $\hold_{c}^{\infty}(\R^{n}\setminus\overline{\Omega};[0,1])$  subordinate to the covering $(\theta Q)_{Q\in\mathcal{W}_{3}}$. For  $u\in\bd(\Omega)$ we then set 
\begin{align}\label{eq:Jonesdef}
\overline{\mathfrak{J}}u \coloneqq  \begin{cases} u&\;\text{in}\;\Omega,\\
\sum_{Q\in\mathcal{W}_{3}}\varphi_{\theta Q}\mathbb{P}_{Q^*}u&\;\text{in}\;\R^{n}\setminus\overline{\Omega}. 
\end{cases}
\end{align}
Under mild (and natural) growth assumptions on the derivatives of the functions in $(\varphi_{\theta Q})$ it is established in \cite[Sec.~4.1]{GmeinederRaita2} that $\overline{\mathfrak{J}}\colon\ld(\Omega)\to\ld(\R^{n})$ is a bounded linear operator, and the $\bd$-case follows by analogous arguments. It remains to give the argument for the $\lebe^{\infty}$-stability asserted in Lemma~\ref{lem:ext}\ref{item:ext2}. To this end, we fix a cube $Q\in\mathcal{W}_{3}$ and obtain for $x\in\R^{n}\setminus\overline{\Omega}$:
\begin{align*}
|(\varphi_{\theta Q}\mathbb{P}_{Q^*}u)(x)| & \leq \|\mathbb{P}_{Q^*}u\|_{\lebe^{\infty}(\theta Q;\R^{n})} \leq  c\|\mathbb{P}_{Q^*}u\|_{\lebe^{\infty}(Q;\R^{n})} \leq c\|\mathbb{P}_{Q^*}u\|_{\lebe^{\infty}(Q^*;\R^{n})} 
\end{align*}
for a constant $c=c(n,\Omega,\theta)>0$. Here, the second inequality is a consequence of the equivalence of all norms on a finite dimensional space with scaling, whereas the third inequality again follows from this equivalence together with~\eqref{eq:reflecto}. In combination with~\eqref{eq:inverse}, we then arrive at
\begin{align*}
|(\varphi_{\theta Q}\mathbb{P}_{Q^*}u)(x)| \leq \|\mathbb{P}_{Q^*}u\|_{\lebe^{\infty}(Q^*;\R^{n})} \leq c \dashint_{Q^*}|u|\,\dif x \leq c\|u\|_{\lebe^{\infty}(\Omega;\R^{n})}, 
\end{align*}
for a constant $c=c(n,\Omega,\theta)>0$. Since the family $(\theta Q)_{Q \in \mathcal{W}_{3}}$ of blown-up Whitney cubes still satisfies~\ref{item:Whitney4} with worse constants, the number of overlapping cubes $\theta Q$ (with $Q\in\mathcal{W}_{3}$) is uniformly bounded by a constant $c=c(n,\theta)>0$. Consequently, $\overline{\mathfrak{J}}\colon\lebe^{\infty}(\Omega;\R^{n})\to\lebe^{\infty}(\R^{n};\R^{n})$ is a bounded linear operator too. 

Finally, if $\Omega_0\subset\R^{n}$ is an open and bounded set with $\Omega\Subset\Omega_0$, we choose a cut-off function $\eta\in\hold_{c}^{\infty}(\Omega_0;[0,1])$ with $\mathbbm{1}_{\Omega}\leq \eta\leq\mathbbm{1}_{\Omega_0}$ and define $\mathfrak{J}u\coloneqq \eta\overline{\mathfrak{J}}u$. It is then straightforward to see that the operator $\mathfrak{J} \colon \bd(\Omega)\to\bd(\R^{n})$ has all the properties claimed in Lemma~\ref{lem:ext}. \qed

\subsection{Construction of $\Upsilon$ and $h$}\label{sec:Ups}
We briefly discuss the elementary construction of the auxiliary functions~$\Upsilon$ and~$h$ which entered the proof of existence of the Ekeland-type approximation sequence in Section~\ref{section:viscosity_approximation}, cf.~\eqref{eq:introUpsilon} and~\eqref{eq:subtlehconstruction}.

  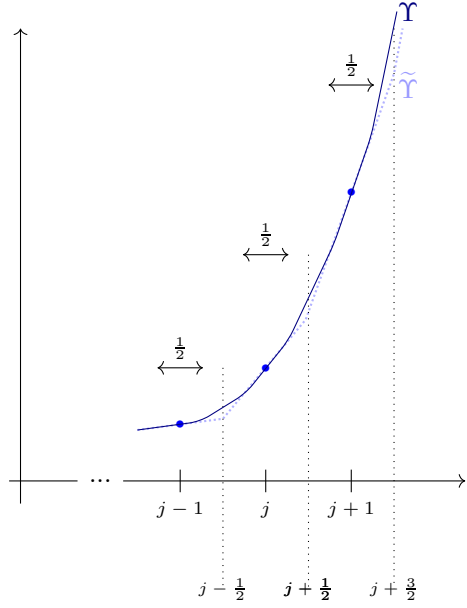
\begin{figure}[t]
	
	\begin{tikzpicture}[scale=1.5]
		\draw[-] (0.5, 0) -- (1.1,0);
		\draw[->] (0.6,-0.2) -- (0.6,4); 
		\node at (1.3,0) {...};
		\draw[->] (1.5, 0) -- (4.5,0);
		\node[below] at (2,-0.1) {\SMALL $j-1$};
		\draw[-] (2,-0.1) -- (2,0.1);
		\draw[-] (2.75,-0.1) -- (2.75,0.1);
		\node[below] at (2.75,-0.1) {\SMALL $j$};
		\draw[-] (3.5,-0.1) -- (3.5,0.1);
		\node[below] at (3.5,-0.1) {\SMALL $j+1$};
		\draw[-,blue!30!white,densely dotted,thick] (1.625,0.45) -- (2.375,0.55) -- (3.1,1.425) -- (3.875,3.6) -- (3.95,4);
		\draw[dotted] (3.125,2) -- (3.125,-0.95);
		\node[below] at (3.125,-0.775) {\tiny $j+\frac{1}{2}$};
		\draw[dotted] (2.375,1) -- (2.375,-0.95);
		\node[below] at (3.125,-0.775) {\tiny $j-\frac{1}{2}$};
		\draw[dotted] (3.875,4) -- (3.875,-0.95);
		\node[below] at (3.875,-0.775) {\tiny $j+\frac{3}{2}$};
		\node[below] at (2.375,-0.775) {\tiny $j-\frac{1}{2}$};
		\node[blue] at (2,0.5) {\tiny \textbullet};
		\node[blue] at (2.75,1) {\tiny \textbullet};
		\node[blue] at (3.5,2.55) {\tiny \textbullet};
		\draw[-,blue!50!black, rounded corners] (1.625,0.45) -- (2.175,0.525) -- (2.575,0.775) -- (2.95,1.245) -- (3.35,2.1) -- (3.675, 3.05) -- (3.9,4.15) ;
		\draw[<->] (1.8,1) -- (2.2,1);
		\node[above] at (2,1) {\tiny$\frac{1}{2}$};
		\draw[<->] (2.55,2) -- (2.95,2);
		\node[above] at (2.75,2) {\tiny$\frac{1}{2}$};
		\draw[<->] (3.3,3.5) -- (3.7,3.5);
		\node[above] at (3.5,3.5) {\tiny$\frac{1}{2}$};   
		\node[blue!50!black] at (4,4.15) {$\Upsilon$};
		\node[blue!40!white] at (4,3.5) {$\widetilde{\Upsilon}$};
	\end{tikzpicture}
	\caption{Construction of $\Upsilon$. As indicated in the figure, one first passes to the affine-linear interpolation of $\widetilde{\Upsilon}$ at the points $j+\frac{1}{2}$, $j\in\mathbb{N}$ (light blue dotted line). Mollification at radius $\varepsilon=\frac{1}{4}$ then leaves  $\widetilde{\Upsilon}$ unchanged in the $\frac{1}{2}$-neighbourhoods of all points $j\in\mathbb{N}$, and yields the desired function $\Upsilon$ (dark blue line).}\label{fig:mollification1}
\end{figure}
Concerning the function~$\Upsilon$, we set $\widetilde{\Upsilon}(0) = 0$ and 
\begin{equation*}
\widetilde{\Upsilon}(j) \coloneqq \sum_{\ell = 1}^j (j+1-\ell) \|\eps(\widetilde{u}_{j})\|_{\L^{n+1}(\Omega;\rsym)} \quad \text{for } j \in \N.
\end{equation*}
By construction, $\widetilde{\Upsilon}$ is an increasing function on~$\N_0$, and the differences 
\begin{equation*}
 \widetilde{\Upsilon}(j) - \widetilde{\Upsilon}(j-1) = \sum_{\ell = 1}^j \|\eps(\widetilde{u}_{j})\|_{\L^{n+1}(\Omega;\rsym)}\quad \text{for } j \in \N 
\end{equation*}
are increasing as well. By use of an affine-linear interpolation of $\widetilde{\Upsilon}$ as indicated in Figure \ref{fig:mollification1}, we may then extend $\widetilde{\Upsilon}$ to the entire $\R_{\geq 0}$. Setting  $\Upsilon:=\rho*\widetilde{\Upsilon}$ with a standard mollifier of radius $\frac{1}{4}$ leaves $\widetilde{\Upsilon}$ unchanged on intervals of length $\frac{1}{2}$ around each $j\in\mathbb{N}$. Moreover, $\Upsilon$ is convex, increasing and of class $\hold^{\infty}$ on $\R_{\geq 0}$ with $\Upsilon(j)\geq \|\sg(\widetilde{u}_{j})\|_{\lebe^{n+1}(\Omega;\rsym)}$ for each $j\in\mathbb{N}$.

Concerning the function~$h$, we start by writing the left-hand side of~\eqref{eq:subtlehconstruction} as $\Phi(t)$ for $t \in [\tfrac{3}{2},2)$. Based on the observation that the resulting function $\Phi\colon[\tfrac{3}{2},2)\to(0,\infty)$ is given in terms of concatenations and products of non-negative, non-decreasing and convex $\hold^2$-functions, it is not difficult to see that~$\Phi$ itself is a non-negative, (even strictly) increasing and convex $\hold^2$-function. We then define the desired function~$h$ on $[0,1]$ as $h \equiv 0$ and on~$[\tfrac{3}{2},2)$ as 
\begin{equation*}
 h(t) \coloneqq \beta \Phi(t) +(1-\beta)\Phi(\tfrac{3}{2}) \geq \Phi(t) \quad \text{for } t \in [\tfrac{3}{2},2)
\end{equation*}
for some constant $\beta > 1$ with $\beta\Phi'(\frac{3}{2})>2\Phi(\frac{3}{2})$. Based on this choice, we then may perform a similar extension as above to obtain an increasing, convex $\hold^2$-function on $[0,2)$ with the requisite properties.

\subsection*{Acknowledgments}
{\small F.G. acknowledges financial support through the Hector foundation (Project Number FP
626/21).}


\bibliographystyle{alpha}		
\bibliography{symgradnew}

\begin{thebibliography}{DHHR11}

\bibitem[ACM97]{ACD}
Luigi Ambrosio, Alessandra Coscia, and Gianni~Dal Maso.
\newblock {F}ine properties of functions with bounded deformation.
\newblock {\em Arch. Rational Mech. Anal.}, 139(3):201--238, 1997.

\bibitem[BBG20]{BeckBulicekGmeineder}
Lisa Beck, Miroslav Bul\'{i}\v{c}ek, and Franz Gmeineder.
\newblock On a {N}eumann problem for variational functionals of linear growth.
\newblock {\em Annali della Scuola Normale Superiore di Pisa -- Classe di
  Scienze}, XXI:695--737, 2020.

\bibitem[BBMS17]{BBMS}
Lisa Beck, Miroslav Bul\'{i}\v{c}ek, Josef M\'{a}lek, and Endre S\"{u}li.
\newblock On the existence of integrable solutions to nonlinear elliptic
  systems and variational problems with linear growth.
\newblock {\em Arch. Ration. Mech. Anal.}, 225:717--769, 2017.

\bibitem[BDG20]{BDG}
Dominic Breit, Lars Diening, and Franz Gmeineder.
\newblock On the {T}race {O}perator for {F}unctions of bounded
  $\mathbb{A}$-{V}ariation.
\newblock {\em Anal. PDE}, (13):559--594, 2020.

\bibitem[BEG]{BeckEitlerGmeineder}
Lisa Beck, Ferdinand Eitler, and Franz Gmeineder.
\newblock Gradient regularity and lower order higher integrablity for
  functionals of $(p,q)$-growth.
\newblock {\em In preparation}.

\bibitem[Bil02]{BildhauerMu}
Michael Bildhauer.
\newblock A priori gradient estimates for bounded generalised solutions of a
  class of variational problems with linear growth.
\newblock {\em J. Convex Ana.}, 9, 2002.

\bibitem[Bil03]{BildhauerLecNotes}
Michael Bildhauer.
\newblock Convex variational problems, linear, nearly linear and anisotropic
  growth conditions.
\newblock {\em Lect. Notes Math. 1818}, 2003.

\bibitem[Bre11]{Brezis}
Haim Brezis.
\newblock {\em Functional {A}nalysis, {S}obolev {S}paces and {P}artial
  {D}ifferential {E}quations}.
\newblock Universitext. Springer New York, NY, 1 edition, 2011.

\bibitem[BS11]{BeckSchmidt}
Lisa Beck and Thomas Schmidt.
\newblock On the {D}irichlet {P}roblem for variational integrals in {BV}.
\newblock {\em Journal für die reine und angewandte Mathematik, Volume 674},
  pages 113--194, 2011.

\bibitem[BS15]{BeckSchmidt1}
Lisa Beck and Thomas Schmidt.
\newblock Interior gradient regularity for {BV} minimizers of singular
  variational problems.
\newblock {\em {N}onlinear {A}nalysis: {T}heory, {M}ethods \& {A}pplications},
  120:86--106, 2015.

\bibitem[Cam82]{CAMPANATO82a}
Sergio Campanato.
\newblock {Differentiability of the solutions of nonlinear elliptic systems
  with natural growth}.
\newblock {\em Ann. Mat. Pura Appl. Ser. 4}, {\bf 131}:75--106, 1982.

\bibitem[CFM05]{ContiFaracoMaggi}
Sergio Conti, Daniel Faraco, and Francesco Maggi.
\newblock {A new approach to counterexamples to $\mathrm{L}^{1}$–estimates:
  Korn's inequality, geometric rigidity, and regularity for gradients of
  separately convex functions}.
\newblock {\em Arch. Rat. Mech. Anal.}, {\bf 175}:287–300, 2005.

\bibitem[Cia14]{Cianchi}
Andrea Cianchi.
\newblock Korn type inequalities in {O}rlicz spaces.
\newblock {\em Journal of Functional Analysis 267}, pages 2313--2352, 2014.

\bibitem[CKdN11]{CKP}
Menita Carozza, Jan Kristensen, and Antonella~Passarelli di~Napoli.
\newblock Higher differentiability of minimizers of convex variational
  integrals.
\newblock {\em Annales de l'Institut Henri Poincare (C) Analyse Non Lineaire},
  \textbf{28}:395--411, 2011.

\bibitem[DHHR11]{DHHR}
Lars Diening, Peter H\"{a}st\"{o}, Petteri Harjulehto, and Michael Ruzicka.
\newblock {\em {L}ebesgue and {S}obolev spaces with variable exponents}, volume
  2017 of {\em Springer Lecture Notes}.
\newblock Springer {V}erlag {B}erlin, 2011.

\bibitem[Eke74]{Ekeland}
Ivar Ekeland.
\newblock On the variational principle.
\newblock {\em J. Math. Anal. Appl.}, 47:324--353, 1974.

\bibitem[ELM99a]{EspositoLeonettiMingione1999A}
Luca Esposito, Francesco Leonetii, and Giuseppe Mingione.
\newblock {R}egularity for minimizers of functionals with $p$-$q$-growth.
\newblock {\em NoDEA, Nonlinear differ. equ. appl.}, (6):133--148, 1999.

\bibitem[ELM99b]{EspLeoMin}
Luca Esposito, Francesco Leonetti, and Giuseppe Mingione.
\newblock Higher integrability for minimizers of integral functionals with $(p,
  q)$-growth.
\newblock {\em J. Differ. Equations}, 4157:414–438, 1999.

\bibitem[FFP24]{DeFilippisDeFilippisPiccinini1}
Cristiana~De Filippis, Filomena~De Filippis, and Marco Piccinini.
\newblock {B}ounded minimizers of double phase problems at nearly linear
  growth.
\newblock {\em ArXiv preprint}, 2024.
\newblock \url{https://arxiv.org/abs/2411.14325}.

\bibitem[FM23]{DeFilippisMingione}
Cristiana~De Filippis and Giuseppe Mingione.
\newblock {R}egularity for {D}ouble {P}hase {P}roblems at {N}early {L}inear
  {G}rowth.
\newblock {\em Arch. Ration. Mech. Anal.}, (247:85), 2023.

\bibitem[FP24]{DeFilippisPiccinini2}
Filomena~De Filippis and Marco Piccinini.
\newblock {R}egularity for multi-phase problems at nearly linear growth.
\newblock {\em ArXiv preprint}, 2024.
\newblock \url{https://arxiv.org/abs/2401.02186}.

\bibitem[FS98a]{FrehseSeregin}
Martin Fuchs and G.~Seregin.
\newblock {Regularity of solutions to variational problems of the deformation
  theory of plasticity with logarithmic hardening}.
\newblock {\em Proc. St. Petersburg Math. Soc.}, 5:184--222, 1998.
\newblock English translation: Amer. Math. Soc. Transl. Ser. 2 193 (1999)
  127--152.

\bibitem[FS98b]{FuchsSereginLog}
Martin Fuchs and Gregory Seregin.
\newblock A regularity theory for variational integrals with llnl-growth.
\newblock {\em Cal. Var. PDE}, 6(2):171--187, 1998.

\bibitem[FS99]{Fuchs1999VariationalMF}
Martin Fuchs and G.~Seregin.
\newblock {Variational methods for fluids of Prandtl–Eyring type and plastic
  materials with logarithmic hardening}.
\newblock {\em Mathematical Methods in The Applied Sciences}, 22:317--351,
  1999.

\bibitem[FS00]{FuchsSeregin}
Martin Fuchs and Gregory Seregin.
\newblock {\em Variational methods for problems from plasticity theory and for
  generalized {N}ewtonian fluids}, volume 1749.
\newblock Springer-Verlag, Berlin, 2000.

\bibitem[Giu03]{Giusti}
Enrico Giusti.
\newblock {\em Direct {M}ethods in the {C}alculus of {V}ariations}.
\newblock {W}orld {S}cientific {P}ublishing {C}o., {I}nc., {R}iver {E}dge,
  {NJ}, 2003.

\bibitem[GK19]{Gmeineder1}
Franz Gmeineder and Jan Kristensen.
\newblock {S}obolev regularity for convex functionals on {BD}.
\newblock {\em J. Calc. Var.}, (58:56), 2019.

\bibitem[Gme16]{Gmeineder2016}
Franz Gmeineder.
\newblock {S}ymmetric-{C}onvex {F}unctionals of {L}inear {G}rowth.
\newblock {\em J. Ell. and Par. Equations}, 1-2(2):59--71, 2016.

\bibitem[Gme20]{Gmeineder}
Franz Gmeineder.
\newblock The {R}egularity of {M}inima for the {D}irichlet problem on {BD}.
\newblock {\em Arch. Rational Mech. Anal. 237}, pages 1099--1171, 2020.

\bibitem[GR19a]{GmeinederRaita2}
Franz Gmeineder and Bogdan Raita.
\newblock {E}mbeddings for $\mathbb{A}$-weakly differentiable functions on
  domains.
\newblock {\em J. Funct. Anal. 277(1)}, (108278), 2019.

\bibitem[GR19b]{GmeinederRaita1}
Franz Gmeineder and Bogdan Raita.
\newblock On critical $\mathrm{L}^{p}$-differentiability of {BD}-maps.
\newblock {\em Rev. Mat. Iberoam.}, 35(7):2071--2078, 2019.

\bibitem[GT77]{GILTRU77}
David Gilbarg and Neil~S. Trudinger.
\newblock {\em {Elliptic Partial Differential Equations of Second Order}}.
\newblock Springer-Verlag, Berlin-Heidelberg-New York, 1977.

\bibitem[Jon81]{Jones}
Peter~W. Jones.
\newblock {Q}uasiconformal mappings and extendability of functions in {S}obolev
  spaces.
\newblock {\em Acta Mathematica}, \textbf{147}(1):71--88, 1981.

\bibitem[KK16]{KirchheimKristensen}
Bernd Kirchheim and Jan Kristensen.
\newblock {On rank one convex functions that are homogeneous of degree one}.
\newblock {\em Arch. Rat. Mech. Anal.}, {\bf 221}(1):527–558, 2016.

\bibitem[Koh82]{Kohn1}
Robert~V. Kohn.
\newblock {N}ew integral estimates for deformations in terms of their nonlinear
  strains.
\newblock {\em Arch. Ration. Mech. Anal.}, 78(2):131--172, 1982.

\bibitem[Lub90]{Lubliner}
Jacob Lubliner.
\newblock {\em {P}lasticity {T}heory}.
\newblock Macmillan, New York, London, 1990.

\bibitem[Mar89]{Marcellini1}
Paolo Marcellini.
\newblock Regularity of minimizers of integrals of the calculus of variations
  with non standard growth conditions.
\newblock {\em Arch. Ration. Mech. Anal.}, 105:267--284, 1989.

\bibitem[Mar91]{Marcellini2}
Paolo Marcellini.
\newblock Regularity and existence of solutions of elliptic equations with
  $p,q$-growth conditions.
\newblock {\em J. Differ. Equations}, 90:1--30, 1991.

\bibitem[Min03a]{MingioneBound}
Giuseppe Mingione.
\newblock Bounds for the singular set of solutions to non linear elliptic
  systems.
\newblock {\em Calc. Var. Partial Differential Equations}, 18(4):373--400,
  2003.

\bibitem[Min03b]{MingioneFract}
Giuseppe Mingione.
\newblock The singular set of solutions to non-differentiable elliptic systems.
\newblock {\em Arch. Ration. Mech. Anal.}, 166(4):287--301, 2003.

\bibitem[Min06]{MinDarkSide}
Giuseppe Mingione.
\newblock Regularity of minima: an invitation to the dark side of the calculus
  of variations.
\newblock {\em Appl. Math.}, 51(4):355--426, 2006.

\bibitem[Orn62]{Ornstein}
Donald Ornstein.
\newblock A non-equality for differential operators in the $\lebe_{1}$–norm.
\newblock {\em Arch. Rational Mech. Anal.}, pages 40--49, 1962.

\bibitem[Res68]{Reshetynak}
Yuri~G. Reshetnyak.
\newblock Weak convergence of completely additive vector functions on a set.
\newblock {\em Sib. Math. J. 9}, pages 1039--1045, 1968.

\bibitem[Res70]{Reshetnyak1}
Yuri~G. Reshetnyak.
\newblock Estimates for certain differential operators with finite-dimensional
  kernel ({R}ussian).
\newblock {\em Sibirsk. Mat. Z}, 11:414–428, 1970.

\bibitem[Sch15]{SchmidtHabil}
Thomas Schmidt.
\newblock {\em {BV} {M}inimizers of {V}ariational {I}ntegrals: {E}xistence,
  {U}niqueness, {R}egularity. \emph{Habilitationsschrift,
  Friedrich-Alexander-Universit\"{a}t Erlangen-N\"{u}rnberg}}.
\newblock 2015.

\bibitem[TS81]{TemamStrang}
Roger Temam and Gilbert Strang.
\newblock Functions of bounded deformation.
\newblock {\em Arch. Ration. Mech. Anal.}, (75):7--21, 1981.

\end{thebibliography}

\end{document}